\documentclass[reqno,11pt]{amsart}
\usepackage{amssymb}
\usepackage{amsmath}
\usepackage{stmaryrd}
\usepackage{graphicx}
\usepackage{color}
\usepackage{cite}
\usepackage[font=small]{caption}
\usepackage{tikz}
\usepackage{pgfplots}
\pgfplotsset{compat=1.10}
\usepgfplotslibrary{fillbetween}
\usetikzlibrary{patterns}
\newtheorem{theorem}{Theorem}[section]
\newtheorem{corollary}[theorem]{Corollary}
\newtheorem{lemma}[theorem]{Lemma}
\newtheorem{proposition}[theorem]{Proposition}

\newtheorem{remark}[theorem]{Remark}
\newtheorem{example}[theorem]{Example}
\newcommand{\R}{\mathbb{R}}
\newcommand{\rd}{\mathrm{d}}

\definecolor{cadmiumgreen}{rgb}{0.0, 0.42, 0.24}
\numberwithin{equation}{section}
\numberwithin{figure}{section}
\begin{document}

\title[Shape derivative of the Dirichlet energy for a transmission problem]{Shape derivative of the Dirichlet energy for\\ a transmission problem}

\author{Philippe Lauren\c{c}ot}
\address{Institut de Math\'ematiques de Toulouse, UMR~5219, Universit\'e de Toulouse, CNRS \\ F--31062 Toulouse Cedex 9, France}
\email{laurenco@math.univ-toulouse.fr}
\author{Christoph Walker}
\address{Leibniz Universit\"at Hannover\\ Institut f\" ur Angewandte Mathematik \\ Welfengarten 1 \\ D--30167 Hannover\\ Germany}
\email{walker@ifam.uni-hannover.de}
%
\thanks{Partially supported by the CNRS Projet International de Coop\'eration Scientifique PICS07710}
\date{\today}
\keywords{Shape derivative, free boundary problem, transmission problem, obstacle problem, fourth-order equation}
\subjclass[2010]{49Q10 - 35J20 - 74G65 - 35R35 - 35Q74}
%
%
\begin{abstract}
For a transmission problem in a truncated two-dimensional cylinder located beneath the graph of a function $u$, the shape derivative of the Dirichlet energy (with respect to $u$) is shown to be well-defined and is computed in terms of $u$. The main difficulties in this context arise from the weak regularity of the domain and the possibly non-empty intersection of the graph of $u$ and the transmission interface. The explicit formula for the shape derivative is then used to identify the partial differential equation solved by the minimizers of an energy functional arising in the modeling of micromechanical systems.
\end{abstract}
%
\maketitle
%
\section{Introduction and Main Results}\label{IMR}

Given  $f\in H^{-1}(\R^n)$ and an open, bounded set $\mathcal{O}\subset\R^n$, let $\varphi_\mathcal{O}\in H_0^1(\mathcal{O})$ be the unique variational solution to the Dirichlet problem
\begin{equation}
-\Delta\varphi_\mathcal{O}=f\quad\text{in }\ \mathcal{O}\,,\qquad \varphi_\mathcal{O}=0\quad \text{on }\ \partial\mathcal{O}\,. \label{Dp}
\end{equation}
Introducing the Dirichlet integral
\begin{equation*}
J(\mathcal{O}):=\frac{1}{2}\int_\mathcal{O} \vert\nabla\varphi_\mathcal{O}\vert^2\,\rd x\,,
\end{equation*}
a classical result in shape optimization states that the shape derivative of $J(\mathcal{O})$ is given by
\begin{equation*}
J'(\mathcal{O})[\theta]:=\frac{\rd}{\rd t} J\big((\mathrm{id}+t\theta)(\mathcal{O})\big)\big\vert_{t=0} = \frac{1}{2} \int_\mathcal{O} \mathrm{div}\big(\vert\nabla\varphi_\mathcal{O}\vert^2 \theta\big)\,\rd x
\end{equation*}
for $\theta\in W_\infty^1(\R^n,\R^n)$ \cite{HP05,SZ92}. When the shape derivative is well-defined, it provides useful information on the Dirichlet energy itself and it is the basis for deriving first-order optimality conditions. However, the integral on the right-hand side of the shape derivative is only meaningful provided $\varphi_\mathcal{O}$ has sufficient regularity (typically $\varphi_\mathcal{O}\in H^2(\mathcal{O})$), which, in turn, requires sufficient regularity of the source term $f$ and the open set $\mathcal{O}$, see \cite[Corollary~5.3.8]{HP05} for instance. Source terms with low Sobolev regularity or depending on the admissible shape $\mathcal{O}$ in a non-smooth way are therefore excluded.

Amongst the simplest situations featuring such a dependence is the differentiability with respect to $\mathcal{O}$ of the Dirichlet energy
\begin{equation}
\mathcal{J}(\mathcal{O}) := \frac{1}{2} \int_\mathcal{O} |\nabla\psi_\mathcal{O}|^2\,\rd x\,, \label{y0}
\end{equation}
associated with Laplace's equation subject to non-homogeneous Dirichlet boundary conditions
\begin{equation*}
-\Delta\psi_\mathcal{O}= 0\quad\text{in }\ \mathcal{O}\,,\qquad \psi_\mathcal{O}= h_\mathcal{O} \quad \text{on }\ \partial\mathcal{O}\,, \end{equation*}
where $h_\mathcal{O}$ is a given function in $H^1(\mathcal{O})$, depending on $\mathcal{O}$ in general. In that particular case, $\psi_\mathcal{O}$ may be interpreted as the electrostatic potential inside $\mathcal{O}$ and $h_\mathcal{O}$ is the potential applied on $\partial\mathcal{O}$. Computing the shape derivative $\mathcal{J}'(\mathcal{O})$ of $\mathcal{J}(\mathcal{O})$ is then of practical importance, since $\mathcal{J}'(\mathcal{O})$ is the electrostatic force acting on $\partial\mathcal{O}$ \cite{FLM12,CDLM,CLWZ13}. Introducing $\varphi_\mathcal{O} := \psi_\mathcal{O} - h_\mathcal{O}$, we see that $\varphi_\mathcal{O}$ solves \eqref{Dp} with $f=\Delta h_\mathcal{O}\in H^{-1}(\mathcal{O})$ and $\mathcal{J}'(\mathcal{O})$ obviously involves the shape derivative of $h_\mathcal{O}$. Summarizing, the shape differentiability of $\mathcal{J}(\mathcal{O})$ relies on the Sobolev regularity of $\varphi_\mathcal{O}$ which is not only governed by that of $h_\mathcal{O}$ but also by the smoothness of $\partial\mathcal{O}$.

The situation just depicted above is actually met in applications, for example, when considering electrostatic actuators consisting of a rigid electrode above which a moving electrode is suspended, both being held at different potentials \cite{BG01}. For an idealized device with simplified geometry, the rigid electrode is the set $D\times\{-H\}$ located at vertical height $-H<0$ with $D:=(-L,L)$, $L>0$, and the shape depends only on the position of the moving electrode, which is assumed to be the graph of a function $u$ ranging in $(-H,\infty)$. The shape $\mathcal{O}(u)$ is then given by
$$
\mathcal{O}(u):=\{(x,z)\in D\times \R\,:\, -H<z<u(x)\}\,.
$$
The corresponding electrostatic potential $\psi_u$ solves Laplace's equation
$-\Delta\psi_u=0$ in $\mathcal{O}(u)$ with non-homogeneous Dirichlet boundary conditions $\psi_u=h_u\not\equiv \mathrm{const}$ on $\partial\mathcal{O}(u)$, reflecting the potential difference. A possible choice for $h_u$ is 
\begin{equation}
h_u(x,z)=\frac{H+z}{H+u(x)}\,, \qquad (x,z)\in\mathcal{O}(u)\,, \label{y1}
\end{equation} 
which corresponds to both electrodes being held at constant potentials and features an explicit dependence on $u$. Incorporating the boundary values into the electrostatic potential by setting $\varphi_u:=\psi_u-h_u$, one obtains that $\varphi_u\in H_0^1(\mathcal{O}(u))$ solves the Dirichlet problem
\begin{equation}\label{y2}
-\Delta\varphi_u=f_u\quad\text{in }\ \mathcal{O}(u)\,,\qquad \varphi_u=0\quad \text{on }\ \partial\mathcal{O}(u)\,,
\end{equation}
where the regularity of the source term $ f_u:=\Delta h_u$ turns out to be two order less than that of~$u$ \cite{LWBible}; that is, $f_u\in H^{k-2}(\mathcal{O}(u))$ only if $u\in H^k(D)$, a property obviously satisfied for the choice \eqref{y1}. Consequently, application of the above mentioned result to compute the derivative of $\mathcal{J}(\mathcal{O}(u))$ with respect to $u$ requires {\it a priori} a sufficiently high regularity of $u$ and hence of the shape $\mathcal{O}(u)$, which may not be available for the problem under consideration. Indeed, the regularity of the solution $\varphi_u$ to \eqref{y2} is not only controlled by that of $u$, but also limited by the fact that $\mathcal{O}(u)$ is only a Lipschitz domain, so that one may only expect $\varphi_{u}\in H^{\min\{k,3/2\}}(\mathcal{O}(u))$ for $u\in H^k(D)$ in general. This restricted regularity does not seem to be sufficient to give a meaning to the shape derivative of $\mathcal{J}(\mathcal{O}(u))$. Nevertheless, for this particular case (and under suitable assumptions) we show in \cite{LW14,LW16} that $\varphi_{u}\in H^2(\mathcal{O}(u))$ for $u\in H_0^1(D)\cap H^{\alpha}(D)$ with $\alpha>3/2$ and that $\mathcal{J}(\mathcal{O}(u))$ has a shape derivative, which is well-defined and given by 

$$
\mathcal{J}'(\mathcal{O}(u))(x) = \frac{1}{2} |\nabla\psi_u(x,u(x))|^2\,,\quad x\in D\,,
$$
as expected. 
 
 \begin{figure}
 	\begin{tikzpicture}[scale=0.9]
 	\draw[black, line width = 1.5pt, dashed] (-7,0)--(7,0);
 	\draw[black, line width = 2pt] (-7,0)--(-7,-5);
 	\draw[black, line width = 2pt] (7,-5)--(7,0);
 	\draw[black, line width = 2pt] (-7,-5)--(7,-5);
 	\draw[black, line width = 2pt] (-7,-4)--(7,-4);
 	\draw[black, line width = 2pt, fill=gray, pattern = north east lines, fill opacity = 0.5] (-7,-4)--(-7,-5)--(7,-5)--(7,-4);
 	\draw[cadmiumgreen, line width = 2pt] plot[domain=-7:7] (\x,{-1-cos((pi*\x/7) r)});
 	\draw[blue, line width = 2pt] plot[domain=-7:-3] (\x,{-2-2*cos((pi*(\x+3)/4) r)});
 	\draw[blue, line width = 2pt] (-3,-4)--(1,-4);
 	\draw[blue, line width = 2pt] plot[domain=1:7] (\x,{-2-2*cos((pi*(\x-1)/6) r)});
 	\draw[cadmiumgreen, line width = 1pt, arrows=->] (3,0)--(3,-1.15);
 	\node at (3.2,-0.6) {${\color{cadmiumgreen} v}$};
 	\draw[blue, line width = 1pt, arrows=->] (-5,0)--(-5,-1.85);
 	\node at (-4.8,-1) {${\color{blue} w}$};
 	\node[draw,rectangle,white,fill=white, rounded corners=5pt] at (2,-4.5) {$\Omega_1$};
 	\node at (2,-4.5) {$\Omega_1$};
 	\node at (-2,-3) {${\color{cadmiumgreen} \Omega_2(v)}$};
 	\node at (-5.75,-2.5) {${\color{blue} \mathcal{O}_1(w)}$};
 	\node at (5.5,-2.5) {${\color{blue} \mathcal{O}_2(w)}$};
 	\node at (3.75,-5.75) {$D$};
 	\node at (7.75,-3) {$\Sigma$};
 	\draw (7.5,-3) edge[->,bend right, line width = 1pt] (5.5,-3.9);
 	\node at (-7.8,1) {$z$};
 	\draw[black, line width = 1pt, arrows = ->] (-7.5,-6)--(-7.5,1);
 	\node at (-8.4,-5) {$-H-d$};
 	\draw[black, line width = 1pt] (-7.6,-5)--(-7.4,-5);
 	\node at (-8,-4) {$-H$};
 	\draw[black, line width = 1pt] (-7.6,-4)--(-7.4,-4);
 	\node at (-7.8,0) {$0$};
 	\draw[black, line width = 1pt] (-7.6,0)--(-7.4,0);
 	\node at (0,-5.5) {$2L$};
 	\draw[black, line width = 1pt, arrows = <->] (-7,-5.25)--(7,-5.25);
 	\node at (1,-3) {${\color{blue} \mathcal{C}(w)}$};
 	\draw (0.45,-3) edge[->,bend right,blue, line width = 1pt] (-0.5,-3.95);
 	\node at (0.5,-1) {${\color{cadmiumgreen} \mathfrak{G}(v)}$};
 	\draw (0,-1) edge[->,bend right,cadmiumgreen, line width = 1pt] (-1,-1.85);
 	\node at (-2.9,-2) {${\color{blue} \mathfrak{G}(w)}$};
 	\draw (-3.5,-2) edge[->,bend right,blue, line width = 1pt] (-4.3,-2.9);
 	\end{tikzpicture}
 	\caption{Geometry of $\Omega(u)$ for a state $u=v$ with empty coincidence set (\textcolor{cadmiumgreen}{green}) and a state $u=w$ with non-empty coincidence set (\textcolor{blue}{blue}).}\label{Fig1}
 \end{figure}
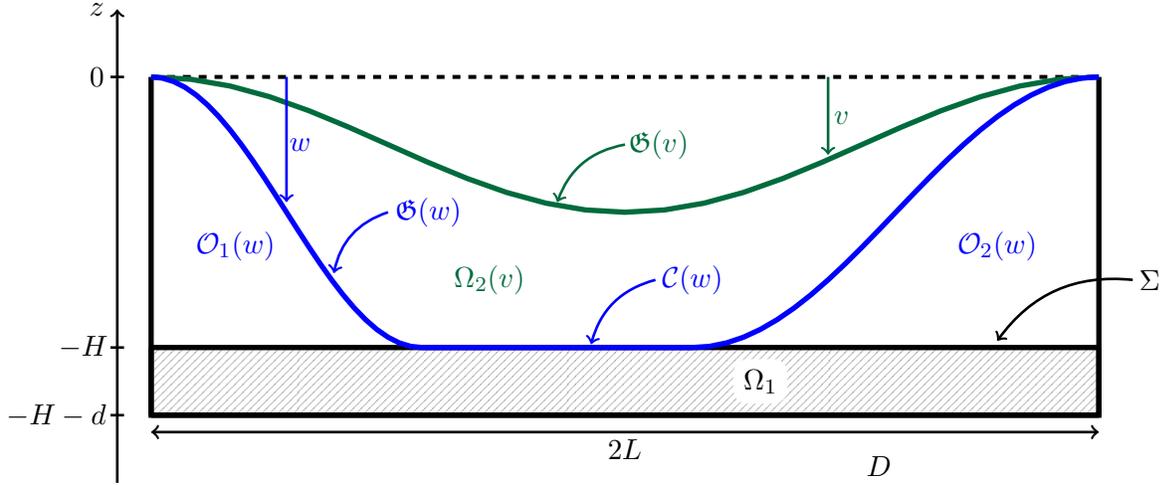
 
Another instance, where a similar difficulty arises, is when the solution to the Dirichlet problem~\eqref{y2} is replaced by the solution to a transmission problem, where the boundary of the domain may contact the transmission interface. Such a situation is encountered in the modeling of microelectromechanical systems (MEMS) \cite{AmEtal,BG01,LW18} and actually provides the impetus for the research performed herein. More details will be given in Section~\ref{LESSMM} below, where the shape derivative computed in Theorems~\ref{Thm1bb} and~\ref{Thm1b} is used to show the existence of stationary solutions to a MEMS model. In such a problem, the geometry of the admissible shapes looks similar to the class $\mathcal{O}(u)$ described above and is defined as follows: Let $H$, $L$, $d>0$ be three positive parameters and set $D:=(-L,L)$. Given a real-valued function $u$ defined on the interval $D$ and ranging in $[-H,\infty)$, the admissible shape $\Omega(u)$ consists of two subregions
$$
 \Omega_1:=D\times (-H-d,-H)
$$
and
$$
 \Omega_2(u):=\left\{(x,z)\in D\times \mathbb{R}\,:\, -H<  z <  u(x)\right\}\,,
$$
which are separated by the interface
$$
\Sigma(u):=\{(x,-H)\,:\, x\in D,\, u(x)>-H\}\,, 
$$
see Figure~\ref{Fig1}; that is,
$$
\Omega({u}):=\left\{(x,z)\in D\times \mathbb{R} \,:\, -H-d<  z <  u(x)\right\}=\Omega_1\cup  \Omega_2( {u})\cup \Sigma(u)\,.
$$
Let us emphasize that we explicitly allow the graph $\mathfrak{G}(u)$ of $u$, defined by
\begin{equation}
\mathfrak{G}(u):=\{(x,u(x))\,:\, x\in D\}\,, \label{Gu}
\end{equation}
to intersect the interface $\Sigma:=D\times\{-H\}$; that is, the \textit{coincidence set} 
\begin{equation}
\mathcal{C}(u):=\{x\in D\,:\, u(x)=-H\} \label{CS}
\end{equation}
of $u$ may be non-empty, resulting in a disconnected top part $\Omega_2(u)$ with connected components $(\mathcal{O}_i(u))_i$ -- see the blue curve in Figure~\ref{Fig1}. If $\mathcal{C}(u)$ is empty -- see the green curve in Figure~\ref{Fig1} -- then
$$
\Sigma(u)=\Sigma=D\times\{-H\}\,.
$$
The dielectric properties of $\Omega_1$ and $\Omega_2(u)$ being different with a jump discontinuity at the interface $\Sigma(u)$, the potential $\psi_u\in H^1(\Omega(u))$ under consideration in this paper is defined as the variational solution to the transmission problem
\begin{subequations}\label{TMP}
\begin{align}
\mathrm{div}(\sigma\nabla\psi_u)&=0 \quad\text{in }\ \Omega(u)\,, \label{TMP1}\\
\llbracket \psi_u \rrbracket = \llbracket \sigma \partial_z \psi_u \rrbracket &=0\quad\text{on }\ \Sigma(u)\,, \label{TMP2}\\
\psi_u&=h_u\quad\text{on }\ \partial\Omega(u)\,, \label{TMP3}
\end{align}   
\end{subequations}
where  $\llbracket\cdot\rrbracket$ denotes the jump across $\Sigma(u)$. Here and in the following, $\sigma_1\in C^2({\overline\Omega_1})$ with $\sigma_1(x,z)> 0$, $\sigma_2>0$  is a positive constant (with $\sigma_1(\cdot,-H)\not\equiv\sigma_2$), and  $h_u\in H^1(\Omega(u))$ is a given function defining the boundary values of $\psi_u$ on $\partial\Omega(u)$. The associated Dirichlet energy is
\begin{equation}
\mathfrak{J}(u):= \frac{1}{2}\int_{\Omega(u)} \sigma \vert\nabla \psi_u\vert^2\,\rd (x,z)\,. \label{DE}
\end{equation}
The main contribution of the present research is the computation of the shape derivative of $\mathfrak{J}(u)$ with respect to $u$ in an appropriate functional setting. Several steps are needed to achieve this goal. According to the discussion above, the first step is to derive sufficient regularity on $\psi_u$, keeping in mind that $\psi_u$ depends on $u$ not only through $\Omega_2(u)$, but also through $h_u$. An appropriate functional setting for $u$ turns out to be the set 
$$
\bar S:=\{u\in H^2(D) \cap H_0^1(D)\,:\, u\ge -H \text{ in } D\}\,.
$$ 
Let us already point out that $\mathcal{C}(u)=\emptyset$ if and only if
$$
u\in S:=\{u\in H^2(D) \cap H_0^1(D)\,:\, u> -H \text{ in } D\}\,.
$$
The variational setting for the potential $\psi_u$ is then 
$$
\mathcal{A}(u):=h_{u} + H_{0}^1(\Omega(u))\,,
$$
where the boundary values $h_u$ are defined by
$$
h_u(x,z):= h(x,z,u(x)) = \left\{ \begin{array}{ll}
h_1(x,z,u(x))\,, & (x,z)\in \overline\Omega_1\,, \\
h_2(x,z,u(x))\,, & (x,z)\in \overline{\Omega_2(u)}\,,
\end{array} \right.
$$
the given function $h$ satisfying \eqref{bobbybrown} below. The well-posedness of \eqref{TMP} is provided by the following result.

\begin{theorem}\label{Thm1}
Let the function $h$ satisfy \eqref{bobbybrown} below. 
\begin{itemize}
\item[(a)] For each $u\in \bar S$, there is a unique variational solution $\psi_u \in \mathcal{A}(u)$ to \eqref{TMP}.  Moreover, $\psi_{u,1}:= \psi_{u}\vert_{\Omega_1} \in H^2(\Omega_1)$ and $\psi_{u,2} := \psi_{u}\vert_{\Omega_2(u)} \in H^2(\Omega_2(u))$, and $\psi_{u}$  is a strong solution to the transmission problem~\eqref{TMP} satisfying $\sigma\partial_z \psi_u\in H^1(\Omega(u))$.
\item[(b)] Given $\kappa>0$, there is $c(\kappa)>0$ such that, for all $u\in\bar S$ satisfying $\|u\|_{H^2(D)}\le \kappa$, 
\begin{equation*}
\|\psi_u\|_{H^1(\Omega(u))} + \|\sigma\partial_z \psi_u\|_{H^1(\Omega(u))} + \|\psi_{u,1}\|_{H^2(\Omega_1)} + \|\psi_{u,2}\|_{H^2(\Omega_2(u))} \le c(\kappa)\,.
\end{equation*}	
\end{itemize}
\end{theorem}

\begin{proof}
This follows from Proposition~\ref{P1} and Corollary~\ref{P1b}.
\end{proof}

While the existence and uniqueness of $\psi_u\in\mathcal{A}(u)$ as a variational solution to \eqref{TMP} are straightforward consequences of Lax-Milgram's theorem, the $H^2$-regularity is more involved, in particular when the coincidence set $\mathcal{C}(u)$ is non-empty. In that case, $\Omega_2(u)$ is not connected and has a non-Lipschitz boundary due to turning points $(x_0,-H)\in \partial\Omega_2(u)$ with $u(x_0)+H=\partial_x u(x_0)=0$. The first issue is then to have a meaningful definition of the trace on the boundary of $\Omega_2(u)$. This is possible here, thanks to the specific geometry of $\Omega_2(u)$ which is enclosed by the graphs of two Lipschitz continuous functions, a feature which has already been noticed in the literature, see \cite{AADL05, MNP00} for instance. Once the issue of traces is settled, we still face the difficulty that $\psi_u$ satisfies the transmission conditions \eqref{TMP2} on $\Sigma(u)\ne \Sigma$ but is subject to the Dirichlet boundary conditions $\psi_u=h_u$ on $\Sigma\setminus\Sigma(u)$. 

We shall thus begin with the simplest situation, where the coincidence set $\mathcal{C}(u)$ is empty -- see the green curve in Figure~\ref{Fig1}. For smooth functions $u\in S\cap W_\infty^2(D)$, the piecewise $H^2$-regularity of solutions to the transmission problem \eqref{TMP} is known \cite{Lem77}. The strategy to extend it to arbitrary functions in $\bar S$ requires to overcome the above mentioned difficulties and includes two steps: on the one hand, we derive quantitative estimates on $\psi_u$ in $H^2(\Omega_1)$ and $H^2(\Omega_2(u))$ for $u\in S\cap W_\infty^2(D)$, which depend, neither on the $W_\infty^2$-norm of $u$, nor on the positivity of $u+H$, as stated in Theorem~\ref{Thm1}~\textbf{(b)}. On the other hand, we show that $u\mapsto \psi_u-h_u$ is a continuous map from $\bar S$ to $H^1(\mathbb{R}^2)$ when $\bar S$ is endowed with the topology of $H^1_0(D)$, the proof relying on the $\Gamma$-convergence of the functionals associated with the variational formulation defining $\psi_u$. Combining these two results leads us to Theorem~\ref{Thm1}. 

\begin{remark}\label{ECwasHere}
As already pointed out, for $u\in\bar{S}\setminus S$, the connected components of $\Omega_2(u)$ are not Lipschitz domains, as they feature at least one cuspidal point $(x_0,-H)\in \partial\Omega_2(u)$ with $u(x_0)+H=\partial_x u(x_0)=0$. Thus, the $H^2(\Omega_2(u))$-regularity of $\psi_{u,2}$ does not guarantee \textit{a priori} well-defined traces on the boundary of such connected components for $\psi_{u,2}$ and $\partial_z \psi_{u,2}$. Nevertheless, these traces are here well-defined, owing to the $H^1(\Omega(u))$-regularity of both $\psi_u$ and $\partial_z \psi_u$, recalling that $\Omega(u)$ as a whole is obviously a Lipschitz domain, see Theorem~\ref{Thm1}~(a). Note that no such issue arises in the rectangle $\Omega_1$.
\end{remark}

Next, due to the regularity properties of $\psi_u$ provided by Theorem~\ref{Thm1}, we can compute the shape derivative of the Dirichlet energy $\mathfrak{J}(u)$ with respect to $u\in S$ in a classical way \cite{HP05,SZ92}.

\begin{theorem}\label{Thm1bb}
Let the function $h$ satisfy \eqref{bobbybrown} below and consider $u\in S$. Introducing
\begin{equation*}
\mathfrak{g}(u)(x) := \frac{\sigma_2}{2} \big(1+(\partial_x u(x))^2\big)\,\left[\partial_z\psi_{u,2}-(\partial_z h_2)_u-(\partial_w h_2)_u\right]^2(x, u(x))\,, \qquad x\in D\,,
\end{equation*}
and endowing $S$ with the $H^2(D)$-topology, the Dirichlet energy $\mathfrak{J}:S\rightarrow\R$ defined in \eqref{DE} is continuously Fr\'echet differentiable with
\begin{equation*}
\begin{split}
\partial_u \mathfrak{J}(u)(x)
=& -\mathfrak{g}(u)(x) +\frac{\sigma_2}{2} \,\left[ \big((\partial_x h_2)_u\big)^2+ \big((\partial_z h_2)_u+(\partial_w h_2)_u\big)^2 \right](x, u(x)) \\
& - \left[ \sigma_1 (\partial_w h_1)_u\,\partial_z\psi_{u,1} \right](x,-H-d)
\end{split}
\end{equation*}
for $u\in S$ and $x\in D$, where $(\psi_{u,1},\psi_{u,2})$ is defined in Theorem~\ref{Thm1}, 
\begin{equation*}
(\partial_w h_1)_u(x,z):= \partial_w h_1(x,z,u(x))\,, \qquad (x,z)\in \overline\Omega_1\,,
\end{equation*} 
and
\begin{equation*}
\big( (\partial_x h_2)_u , (\partial_z h_2)_u , (\partial_w h_2)_u  \big)(x,z) := \left( \partial_x h_2, \partial_z h_2, \partial_w h_2 \right)(x,z,u(x))
\end{equation*}
for $(x,z)\in \overline{\Omega_2(u)}$.
\end{theorem}

\begin{proof}
This follows from Proposition~\ref{ACDC1} and Proposition~\ref{C15}.
\end{proof}

The proof of Theorem~\ref{Thm1bb} is performed along the lines of the proof of \cite[Theorem~5.3.2]{HP05} and relies on the following observation: for $u\in S$, there is a neighborhood $\mathcal{V}$ of $u$ in $S$ such that, for any $v\in\mathcal{V}$, there is a bi-Lipschitz transformation mapping $\Omega(v)$ onto $\Omega(u)$. Such a transformation then allows us to convert $\mathfrak{J}(v)$ for each $v\in\mathcal{V}$ to an integral over $\Omega(u)$ and investigate the behavior of the difference $\mathfrak{J}(v)-\mathfrak{J}(u)$ as $v\to u$.

\bigskip

The just outlined approach obviously fails for $u\in \bar S\setminus S$, since the coincidence set $\mathcal{C}(u)$ is non-empty. Indeed, in that case, it does not seem to be possible to find a bi-Lipschitz transformation mapping $\Omega(v)$ onto $\Omega(u)$, unless their coincidence sets are equal, $\mathcal{C}(v)=\mathcal{C}(u)$, an assumption which is far too restrictive. We instead use an approximation argument and show that the Dirichlet energy $\mathfrak{J}$ admits directional derivatives in the directions $-u+S$, as stated in the next result.

\begin{theorem}\label{Thm1b}
Let the function $h$ satisfy \eqref{bobbybrown} below and consider $u\in\bar S$. Introducing
\begin{equation*}
\mathfrak{g}(u)(x):=\left\{\begin{array}{ll}  \displaystyle{\frac{\sigma_2}{2} \big(1+(\partial_x u(x))^2\big)\,\big[\partial_z\psi_{u,2}-(\partial_z h_2)_u-(\partial_w h_2)_u\big]^2(x, u(x))} \,, & x\in D\setminus \mathcal{C}(u)\,,\\
\hphantom{x}\vspace{-3.5mm}\\
\displaystyle{\frac{\sigma_2}{2}\, \left[ \frac{\sigma_1}{\sigma_2}\partial_z\psi_{u,1}-(\partial_z h_2)_u-(\partial_w h_2)_u \right]^2(x, -H)}\,,  & x\in  \mathcal{C}(u)\,,
\end{array}\right.
\end{equation*}
then, for $w\in S$,
\begin{equation*}
\begin{split}
\lim_{t\rightarrow 0^+} \frac{1}{t}\big(\mathfrak{J}(&u+t(w-u))-\mathfrak{J}(u)\big)\\
= &-\int_D \mathfrak{g}(u)(x) (w-u)(x)\, \rd x\\
& +\frac{\sigma_2}{2} \int_D \left[ \big((\partial_x h_2)_u\big)^2+ \big((\partial_z h_2)_u+(\partial_w h_2)_u\big)^2 \right](x, u(x))\,(w-u)(x)\,\rd x\\
& -\int_D \left[ \sigma_1(\partial_w h_1)_u\,\partial_z\psi_{u,1} \right](x,-H-d)\,(w-u)(x)\,\rd x\,,
\end{split}
\end{equation*}
the notation being the same as in Theorem~\ref{Thm1bb}. Moreover, the function $\mathfrak{g}: \bar S\rightarrow L_p(D)$ is continuous for each $p\in [1,\infty)$, the set $\bar{S}$ being endowed with the topology of $H^2(D)$.
\end{theorem}

\begin{proof}
This follows from Proposition~\ref{ACDC1} and Corollary~\ref{C17}.
\end{proof}

Observe that, for $u\in S$, the formula for $\mathfrak{g}(u)$ in Theorem~\ref{Thm1b} matches that of $\mathfrak{g}(u)$ in Theorem~\ref{Thm1bb}, since $\mathcal{C}(u)$ is empty in that case. The proof of Theorem~\ref{Thm1b} relies on Theorem~\ref{Thm1bb}, using the fact that $u+t(w-u)\in S$ for $t\in (0,1)$ when $u\in\bar S$ and $w\in S$. The main step is actually the computation of $\mathfrak{g}(u)$ for $u\in\bar S\setminus S$. To this end, we consider a bounded sequence $(u_n)_{n\ge 1}$ in $S$ converging to $u$ in $H_0^1(D)$ and identify the limit of $\mathfrak{g}(u_n)$ as $n\to\infty$. Of importance here are the uniform  $H^2$-estimates on $(\psi_{u_n})_{n\ge 1}$ proved in Theorem~\ref{Thm1}~\textbf{(b)}.

\bigskip

We end the introduction with a description of the contents of the subsequent sections. 

In Section~\ref{NaC} we provide the precise assumptions on the function $h$ defining the boundary conditions \eqref{TMP3} of the potential $\psi_u$, see \eqref{bobbybrown} and \eqref{bb}. 

The derivation of the $H^2$-estimates stated in Theorem~\ref{Thm1}~\textbf{(b)} is next performed in Section~\ref{Potential}. We begin Section~\ref{Potential} by recalling the well-posedness of the variational formulation associated with the transmission problem \eqref{TMP} and $H^2$-regularity properties of $\psi_u$ when $u\in S\cap W_\infty^2(D)$. For such $u$ we derive in Section~\ref{UEPP} quantitative estimates on $\psi_u$ in $H^2(\Omega_1)$ and $H^2(\Omega_2(u))$. To this end, we further develop the approach from \cite{LW16} and heavily use the property that $\Omega_2(u)$ can be mapped in a bi-Lipschitz way onto the rectangle $D\times (0,1)$ when $u\in S$. To extend the validity of the $H^2$-estimates to all $u\in \bar S$, special attention is paid to the dependence of the various constants arising in the estimates derived for $u\in S\cap W_\infty^2(D)$, including that involved in Sobolev embeddings. We show in particular that the estimates depend, neither on the $W_\infty^2$-regularity of $u$, nor on the positivity of $u+H$. For the extension to $u\in \bar S$, we employ then an approximation argument, relying on the density of $S\cap W_\infty^2(D)$ in $\bar S$. Specifically, given $u \in \bar S$, we consider a sequence $(u_n)_{n\ge 1}$ in $S\cap W_\infty^2(D)$, which is bounded in $H^2(D)$ and converges to $u$ in $H^1(D)$. A $\Gamma$-convergence argument provided in Section~\ref{GCDE} then implies  that $(\psi_{u_n}-h_{u_n})_{n\ge 1}$ converges to $\psi_u - h_u$ in $H^1(\mathbb{R}^2)$. Combining the outcome of Sections~\ref{UEPP} and~\ref{GCDE} allows us to complete the proof of Theorem~\ref{Thm1} in Section~\ref{HEPP}.  Finally, in preparation of the proof of Theorem~\ref{Thm1b}, we identify in Section~\ref{LBTVD} the behavior of the vertical derivative $x\mapsto \partial_z \psi_{u_n,2}(x,u_n(x))$, $x\in D$, as $n\to\infty$ for a sequence $(u_n)_{n\ge 1}$ in $S$ converging to $u\in\bar S$ in the norm of $H^1(D)$. Since the coincidence set $\mathcal{C}(u)$ of $u$ may be non-empty and possibly includes countably many connected components, this step requires some care  for the analysis in $\mathcal{C}(u)$, while a different argument is needed in $D\setminus \mathcal{C}(u)$. 

In Section~\ref{SDDE}, we turn to the study of the differentiability of the Dirichlet energy $\mathfrak{J}(u)$, see  \eqref{DE}, with respect to $u\in \bar S$. In this regard, we first establish the  Fr\'echet differentiability of $\mathfrak{J}$ on $S$, the proof following closely \cite{HP05}. We thus obtain the Fr\'echet derivative $\partial_u\mathfrak{J}(u)$ for $u\in S$ in the form given in Theorem~\ref{Thm1bb}. We then consider $u\in \bar S \setminus S$ and combine the outcome of Theorem~\ref{Thm1bb} and Section~\ref{LBTVD} to prove Theorem~\ref{Thm1b}.

Finally, Section~\ref{LESSMM} is devoted to an application of Theorems~\ref{Thm1bb} and~\ref{Thm1b} to identify the Euler-Lagrange equation satisfied by the minimizers of a functional arising in the modeling of microelectromechanical systems.

\section{Notations and Conventions}\label{NaC}

Given a subset $R$ of $\R^n$ with Lipschitz boundary, we let $H_{0}^1(R)$ denote the space of functions in $H^1(R)$ vanishing on the boundary $\partial R$ (in the sense of traces) and denote its dual space by $H^{-1}(R)$.

Recall that
$$
S=\{v\in H^2(D)\cap H_0^1(D)\,:\, v>-H \text{ in } D\}\,,
$$
so that its $H^2$-closure is $\bar S$ introduced above.
Given $v\in \bar S$ and a pair of real-valued functions $(\vartheta_1,\vartheta_2)$ with $\vartheta_1$ defined on $\Omega_1$ and  $\vartheta_2$ defined on $\Omega_2(v)$, we put
$$
\vartheta:=\left\{\begin{array}{ll} \vartheta_1 & \text{in}\ \Omega_1\,,\\
\vartheta_2 & \text{in}\ \Omega_2(v)\,,
\end{array}\right.
$$
and let
$$ 
\llbracket \vartheta \rrbracket (x,-H):=\vartheta_1(x,-H)-\vartheta_2(x,-H)\,,\qquad x\in D\setminus\mathcal{C}(v)\,,
$$ 
denote the jump across the interface $\Sigma(v)$ (if meaningful). Recall that the coincidence set $\mathcal{C}(v)$ is defined in \eqref{CS}. In particular, we set
$$
\sigma:=\left\{\begin{array}{ll} \sigma_1 & \text{in}\ \Omega_1\,,\\
\sigma_2 & \text{in}\ \Omega_2(v)\,.
\end{array}\right.
$$
Conversely, if $\vartheta$ is defined in $\Omega(v)$, then we denote the corresponding restrictions by $\vartheta_1 := \vartheta\vert_{\Omega_1}$ and $\vartheta_2 := \vartheta\vert_{\Omega_2(v)}$.

For further use we set 
\begin{equation}\label{sigmamin}
\sigma_{min}:=\min\left\{ \sigma_2,\min_{\overline\Omega_1}\sigma_1 \right\} >0\,, \qquad \sigma_{max}:=\max\left\{ \sigma_2, \|\sigma_1\|_{C^2(\bar\Omega_1)} \right\}<\infty\,.
\end{equation}

As described in the introduction, for $v\in\bar{S}$, the values of the potential $\psi_v$ on the boundary $\partial\Omega(v)$ are given by a function $h_v$. For technical reasons we assume that $h_v$ is not only defined on $\partial\Omega(v)$ but also has an extension to $\overline{\Omega(v)}$. More precisely, we fix $C^2$-functions
\begin{subequations}\label{bobbybrown}
\begin{equation}\label{bobbybrown2a}
h_1: \bar{D}\times [-H-d,-H]\times [-H,\infty)\rightarrow [0,\infty)
\end{equation}
and 
\begin{equation}\label{bobbybrown2aa}
h_2: \bar{D}\times [-H,\infty)\times [-H,\infty)\rightarrow [0,\infty)
\end{equation}
satisfying
\begin{align}
h_1(x,-H,w)&=h_2(x,-H,w)\,,\quad (x,w)\in D\times [-H,\infty)\,,\label{bobbybrown2}\\
 \sigma_1(x,-H)\partial_z h_1(x,-H,w)& =\sigma_2\partial_z h_2(x,-H,w)\,,\quad (x,w)\in D\times [-H,\infty)\,.\label{bobbybrown3}
\end{align}
\end{subequations}
For a given function $v\in \bar S$ we then define 
\begin{equation}\label{bb}
h_v(x,z):=\left\{\begin{array}{ll} h_{v,1}(x,z):= h_1(x,z,v(x))\,, & (x,z)\in \overline\Omega_1\,,\\
h_{v,2}(x,z):=h_2(x,z,v(x))\,, & (x,z)\in D\times [-H,\infty)\,.
\end{array}\right.  
\end{equation}
Note that \eqref{bobbybrown2}-\eqref{bobbybrown3} imply
\begin{align}\label{bobbybrown4}
\llbracket h_v \rrbracket &= \llbracket \sigma \partial_z h_v \rrbracket =0\quad\text{on }\ \Sigma(v)\,.
\end{align}
Consequently, by \eqref{bobbybrown4},
\begin{equation}
h_v\in H^1(\Omega(v)) \;\;\text{ with }\;\; (h_{v,1},h_{v,2})\in H^2(\Omega_1)\times H^2(\Omega_2(v))\,. \label{pif}
\end{equation} 

\section{The Potential}\label{Potential}

\newcounter{NumS3Const}

Given $v\in \bar S$ we recall the set of admissible potentials
$$
\mathcal{A}(v)=h_{v}+H_{0}^1(\Omega(v))
$$
and define the functional
\begin{equation}\label{sos}
\mathcal{J}(v)[\vartheta]:=\frac{1}{2}\int_{ \Omega(v)} \sigma \vert\nabla \vartheta\vert^2\,\rd (x,z)\,,\quad \vartheta\in \mathcal{A}(v)\,.
\end{equation}
The potential $\psi_v$ corresponding to $v\in \bar S$ and solving the transmission problem \eqref{TMP} is then the minimizer of the functional $\mathcal{J}(v)$ on the set $\mathcal{A}(v)$; that is,
\begin{equation*}
\mathcal{J}(v)[\psi_v] = \min_{\vartheta\in \mathcal{A}(v)} \{\mathcal{J}(v)[\vartheta]\} \,.
\end{equation*}

We first prove Theorem~\ref{Thm1} for $v\in S\cap W_\infty^2(D)$; that is, for smooth $v$ with empty coincidence set $\mathcal{C}(v)$.

\begin{proposition}\label{P1}
{\bf (a)} For each $v\in \bar S$ there is a unique minimizer $\psi_v \in \mathcal{A}(v)$ of $\mathcal{J}(v)$ on~$\mathcal{A}(v)$. 

\noindent{\bf (b)} If $v\in S \cap W_\infty^2(D)$, then $\psi_v=(\psi_{v,1},\psi_{v,2})\in H^2(\Omega_1)\times H^2(\Omega_2(v))$  satisfies the transmission problem
\begin{subequations}\label{a1}
\begin{align}
\mathrm{div}(\sigma\nabla\psi_v)&=0 \quad\text{in }\ \Omega(v)\,,\label{a1a}\\
\llbracket \psi_v \rrbracket = \llbracket \sigma \partial_z \psi_v \rrbracket &=0\quad\text{on }\ \Sigma\,,\label{a1b}\\
\psi_v&=h_v\quad\text{on }\ \partial\Omega(v)\,.\label{a1c}
\end{align}
\noindent{\bf (c)} If $v\in S \cap W_\infty^2(D)$, then $\sigma\partial_z\psi_v\in H^1(\Omega(v))$.
\end{subequations}
\end{proposition}

\begin{proof}
\noindent {\bf (a)}  Let $v\in \bar S$. The existence and uniqueness of a minimizer $\psi_v$ of $\mathcal{J}(v)$ on the set $\mathcal{A}(v)$ follow at once from the Lax-Milgram theorem, the positive lower bound \eqref{sigmamin} on $\sigma$, Poincar\'e's inequality, the convexity of $\mathcal{J}(v)$, and the property $\mathrm{div}(\sigma\nabla h_v)\in H^{-1}(\Omega(v))$ due to \eqref{pif}. 

\noindent {\bf (b)}  Next, the minimizing property of $\psi_v$ entails that
 $\mathcal{J}(v)[\psi_v]\le \mathcal{J}(v)[\psi_v+t\vartheta]$ for each $t\in \R$ and 
$\vartheta\in H_0^1(\Omega(v))$. By definition of $\mathcal{J}(v)$, this readily gives
\begin{equation}\label{e1}
\int_{\Omega(v)}\sigma\nabla\psi_v\cdot\nabla \vartheta\,\rd (x,z)=0\,,\quad \vartheta\in H_0^1(\Omega(v))\,.
\end{equation}
Now, if $v\in S \cap W_\infty^2(D)$, then \cite[Theorem III.4.6]{Lem77} ensures the existence of a unique solution $\tilde\psi=(\tilde\psi_1,\tilde\psi_2)\in H^2(\Omega_1)\times H^2(\Omega_2(v))$ to the transmission problem \eqref{a1}. Clearly, $\tilde \psi\in \mathcal{A}(v)$ satisfies \eqref{e1}, thus $\psi_v=\tilde\psi$.

\noindent {\bf (c)} It follows from the regularity of $\sigma_1$, Proposition~\ref{P1}~{\bf (b)}, and \eqref{a1b} that
$$
\sigma\partial_z\psi_v\in H^1(\Omega_1)\cup H^1(\Omega_2(v))
$$
with zero jump across the interface; that is, $\llbracket \sigma \partial_z \psi_v \rrbracket =0$ on $\Sigma$. This implies $\sigma\partial_z\psi_v\in H^1(\Omega(v))$.
\end{proof}

We shall later prove that Proposition~\ref{P1}~{\bf (b)} extends to all $v\in \bar S$. To this end, we need to give a meaning to the transmission condition \eqref{a1b} when $v\in\bar{S}$ and the boundary of $\Omega_2(v)$ is not Lipschitz, see Corollary~\ref{P1b}. We also note the following $H^1$-estimate for $\psi_v$.

\begin{lemma}\label{C1}
Given $v\in \bar S$,
$$
\int_{\Omega(v)}\sigma\vert\nabla\psi_v\vert^2\,\rd(x,z)\le \int_{\Omega(v)}\sigma\vert\nabla h_v\vert^2\,\rd(x,z)\,.
$$
\end{lemma}

\begin{proof}
This follows from $h_v\in\mathcal{A}(v)$ and the minimizing property of $\psi_v$ stated in Proposition~\ref{P1}~{\bf (a)}.
\end{proof}

For our purpose we need, besides the extension of Proposition~\ref{P1}~{\bf (b)} to all $v\in \bar S$, precise information on the dependence of $\psi_v$ on $v$. Such information is, unfortunately, not included in the approach of \cite{Lem77}.

\subsection{Uniform Estimates on the  Potential $\psi_v$}\label{UEPP}

For $v\in S \cap W_\infty^2(D)$ we  denote the unique minimizer of $\mathcal{J}(v)$ on $\mathcal{A}(v)$  by $\psi_v \in \mathcal{A}(v)$, with 
$$
\psi_v=(\psi_{v,1},\psi_{v,2})\in H^2(\Omega_1)\times H^2(\Omega_2(v))\,,
$$ 
as provided by Proposition~\ref{P1}. In that case, the coincidence $\mathcal{C}(v)$ of $v$, defined in \eqref{CS}, is empty, so that $\Sigma(v)=\Sigma$, see Figure~\ref{Fig2}.  
 We next define
\begin{equation}\label{chi}
\chi=\chi_v:=\psi_v-h_v\in H_{0}^1(\Omega(v))\,,\quad (\chi_{1},\chi_{2}) := (\chi_{v,1},\chi_{v,2})\in H^2(\Omega_1)\times H^2(\Omega_2(v))\,,
\end{equation}
suppressing in the following the dependence of $\chi$ on the fixed $v$ for ease of notation. Recalling~\eqref{bobbybrown4}, we obtain from Proposition~\ref{P1} that $\chi$ satisfies the transmission problem
\begin{subequations}\label{a2}
\begin{align}
\mathrm{div}(\sigma\nabla\chi)&=-\mathrm{div}(\sigma\nabla h_v) \quad\text{in }\ \Omega(v)\,,\label{a2a}\\
\llbracket \chi \rrbracket = \llbracket \sigma \partial_z \chi \rrbracket &=0\quad\text{on }\ \Sigma\,,\label{a2b}\\
\chi&=0\quad\text{on }\ \partial\Omega(v)\,.
\end{align}
\end{subequations}

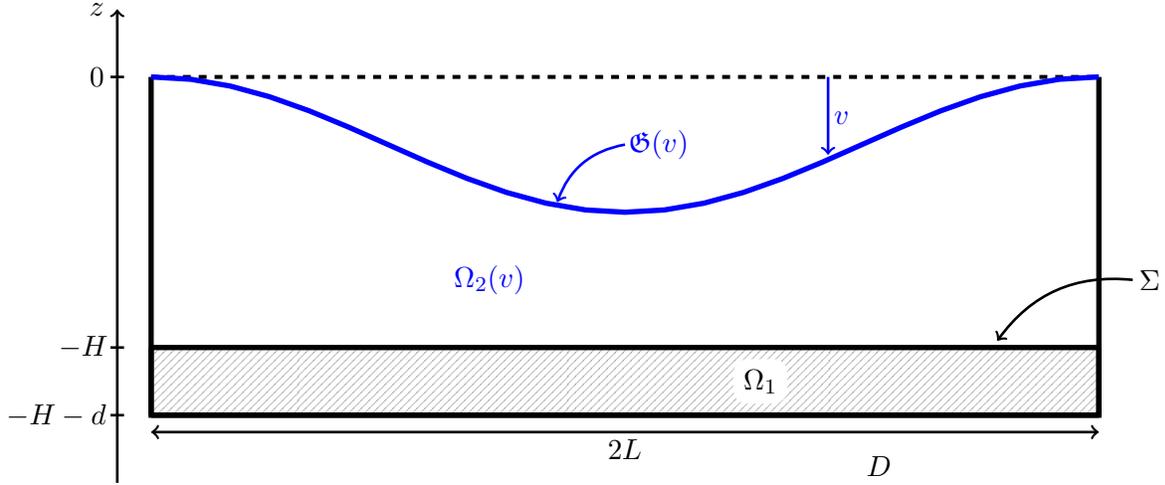
\begin{figure}
	\begin{tikzpicture}[scale=0.9]
	\draw[black, line width = 1.5pt, dashed] (-7,0)--(7,0);
	\draw[black, line width = 2pt] (-7,0)--(-7,-5);
	\draw[black, line width = 2pt] (7,-5)--(7,0);
	\draw[black, line width = 2pt] (-7,-5)--(7,-5);
	\draw[black, line width = 2pt] (-7,-4)--(7,-4);
	\draw[black, line width = 2pt, fill=gray, pattern = north east lines, fill opacity = 0.5] (-7,-4)--(-7,-5)--(7,-5)--(7,-4);
	\draw[blue, line width = 2pt] plot[domain=-7:7] (\x,{-1-cos((pi*\x/7) r)});
	\draw[blue, line width = 1pt, arrows=->] (3,0)--(3,-1.15);
	\node at (3.2,-0.6) {${\color{blue} v}$};
	\node[draw,rectangle,white,fill=white, rounded corners=5pt] at (2,-4.5) {$\Omega_1$};
	\node at (2,-4.5) {$\Omega_1$};
	\node at (-2,-3) {${\color{blue} \Omega_2(v)}$};
	\node at (3.75,-5.75) {$D$};
	\node at (7.75,-3) {$\Sigma$};
	\draw (7.5,-3) edge[->,bend right, line width = 1pt] (5.5,-3.9);
	\node at (-7.8,1) {$z$};
	\draw[black, line width = 1pt, arrows = ->] (-7.5,-6)--(-7.5,1);
	\node at (-8.4,-5) {$-H-d$};
	\draw[black, line width = 1pt] (-7.6,-5)--(-7.4,-5);
	\node at (-8,-4) {$-H$};
	\draw[black, line width = 1pt] (-7.6,-4)--(-7.4,-4);
	\node at (-7.8,0) {$0$};
	\draw[black, line width = 1pt] (-7.6,0)--(-7.4,0);
	\node at (0,-5.5) {$2L$};
	\draw[black, line width = 1pt, arrows = <->] (-7,-5.25)--(7,-5.25);
	\node at (0.5,-1) {${\color{blue} \mathfrak{G}(v)}$};
	\draw (0,-1) edge[->,bend right,blue, line width = 1pt] (-1,-1.85);
	\end{tikzpicture}
	\caption{Geometry of $\Omega(v)$ for a state $v\in S$ with empty coincidence set.}\label{Fig2}
\end{figure}

Our aim is now to derive an estimate for $\chi=\chi_v=(\chi_{v,1},\chi_{v,2})$ in  the norm of $H^2(\Omega_1)\times H^2(\Omega_2(v))$, which only depends on $\|v\|_{H^2(D)}$ but, neither on the norm of $v$ in $W_\infty^2(D)$, nor on the value of $\min_D\{v+H\}$. This then allows us to extend Proposition~\ref{P1}~\textbf{(b)} to all $v\in \bar S$. 

Analogously to Proposition~\ref{P1}~\textbf{(c)}, an immediate consequence of \eqref{chi} and \eqref{a2b} is the $H^1$-regularity of $\sigma\partial_z\chi$.

\begin{lemma}\label{L-1} \refstepcounter{NumS3Const}\label{S3cst1}
Let $v\in S \cap W_\infty^2(D)$ and $\chi=\psi_v-h_v$. Then $\sigma\partial_z\chi\in H^1(\Omega(v))$ and $$\|\sigma\partial_z\chi\|_{H^1(\Omega(v))}\le c_{\ref{S3cst1}} \left( \|\partial_z\chi_1\|_{H^1(\Omega_1)} + \|\partial_z\chi_2\|_{H^1(\Omega_2(v))} \right)
$$
for some constant $c_{\ref{S3cst1}}>0$ independent of $v$.
\end{lemma}

To derive an $H^2$-estimate on $\chi$ (more precisely, on $\chi_{1}$  and on $\chi_{2}$) we  transform \eqref{a2} to a transmission problem on a rectangle. To keep a flat interface between the two subregions, we transform $\Omega_1$ to the rectangle $D\times (-d,0)$ and $\Omega_2(v)$ to the rectangle $D\times (0,1)$. More precisely, we introduce the transformation
\begin{equation}\label{t1}
T_1(x,z):=\left(x,H+z\right)\ ,\quad (x,z)\in {\Omega_1}\ ,
\end{equation}
mapping $\Omega_1$ onto the rectangle \mbox{$\mathcal{R}_1:=D\times (-d,0)$}, and the transformation
\begin{equation}\label{t2}
T_2(x,z):=\left(x,\frac{H+z}{H+v(x)}\right)\ ,\quad (x,z)\in {\Omega_2(v)}\ ,
\end{equation}
mapping $\Omega_2(v)$ onto the fixed rectangle \mbox{$\mathcal{R}_2:=D\times (0,1)$}. Then
$$
\Sigma_0:=D\times \{0\}
$$ 
is the interface separating $\mathcal{R}_1$ and $\mathcal{R}_2$. We set
$$
\mathcal{R}:=D\times (-d,1)=\mathcal{R}_1\cup \mathcal{R}_2\cup \Sigma_0
$$
and let $(x,\eta)$ denote the new variables in $\mathcal{R}$, i.e. $(x,\eta)=T_1(x,z)$ in $\mathcal{R}_1$ and $(x,\eta)=T_2(x,z)$ in $\mathcal{R}_2$.
We also introduce
$$
\hat\sigma:=\left\{\begin{array}{ll} \sigma_1\circ (T_1)^{-1} & \text{in}\ \mathcal{R}_1\,,\\
\hphantom{x}\vspace{-3.5mm}\\
\dfrac{\sigma_2}{H+v} & \text{in}\ \mathcal{R}_2\,.
\end{array}\right.
$$
Then, by \eqref{chi} and \eqref{a2b}, 
\begin{equation}\label{C0}
\Phi=(\Phi_1,\Phi_2)\in H_0^1(\mathcal{R})\,,\qquad
\Phi_j:= \chi_j\circ (T_j)^{-1}\in H^2(\mathcal{R}_j)\,,\quad j=1,2\,,
\end{equation} 
and
\begin{equation}\label{transPhi}
\llbracket \Phi\rrbracket =\llbracket \hat\sigma \partial_\eta\Phi\rrbracket=0\quad\text{on }\ \Sigma_0\,.
\end{equation}
We will make use of this regularity often in the following without mention. In particular, as $\Phi$ vanishes on $\partial \mathcal{R}$ (and is smooth enough),
\begin{equation}\label{C}
\partial_x\Phi(x,-d)=\partial_x\Phi(x,1)=\partial_\eta \Phi (\pm L,\eta)=0\,,\qquad x\in D\,,\quad \eta\in (-d,1)\,.
\end{equation}
We begin with an identity for $\Phi$, which is based on \cite[Lemma~4.3.1.2]{Gr85} and  fundamental for the forthcoming analysis.

\begin{lemma}\label{L1}
Given $v\in S \cap W_\infty^2(D)$ and with the above notation,
$$
\int_\mathcal{R}\partial_x^2\Phi\,\partial_\eta\left(\hat\sigma\partial_\eta\Phi\right)\,\rd (x,\eta)=\int_\mathcal{R}\partial_{x} \partial_{\eta}\Phi\,\partial_x\left(\hat\sigma\partial_\eta\Phi\right)\,\rd (x,\eta)\,.
$$
\end{lemma}

\begin{proof}
We adapt the proof of \cite[Lemma~II.2.2]{Lem77}. Since $\llbracket \Phi\rrbracket=0$ on $\Sigma_0$ by \eqref{transPhi}, we get 
\begin{equation}\label{i}
 \llbracket \partial_x\Phi\rrbracket=0\quad\text{ on }\ \Sigma_0\,.
\end{equation}
Then \eqref{i} along with \eqref{C0} imply that 
$$
F:=\partial_x\Phi=(\partial_x\Phi_1,\partial_x\Phi_2)\in H^1(\mathcal{R})\,,
$$
while \eqref{transPhi} along with \eqref{C0} imply that
$$
G:=\hat\sigma\partial_\eta\Phi=(\hat\sigma_1\partial_\eta\Phi_1,\hat\sigma_2\partial_\eta\Phi_2)\in H^1(\mathcal{R})\,.
$$
Consequently, the regularity of $(F,G)$ together with \eqref{transPhi}, \eqref{C}, and \eqref{i} allow us to apply \cite[Lemma~4.3.1.2, Lemma~4.3.1.3]{Gr85} from which we deduce that
$$
\int_\mathcal{R}\partial_xF\partial_\eta G\,\rd (x,\eta)=\int_\mathcal{R}\partial_{\eta}F\partial_x G\,\rd (x,\eta)\,,
$$
as claimed.
\end{proof}

Based on the previous lemma we derive the following identity, which subsequently leads to the desired $H^2$-estimates on $\psi=\psi_v$.

\begin{lemma}\label{L2}
Let $v\in S \cap W_\infty^2(D)$. Then, for $\chi=\psi_v-h_v$,
\begin{equation*}
\begin{split}
\int_{\Omega_1 \cup \Omega_2(v)}&\mathrm{div}\left(\sigma\nabla\chi\right)\,\partial_z^2\chi\,\rd (x,z)\\
=&
\int_{\Omega_1 \cup \Omega_2(v)}\partial_z\left(\sigma\partial_x\chi\right)\,\partial_{x}\partial_{z}\chi\,\rd (x,z)+\int_{\Omega_1 \cup \Omega_2(v)}\partial_z\left(\sigma\partial_z\chi\right)\,\partial_{z}^2\chi\,\rd (x,z)\\
&+\int_{D}\big(\partial_x\sigma_1 \partial_x\chi_1\partial_z\chi_1\big)(x,-H)\,\rd x-\frac{\sigma_2}{2}\int_{D}\partial_x^2 v(x)\left(\partial_z\chi_2(x,v(x))\right)^2\,\rd x\,.
\end{split}
\end{equation*}
\end{lemma}

\begin{proof}
Let us first emphasize that the regularity of $\Phi$ stated in \eqref{C0} ensures the validity of the subsequent computations. Using the transformations $T_1$ and $T_2$ introduced above (and the fact that $\sigma_2$ is constant) we get
\begin{equation*}
\begin{split}
J:=&\int_{\Omega_1 \cup \Omega_2(v)}\partial_x\left(\sigma\partial_x\chi\right)\,\partial_z^2\chi\,\rd (x,z)\\
&=
\int_{\mathcal{R}_1}\partial_x\left(\hat\sigma_1\partial_x\Phi_1\right)\,\partial_{\eta}^2\Phi_1\,\rd (x,\eta)\\
&\qquad+\int_{\mathcal{R}_2}\hat\sigma_2\partial_\eta^2\Phi_2\,\left[\partial_{x}^2\Phi_2 -2\eta\frac{\partial_x v}{H+v}\partial_{x}\partial_{\eta}\Phi_2 +\eta^2\left(\frac{\partial_x v}{H+v}\right)^2 \partial_\eta^2\Phi_2\right.\\
&\left. \qquad\qquad\qquad\qquad\qquad +2\eta\left(\frac{\partial_x v}{H+v}\right)^2 \partial_\eta\Phi_2 -\eta\frac{\partial_x^2 v}{H+v}\partial_\eta\Phi_2\right]\,\rd (x,\eta)\,.
\end{split}
\end{equation*}
We next combine the integral on $\mathcal{R}_1$ and the integral on $\mathcal{R}_2$ stemming from the first term in the square brackets to get
\begin{equation*}
\begin{split}
\int_{\mathcal{R}_1}&\partial_x\left(\hat\sigma_1\partial_x\Phi_1\right)\,\partial_{\eta}^2\Phi_1\,\rd (x,\eta)+\int_{\mathcal{R}_2}\hat\sigma_2\partial_\eta^2\Phi_2\,\partial_{x}^2\Phi_2\,\rd (x,\eta)\\
&= \int_{\mathcal{R}}\partial_x^2\Phi\,\partial_\eta\left(\hat\sigma\partial_\eta\Phi\right)\,\rd (x,\eta)
+\int_{\mathcal{R}_1}\left[\partial_x\hat\sigma_1\partial_x\Phi_1\partial_\eta^2\Phi_1-\partial_\eta\hat\sigma_1\partial_x^2\Phi_1\partial_\eta\Phi_1\right]\,\rd (x,\eta)\\
&= \int_{\mathcal{R}}\partial_{x}\partial_{\eta}\Phi\,\partial_x\left(\hat\sigma\partial_\eta\Phi\right)\,\rd (x,\eta)
+\int_{\mathcal{R}_1}\left[\partial_x\hat\sigma_1\partial_x\Phi_1\partial_\eta^2\Phi_1-\partial_\eta\hat\sigma_1\partial_x^2\Phi_1\partial_\eta\Phi_1\right]\,\rd (x,\eta)\,,
\end{split}
\end{equation*}
where we have used Lemma~\ref{L1} to obtain the second identity. Splitting again the integral on $\mathcal{R}$ into integrals on $\mathcal{R}_1$ and $\mathcal{R}_2$ and gathering the above computations give
\begin{equation}\label{7}
\begin{split}
J&=\int_{\mathcal{R}_1}\hat\sigma_1\left(\partial_{x}\partial_{\eta}\Phi_1\right)^2\,\rd (x,\eta)
+\int_{\mathcal{R}_1}\partial_x\hat\sigma_1 \partial_{\eta}\Phi_1\partial_{x}\partial_{\eta}\Phi_1\,\rd (x,\eta)\\
&\quad +\int_{\mathcal{R}_1}\left[\partial_x\hat\sigma_1 \partial_{x}\Phi_1\partial_{\eta}^2\Phi_1-\partial_\eta\hat\sigma_1\partial_x^2\Phi_1\partial_\eta\Phi_1\right]\,\rd (x,\eta)\\
&\quad +\int_{\mathcal{R}_2}\hat\sigma_2\left(\partial_{x}\partial_{\eta}\Phi_2\right)^2\,\rd (x,\eta)
-\int_{\mathcal{R}_2}\frac{\hat\sigma_2\partial_x v}{H+v}\partial_{\eta}\Phi_2\partial_{x}\partial_{\eta}\Phi_2\,\rd (x,\eta)\\
&\quad +\int_{\mathcal{R}_2}\hat\sigma_2\partial_\eta^2\Phi_2\,\left[-2\eta\frac{\partial_x v}{H+v}\partial_{x}\partial_{\eta}\Phi_2 +\eta^2\left(\frac{\partial_x v}{H+v}\right)^2 \partial_\eta^2\Phi_2\right.\\
&\left. \qquad\qquad\qquad \qquad +2\eta\left(\frac{\partial_x v}{H+v}\right)^2 \partial_\eta\Phi_2 -\eta\frac{\partial_x^2 v}{H+v}\partial_\eta\Phi_2\right]\,\rd (x,\eta)\\
&=: \, I(\mathcal{R}_1)+I(\mathcal{R}_2)\,.
\end{split}
\end{equation}
To handle $I(\mathcal{R}_1)$ we first consider the third integral involving the square brackets. We integrate by parts its first term with respect to $\eta$ and its second term with respect to $x$. Using \eqref{C} to get rid of the corresponding boundary terms and noticing that the resulting terms involving $\partial_x\partial_\eta\hat\sigma_1$ cancel, this yields
\begin{equation*}
\begin{split}
\int_{\mathcal{R}_1}&\left[\partial_x\hat\sigma_1 \partial_{x}\Phi_1\partial_{\eta}^2\Phi_1-\partial_\eta\hat\sigma_1\partial_x^2\Phi_1\partial_\eta\Phi_1\right]\,\rd (x,\eta)\\
&=\int_{D}\int_{-d}^0 \partial_x\hat\sigma_1 \partial_{x}\Phi_1\partial_{\eta}^2\Phi_1 \,\rd\eta\rd x-\int_{-d}^0 \int_{D} \partial_\eta\hat\sigma_1\partial_x^2\Phi_1\partial_\eta\Phi_1\,\rd x\rd\eta\\
&=\int_{D} \left(\partial_x\hat\sigma_1\partial_x\Phi_1\partial_\eta\Phi_1\right)(x,0)\,\rd x -\int_{\mathcal{R}_1}\partial_x\hat\sigma_1\partial_{x}\partial_{\eta}\Phi_1\partial_\eta\Phi_1\,\rd (x,\eta)\\
&\quad+\int_{\mathcal{R}_1}\partial_\eta\hat\sigma_1\partial_{x}\Phi_1\partial_{x}\partial_{\eta}\Phi_1\,\rd (x,\eta)\,.
\end{split}
\end{equation*}
Consequently,
\begin{equation*}
\begin{split}
I(\mathcal{R}_1)=& \int_{\mathcal{R}_1}\hat\sigma_1\left(\partial_{x}\partial_{\eta}\Phi_1\right)^2\,\rd (x,\eta)
+ \int_{\mathcal{R}_1}\partial_\eta\hat\sigma_1\partial_{x}\Phi_1\partial_{x}\partial_{\eta}\Phi_1\,\rd (x,\eta)\\
&+\int_{D} \left(\partial_x\hat\sigma_1\partial_x\Phi_1\partial_\eta\Phi_1\right)(x,0)\,\rd x\,;
\end{split}
\end{equation*}
that is, using the transformation $T_1$ to write the integral in terms of $\chi_1$,
\begin{equation}\label{5}
\begin{split}
I(\mathcal{R}_1)=& \int_{\Omega_1}\partial_z\left(\sigma_1\partial_{x}\chi_1\right)\partial_{x}\partial_{z}\chi_1\,\rd (x,z)
+ \int_D \left(\partial_x\sigma_1\partial_x\chi_1\partial_z\chi_1\right)(x,-H)\,\rd x\,.
\end{split}
\end{equation}
We next turn to $I(\mathcal{R}_2)$ and gather some of the terms to get
\begin{equation*}
\begin{split}
I(\mathcal{R}_2)=
&\int_{\mathcal{R}_2}\sigma_2\left[\frac{\partial_{x}\partial_{\eta}\Phi_2}{H+v}-\frac{\partial_x v}{(H+v)^2}\partial_{\eta}\Phi_2-\eta\frac{\partial_x v}{(H+v)^2}\partial_\eta^2\Phi_2\right]^2(H+v)\,\rd (x,\eta)\\
&+\int_{\mathcal{R}_2}\hat\sigma_2\frac{\partial_x v}{H+v}\partial_\eta\Phi_2\left[\partial_{x}\partial_{\eta}\Phi_2
-\frac{\partial_x v}{H+v}\partial_{\eta}\Phi_2\right]\,\rd (x,\eta)\\
&-\int_{\mathcal{R}_2}\hat\sigma_2\eta\frac{\partial_x^2 v}{H+v}\partial_\eta\Phi_2\partial_\eta^2\Phi_2\,\rd (x,\eta)\,.
\end{split}
\end{equation*}
We then focus on the last term of this identity. Integrating first in $\eta$ and then by parts in $x$, using again \eqref{C} to cancel the corresponding boundary terms, yields
\begin{equation*}
\begin{split}
-\int_{\mathcal{R}_2}\hat\sigma_2\eta&\frac{\partial_x^2 v}{H+v}\partial_\eta\Phi_2\partial_\eta^2\Phi_2\,\rd (x,\eta)\\
& = -\frac{\sigma_2}{2}\int_{D}\partial_x^2 v(x)\left(\frac{\partial_\eta\Phi_2(x,1)}{H+v(x)}\right)^2\,\rd x +\frac{\sigma_2}{2}\int_{\mathcal{R}_2}\frac{\partial_x^2 v}{(H+v)^2}\left(\partial_\eta\Phi_2\right)^2\,\rd (x,\eta)\\
& =-\frac{\sigma_2}{2}\int_{D}\partial_x^2 v(x)\left(\frac{\partial_\eta\Phi_2(x,1)}{H+v(x)}\right)^2\,\rd x\\
 & \quad \ +\int_{\mathcal{R}_2} \hat\sigma_2\left[\frac{(\partial_x v)^2}{(H+v)^2}\left(\partial_\eta\Phi_2\right)^2-\frac{\partial_x v}{H+v}\partial_\eta\Phi_2\partial_{x}\partial_{\eta}\Phi_2\right]\,\rd (x,\eta)\,.
\end{split}
\end{equation*}
Hence, gathering the previous two identities and noticing the cancellation of terms entail
\begin{equation*}
\begin{split}
I(\mathcal{R}_2)=
&\int_{\mathcal{R}_2}\sigma_2\left[\frac{\partial_{x}\partial_{\eta}\Phi_2}{H+v}-\frac{\partial_x v}{(H+v)^2}\partial_{\eta}\Phi_2-\eta\frac{\partial_x v}{(H+v)^2}\partial_\eta^2\Phi_2\right]^2(H+v)\,\rd (x,\eta)\\
&-\frac{\sigma_2}{2}\int_{D}\partial_x^2 v(x)\left(\frac{\partial_\eta\Phi_2(x,1)}{H+v(x)}\right)^2\,\rd x\,.
\end{split}
\end{equation*}
Since
\begin{equation*}
\begin{split}
\partial_{x}\partial_{z}\chi_2(x,z)= & \frac{1}{H+v(x)}\partial_{x}\partial_{\eta}\Phi_2\left(x,\frac{H+z}{H+v(x)}\right)
 -\frac{\partial_x v(x)}{(H+v(x))^2}\partial_\eta \Phi_2\left(x,\frac{H+z}{H+v(x)}\right)\\
&- \frac{\partial_x v(x) (H+z)}{(H+v(x))^3}\partial_\eta^2\Phi_2\left(x,\frac{H+z}{H+v(x)}\right)\,,
\end{split}
\end{equation*}
we use $T_2$ to transform $I(\mathcal{R}_2)$ back to $\chi_2$ and find
\begin{equation}\label{6}
I(\mathcal{R}_2)= \int_{\Omega_2(v)}\sigma_2\left(\partial_{x}\partial_{z}\chi_2\right)^2\,\rd (x,z)-\frac{1}{2}\int_D\sigma_2\partial_x^2 v(x)\left(\partial_z\chi_2(x,v(x))\right)^2\,\rd x\,.
\end{equation}
Plugging \eqref{5}-\eqref{6} into \eqref{7} and recalling that $\sigma_2$ is constant, the assertion readily follows.
\end{proof}

The right-hand side of the identity of Lemma~\ref{L2} involves ``bulk'' terms in $\Omega_1\cup \Omega_2(v)$ and a contribution on the interface $\Sigma$ and the top part $\mathfrak{G}(v)$, see \eqref{Gu}, which all require to be handled differently in order to derive the desired $H^2$-estimates on $\chi_v$. We begin with the first interface integral on $\Sigma$ and observe:

\begin{lemma}\label{L3} \refstepcounter{NumS3Const}\label{S3cst2}
Given $\alpha\in (1/2,1)$ there is $c_{\ref{S3cst2}}(\alpha)>0$ such that, for $v\in S \cap W_\infty^2(D)$ and $\chi=\psi_v-h_v$,
\begin{equation*}
\left\vert\int_D\big(\partial_x\sigma_1 \partial_x\chi_1\partial_z\chi_1\big)(x,-H)\,\rd x\right\vert
\le c_{\ref{S3cst2}}(\alpha)\,\|\chi_1\|_{H^1(\Omega_1)}^{2(1-\alpha)} \,\|\chi_1\|_{H^2(\Omega_1)}^{2\alpha}\,.
\end{equation*}
\end{lemma}

\begin{proof}
By complex interpolation, $H^\alpha(\Omega_1)\doteq [L_2(\Omega_1),H^1(\Omega_1)]_\alpha$, which guarantees
$$
\| w\|_{H^\alpha(\Omega_1)}\le c(\alpha) \|w\|_{L_2(\Omega_1)}^{1-\alpha} \|w\|_{H^1(\Omega_1)}^{\alpha}\,,\quad w\in H^1(\Omega_1)\,.
$$
Since the trace operator is continuous from $H^\alpha(\Omega_1)$ to $L_2(D\times \{-H\})$, see \cite[Theorem~1.5.1.2]{Gr85}, we deduce from \eqref{sigmamin} that
\begin{equation*}
\begin{split}
\left\vert\int_D\big(\partial_x\sigma_1 \partial_x\chi_1\partial_z\chi_1\big)(x,-H)\,\rd x\right\vert
&\le \|\partial_x\sigma_1 \|_\infty\, \|\partial_x\chi_1\|_{L_2(D\times \{-H\})}\,\|\partial_z\chi_1\|_{L_2(D\times \{-H\})}\\
&\le c(\alpha)\, \|\partial_x\chi_1\|_{H^\alpha(\Omega_1)}\,\|\partial_z\chi_1\|_{H^\alpha(\Omega_1)}\\
&\le 
c(\alpha)\,\|\chi_1\|_{H^1(\Omega_1)}^{2(1-\alpha)} \,\|\chi_1\|_{H^2(\Omega_1)}^{2\alpha}\,,
\end{split}
\end{equation*}
as claimed.
\end{proof}

Let us point out that the transformation $(T_1,T_2)$ introduced in \eqref{t1}-\eqref{t2} and used in the proof of Lemma~\ref{L2}
features a singularity as $v$ approaches $-H$, a property which prevents its use for $v\in\bar S$. To circumvent this drawback, we shall introduce a different transformation which maps $\Omega(v)$ as a whole onto a fixed rectangle, but does not 
preserve the flatness of the interface between the two subregions $\Omega_1$ and $\Omega_2(v)$ (see \eqref{TT} below). As we shall see, such a transformation allows us to derive functional inequalities for all $v\in\bar S$ depending only on the $H^2$-norm of $v$. This mild dependence turns out to be of utmost importance for the forthcoming analysis.

\begin{lemma}\label{L4q} \refstepcounter{NumS3Const}\label{S3cst3}
Given $\kappa>0$ and $q\in [1,\infty)$, there is $c_{\ref{S3cst3}}(q,\kappa)>0$ such that, 
for $v\in\bar S$ with $\|v\|_{H^2(D)}\le \kappa$, 
\begin{equation}\label{n10}
\|\theta\|_{L_q(\Omega(v))}\le c_{\ref{S3cst3}}(q,\kappa) \|\theta\|_{H^1(\Omega(v))}\,,\quad \theta\in H^1(\Omega(v))\,.
\end{equation}
\refstepcounter{NumS3Const}\label{S3cst4} 
Moreover, given $\alpha\in (0,1/2]$, there is $c_{\ref{S3cst4}}(\alpha,\kappa)>0$ such that, for $v\in\bar S$ with $\|v\|_{H^2(D)}\le \kappa$, 
\begin{equation}\label{n11}
\begin{split}
\|\theta(\cdot,v)\|_{H^\alpha(D)}\le c_{\ref{S3cst4}}(\alpha,\kappa) \|\theta\|_{L_2(\Omega(v))}^{(1-2\alpha)/2} \|\theta\|_{H^1(\Omega(v))}^{(2\alpha+1)/2}\,, \quad \theta\in H^1(\Omega(v))\,.
\end{split}
\end{equation}
\end{lemma}

\begin{proof}
We use the transformation
\begin{equation}\label{TT}
\mathfrak{T}(x,z):=\mathfrak{T}_v(x,z):=\left(x,\frac{H+d+z}{H+d+v(x)}\right)\ ,\quad (x,z)\in {\Omega(v)}\ ,
\end{equation}
to map $\Omega(v)$ onto the rectangle $\mathcal{R}_2 =D\times (0,1)$. Given $\theta\in H^1(\Omega(v))$, we define
 $\phi:=\theta\circ \mathfrak{T}^{-1}$ so that
\begin{align*}
\phi(x,\eta)&=\theta\big(x,-H-d+(H+d+v(x))\eta\big)\,,\\
 \partial_x\phi(x,\eta)&=\partial_x\theta\big(x,-H-d+(H+d+v(x))\eta\big)\\
&\qquad+\eta\partial_x v(x)\partial_z\theta\big(x,-H-d+(H+d+v(x))\eta\big)\,,\\
\partial_\eta\phi(x,\eta)&=(H+d+v(x)) \partial_z \theta\big(x,-H-d+(H+d+v(x))\eta\big)
\end{align*}
for $(x,\eta)\in \mathcal{R}_2$. It easily follows from the previous formulas, the continuous embedding of $H^2(D)$ in $W_\infty^1(D)$, and the assumed bound on $v$ that
\begin{align}
\|\phi\|_{H^1(\mathcal{R}_2)}& \le c(\kappa)\|\theta\|_{H^1(\Omega(v))}\,,\label{n1}\\
 \|\phi\|_{L_q(\mathcal{R}_2)}& \le \frac{1}{d^{1/q}}\|\theta\|_{L_q(\Omega(v))}\,,\label{n2}\\
\|\theta\|_{L_q(\Omega(v))}& \le c(q,\kappa)\|\phi\|_{L_q(\mathcal{R}_2)}\,.\label{n3}
\end{align}
On the one hand, \eqref{n10} now readily follows from \eqref{n1}, \eqref{n3} and the continuous embedding of $H^1(\mathcal{R}_2)$  in $L_q(\mathcal{R}_2)$ for all $q\in [1,\infty)$. On the other hand, the continuity of the trace as a mapping from $H^1(\mathcal{R}_2 )$ to $H^{1/2}(D\times \{1\})$, see \cite[Theorem~1.5.1.2]{Gr85}, and \eqref{n1} ensure that
\begin{equation*}
\begin{split} 
\|\theta(\cdot,v)\|_{H^{1/2}(D)} &=\|\phi(\cdot, 1) \|_{H^{1/2}(D)}\le c\|\phi\|_{H^1(\mathcal{R}_2 )}\le c(\kappa) \|\theta\|_{H^{1}(\Omega(v))}\,.
\end{split}
\end{equation*}
Finally, let $\alpha\in (0,1/2)$. By complex interpolation,
$$
[L_2(\mathcal{R}_2 ),H^1(\mathcal{R}_2 )]_{\alpha+1/2}\doteq H^{\alpha+1/2} (\mathcal{R}_2 )\,,
$$
so that
$$
\|\phi\|_{H^{\alpha+1/2}(\mathcal{R}_2 )}\le c(\alpha) \|\phi\|_{L_2(\mathcal{R}_2 )}^{(1-2\alpha)/2} \|\phi\|_{H^1(\mathcal{R}_2 )}^{(2\alpha+1)/2}\,.
$$
Since $\alpha>0$, the  trace maps $H^{\alpha+1/2}(\mathcal{R}_2 )$ continuously to $H^{\alpha}(D\times \{1\})$ and we thus deduce that
\begin{equation*}
\begin{split} 
\|\theta(\cdot,v)\|_{H^{\alpha}(D)} &=\|\phi(\cdot, 1) \|_{H^{\alpha}(D)}\le c(\alpha)\|\phi\|_{L_2(\mathcal{R}_2 )}^{(1-2\alpha)/2} \|\phi\|_{H^1(\mathcal{R}_2 )}^{(2\alpha+1)/2}\\
&\le  c(\alpha,\kappa)\|\theta\|_{L_2(\Omega(v))}^{(1-2\alpha)/2}\|\theta\|_{H^1(\Omega(v))}^{(2\alpha+1)/2}\,,
\end{split}
\end{equation*}
the last inequality stemming from \eqref{n1} and \eqref{n2}.
\end{proof}

As for the boundary integral over $\mathfrak{G}(v)$ on the right-hand side of the identity of Lemma~\ref{L2} we note:

\begin{lemma}\label{L4} \refstepcounter{NumS3Const}\label{S3cst5} 
Given $\zeta\in (3/4,1)$ and $\kappa>0$, there is $c_{\ref{S3cst5}}(\zeta,\kappa)>0$ such that, for $v\in S \cap W_\infty^2(D)$ with $\|v\|_{H^2(D)}\le \kappa$ and \mbox{$\chi=\psi_v-h_v$},
\begin{equation*}
\begin{split}
\left\vert\frac{\sigma_2}{2}\int_D\partial_x^2 v(x)\left(\partial_z\chi_2(x,v(x))\right)^2\,\rd x\right\vert
\le c_{\ref{S3cst5}}(\zeta,\kappa)\,\|\partial_z\chi\|_{L_2(\Omega(v))}^{2(1-\zeta)}\,\|\sigma\partial_z\chi\|_{H^1(\Omega(v))}^{2\zeta}\,.
\end{split}
\end{equation*}
\end{lemma}

\begin{proof}
By the Cauchy-Schwarz inequality,
\begin{equation*}
\begin{split}
\left\vert\frac{\sigma_2}{2}\int_D\partial_x^2 v(x)\left(\partial_z\chi_2(x,v(x))\right)^2\,\rd x\right\vert
\le& (\sigma_2)^{-1}
\|\partial_x^2 v\|_{L_2(D)}\, \|\sigma_2 \partial_z\chi_2(\cdot,v)\|_{L_4(D)}^2
\end{split}
\end{equation*}
and it remains to estimate the term involving $\chi_2$. To this end, since $\sigma \partial_z \chi$ belongs to $H^1(\Omega(v))$ by Lemma~\ref{L-1}, we can use the functional inequality \eqref{n11} (with $\alpha=\zeta-1/2$), \eqref{sigmamin}, and the continuous embedding of $H^{\zeta-1/2}(D)$ in $L_4(D)$ to obtain
\begin{equation*}
\begin{split} 
\|\sigma_2 \partial_z\chi_2(\cdot,v)\|_{L_4(D)}^2 &\le \|\sigma_2 \partial_z\chi_2(\cdot,v)\|_{H^{(2\zeta-1)/2}(D)}^2
\le c(\zeta,\kappa) \|\partial_z\chi\|_{L_2(\Omega(v))}^{2(1-\zeta)} \|\sigma\partial_z\chi\|_{H^1(\Omega(v))}^{2\zeta}\,,
\end{split}
\end{equation*}
as claimed.
\end{proof}

We are now in a position to derive the desired estimate on $\psi_v$.

\begin{proposition}\label{P2} \refstepcounter{NumS3Const}\label{S3cst6} 
Given $\kappa>0$, there is a constant $c_{\ref{S3cst6}}(\kappa)>0$ such that
\begin{subequations}\label{est}
\begin{equation}\label{est1}
\|\chi_v\|_{H^1(\Omega(v))} + \|\chi_{v,1}\|_{H^2(\Omega_1)} + \|\chi_{v,2}\|_{H^2(\Omega_2(v))}\le c_{\ref{S3cst6}}(\kappa)
\end{equation}
and
\begin{equation}\label{est2}
\|\psi_v\|_{H^1(\Omega(v))}+\|\psi_{v,1}\|_{H^2(\Omega_1)}+\|\psi_{v,2}\|_{H^2(\Omega_2(v))}  + \|\sigma \partial_z \psi_v \|_{H^1(\Omega(v))} \le c_{\ref{S3cst6}}(\kappa)\,,
\end{equation}
\end{subequations}
whenever $v\in S \cap W_\infty^2(D)$ with $\|v\|_{H^2(D)}\le \kappa$.
\end{proposition}

\begin{proof}
Since $\mathrm{div}(\sigma\nabla\chi)=-\mathrm{div}(\sigma\nabla h_v)$ in $\Omega(v)$ by \eqref{a2a}, it follows from Lemma~\ref{L2} that
\begin{equation*}
\begin{split}
&\int_{\Omega_1 \cup \Omega_2(v)}\sigma\left(\partial_{x}\partial_{z}\chi\right)^2\,\rd (x,z)+\int_{\Omega_1 \cup \Omega_2(v)}\sigma\left(\partial_z^2\chi\right)^2\,\rd (x,z) \\
&\quad = -\int_{\Omega_1 \cup \Omega_2(v)}\mathrm{div}\left(\sigma\nabla h_v\right)\,\partial_z^2\chi\,\rd (x,z)-
\int_{\Omega_1}\partial_z\sigma_1\left\{\partial_x\chi_1\,\partial_{x}\partial_{z}\chi_1+\partial_z\chi_1\,\partial_{z}^2\chi_1\right\}\,\rd (x,z)\\
&\qquad-\int_D\big(\partial_x\sigma_1 \partial_x\chi_1\partial_z\chi_1\big)(x,-H)\,\rd x+\frac{\sigma_2}{2}\int_D\partial_x^2 v(x)\left(\partial_z\chi_2(x,v(x))\right)^2\,\rd x\,.
\end{split}
\end{equation*}
We next use \eqref{sigmamin} and the Cauchy-Schwarz inequality for the integrals on $\Omega_1 \cup \Omega_2(v)$ on the right-hand side. Incorporating the resulting terms involving second order derivatives of $\chi$ on the left-hand side and recalling Lemma~\ref{L3} and Lemma~\ref{L4}, we deduce
\begin{equation}\label{u}
\begin{split}
\int_{\Omega_1 \cup \Omega_2(v)}&\sigma\left(\partial_{x}\partial_{z}\chi\right)^2\,\rd (x,z)+\int_{\Omega_1 \cup \Omega_2(v)}\sigma\left(\partial_z^2\chi\right)^2\,\rd (x,z) \\
&\le  c \int_{\Omega_1 \cup \Omega_2(v)}\vert\mathrm{div}\left(\sigma\nabla h_v\right)\vert^2\,\rd (x,z)+
c \int_{\Omega_1} \sigma_1\vert\nabla\chi_1\vert^2\,\rd(x,z)\\
& \quad 
+c_{\ref{S3cst2}}(\alpha)\,\|\chi_1\|_{H^1(\Omega_1)}^{2(1-\alpha)} \,\|\chi_1\|_{H^2(\Omega_1)}^{2\alpha}
+c_{\ref{S3cst5}}(\alpha,\kappa)\,\|\partial_z\chi\|_{L_2(\Omega(v))}^{2(1-\alpha)}\,\|\sigma\partial_z\chi\|_{H^1(\Omega(v))}^{2\alpha}
\end{split}
\end{equation}
for some fixed $\alpha\in (3/4,1)$. We now aim at controlling the last two terms of the right-hand side by the term on the left-hand side. For the first term we obtain from \eqref{sigmamin} and \eqref{a2a} 
\begin{align*}
\|\chi_1\|_{H^2(\Omega_1)}^{2\alpha}
&\le 
\left( \|\partial_x^2\chi_1\|_{L_2(\Omega_1)}^2+\|\partial_{x}\partial_{z}\chi_1\|_{L_2(\Omega_1)}^2+ \|\partial_z^2\chi_1\|_{L_2(\Omega_1)}^2\right)^\alpha  +\|\chi_1\|_{H^1(\Omega_1)}^{2\alpha}\\
&\le \Big( \sigma_{min}^{-2} \|-\mathrm{div}(\sigma_1\nabla h_{v,1})-\partial_z(\sigma_1\partial_z\chi_1)-\partial_x\sigma_1\partial_x\chi_1\|_{L_2(\Omega_1)}^2\\
&\qquad \qquad\qquad \qquad+\|\partial_{x}\partial_{z}\chi_1\|_{L_2(\Omega_1)}^2 +\|\partial_z^2\chi_1\|_{L_2(\Omega_1)}^2\Big)^\alpha +\|\chi_1\|_{H^1(\Omega_1)}^{2\alpha}\,.
\end{align*}
By Young's inequality and \eqref{sigmamin}, we obtain for $\epsilon\in (0,1)$
\begin{align*}
\|\chi_1\|_{H^1(\Omega_1)}^{2(1-\alpha)} \|\chi_1\|_{H^2(\Omega_1)}^{2\alpha}&\le
\epsilon
\|-\mathrm{div}(\sigma_1\nabla h_{v,1})-\partial_z(\sigma_1\partial_z\chi_1)-\partial_x\sigma_1\partial_x\chi_1\|_{L_2(\Omega_1)}^2\\
&\qquad+\epsilon \|\partial_{x}\partial_{z}\chi_1\|_{L_2(\Omega_1)}^2 +\epsilon\|\partial_z^2\chi_1\|_{L_2(\Omega_1)}^2 +\left(1+\frac{c}{\epsilon}\right)\|\chi_1\|_{H^1(\Omega_1)}^2\\
&\le 2\epsilon \|\sigma_1\partial_z^2\chi_1\|_{L_2(\Omega_1)}^2+\epsilon \|\partial_{x}\partial_{z}\chi_1\|_{L_2(\Omega_1)}^2 +\epsilon\|\partial_z^2\chi_1\|_{L_2(\Omega_1)}^2 \\
&\qquad + c\left(1+\frac{1}{\epsilon}\right)\|\chi_1\|_{H^1(\Omega_1)}^2  +c\|\mathrm{div}(\sigma_1\nabla h_{v,1})\|_{L_2(\Omega_1)}^2\,.
\end{align*}
We use once more \eqref{sigmamin} and choose
$$
\epsilon:=\min\left\{ \frac{1}{16 \sigma_{max}},\frac{\sigma_{min}}{8}, \frac{1}{2}\right\}
$$
to obtain
\begin{align*}
\|\chi_1\|_{H^1(\Omega_1)}^{2(1-\alpha)} \|\chi_1\|_{H^2(\Omega_1)}^{2\alpha}&\le
\frac{1}{4}\int_{\Omega_1}\sigma_1\left\{\left(\partial_{x}\partial_{z}\chi_1\right)^2+\left(\partial_z^2\chi_1\right)^2\right\}\,\rd (x,z)\\
&\qquad + c\int_{\Omega_1} \left( \vert\chi_1\vert^2+ \sigma_1\vert\nabla\chi_1\vert^2\right)\,\rd(x,z)\\
&\qquad + c  \int_{\Omega_1}\left\vert\mathrm{div}\left(\sigma_1\nabla h_{v,1}\right)\right\vert^2\,\rd (x,z)\,.
\end{align*}
Finally, since $\chi_1(x,-H-d)=0$ for $x\in D$, a generalized  Poincar\'e's inequality (see \cite[II.Section~1.4]{Tem}) and \eqref{sigmamin} entail that
\begin{align*}
\int_{\Omega_1} \vert\chi_1\vert^2\,\rd(x,z)&\le c \int_{\Omega_1} \vert\nabla\chi_1\vert^2\,\rd(x,z)
\le c \int_{\Omega_1} \sigma_1 \vert\nabla\chi_1\vert^2\,\rd(x,z)\,,
\end{align*}
so that
\begin{align}
\|\chi_1\|_{H^1(\Omega_1)}^{2(1-\alpha)} \|\chi_1\|_{H^2(\Omega_1)}^{2\alpha}&\le
\frac{1}{4}\int_{\Omega_1}\sigma_1\left\{\left(\partial_{x}\partial_{z}\chi_1\right)^2+\left(\partial_z^2\chi_1\right)^2\right\}\,\rd (x,z)\nonumber\\
&\quad  + c\int_{\Omega_1} \sigma_1\vert\nabla\chi_1\vert^2\,\rd(x,z) \nonumber\\
&\quad  + c  \int_{\Omega_1}\left\vert\mathrm{div}\left(\sigma_1\nabla h_{v,1}\right)\right\vert^2\,\rd (x,z)\,.\label{uu}
\end{align}
Similarly, for $\epsilon\in (0,1)$, it follows from \eqref{sigmamin} and Young's inequality that
\begin{align*}
\|\partial_z\chi\|_{L_2(\Omega(v))}^{2(1-\alpha)}\,&\|\sigma\partial_z\chi\|_{H^1(\Omega(v))}^{2\alpha} \\
& \le \|\partial_z\chi\|_{L_2(\Omega(v))}^{2(1-\alpha)}\,\|\sigma\partial_z\chi\|_{L_2(\Omega(v))}^{2\alpha} \\
& \quad + \|\partial_z\chi\|_{L_2(\Omega(v))}^{2(1-\alpha)} \left( \|\partial_x(\sigma\partial_z\chi)\|_{L_2(\Omega(v))}^2 + \|\partial_z(\sigma\partial_z\chi)\|_{L_2(\Omega(v))}^2 \right)^\alpha \\
& \le \frac{\sigma_{max}^{\alpha}}{\sigma_{min}^{1-\alpha}} \int_{\Omega(v)} \sigma |\partial_z \chi|^2 \,\rd (x,z) + \frac{1}{\epsilon} \int_{\Omega(v)} |\partial_z \chi|^2 \,\rd (x,z) \\
& \quad + \epsilon \|\partial_x(\sigma\partial_z \chi)\|_{L_2(\Omega(v))}^2 + \epsilon \|\partial_z(\sigma\partial_z \chi)\|_{L_2(\Omega(v))}^2 \\
& \le c \left( 1 + \frac{1}{\epsilon} \right) \int_{\Omega(v)} \sigma |\partial_z \chi|^2 \,\rd (x,z) + 2 \epsilon \int_{\Omega_1} \left( |\partial_x \sigma_1|^2 + |\partial_z\sigma_1|^2 \right) |\partial_z\chi|^2 \,\rd (x,z)\\
& \quad + 2 \epsilon \sigma_{max} \left( \int_{\Omega_1\cup \Omega_2(v)} \sigma |\partial_{x}\partial_{z}\chi|^2 \,\rd (x,z) + \int_{\Omega_1\cup \Omega_2(v)} \sigma |\partial_{z}^2\chi|^2 \,\rd (x,z) \right)\,.
\end{align*}
Choosing $\epsilon=1/(8 \sigma_{max})$ and using once more \eqref{sigmamin}, we end up with
\begin{equation}\label{uuuu}
\begin{split}
\|\partial_z\chi\|_{L_2(\Omega(v))}^{2(1-\alpha)}\,\|\sigma\partial_z\chi\|_{H^1(\Omega(v))}^{2\alpha} \le &\
c\int_{\Omega(v)} \sigma\vert\nabla\chi\vert^2\,\rd(x,z)\\
& +\frac{1}{4} \int_{\Omega_1\cup \Omega_2(v)} \sigma \left\{\left(\partial_{x}\partial_{z}\chi\right)^2 + \left(\partial_z^2\chi\right)^2\right\}\,\rd (x,z)\,.
\end{split}
\end{equation}
Taking \eqref{uu}-\eqref{uuuu} into account, we derive from \eqref{u} that
\begin{equation*}
\begin{split}
\int_{\Omega_1\cup \Omega_2(v)}\sigma&\left(\partial_{x}\partial_{z}\chi\right)^2\,\rd (x,z)+\int_{\Omega_1\cup \Omega_2(v)}\sigma\left(\partial_z^2\chi\right)^2\,\rd (x,z) \\
&\le  c \int_{\Omega_1\cup \Omega_2(v)}\vert\mathrm{div}\left(\sigma\nabla h_v\right)\vert^2\,\rd (x,z)+
c\int_{\Omega(v)} \sigma\vert\nabla\chi\vert^2\,\rd(x,z)\,.
\end{split}
\end{equation*}
We then use again the identity 
$$
\sigma\partial_x^2\chi=-\partial_x\sigma\partial_x\chi - \partial_z\sigma \partial_z\chi - \sigma\partial_z^2\chi - \mathrm{div}(\sigma\nabla h_{v})\quad\text{in }\ \Omega_1\cup \Omega_2(v)\,,
$$
stemming from \eqref{a2a} along with Lemma~\ref{C1} (recalling $\psi_v=\chi+h_v$) and \eqref{sigmamin} to derive
\begin{align}
\|\chi_1\|_{H^2(\Omega_1)}^2+\|\chi_2\|_{H^2(\Omega_2(v))}^2&\le c  \int_{\Omega_1\cup \Omega_2(v)}\vert\mathrm{div}\left(\sigma\nabla h_v\right)\vert^2\,\rd (x,z)\nonumber\\
&\qquad +\int_{\Omega(v)} \sigma\vert\nabla h_v\vert^2\,\rd (x,z) \,. \label{io}
\end{align}
Finally, since $h_{v,j}(x,z)=h_j(x,z,v(x))$ for $(x,z)\in \Omega(v)$ and  $j=1,2$, it follows from the assumed bound on $v$ and the continuous embedding of $H^2(D)$ in $C(\bar D)$ that
\begin{subequations}\label{pim}
\begin{align}
\|h_{v,1}\|_{H^2(\Omega_1)} & \le c \left( 1+\|v\|_{H^1(D)}^2+\|v\|_{H^2(D)} \right)\,\|h_1\|_{C^2(\bar D\times[-H-d,-H]\times[-H,c\kappa])} \nonumber \\
& \le c(\kappa) \label{pim1}
\end{align}
and 
\begin{align}
\|h_{v,2}\|_{H^2(\Omega_2(v))} & \le c(1+\|v\|_{H^1(D)}^2+\|v\|_{H^2(D)})\,\vert\Omega_2(v)\vert^{1/2}\, \|h_2\|_{C^2(\bar D\times[-H,c\kappa]\times[-H,c\kappa])} \nonumber\\
& \le c(\kappa)\,. \label{pim2}
\end{align}
\end{subequations}
Consequently, the right-hand side of \eqref{io} is bounded by $c(\kappa)$ and the estimate \eqref{est1} follows from \eqref{io} and Lemma~\ref{C1}. Arguing as in the proof of Lemma~\ref{L-1}, we also deduce from \eqref{bobbybrown4} and \eqref{pim} that 
\begin{equation*}
\|\sigma \partial_z h_v \|_{H^1(\Omega(v))} \le c(\kappa)\,,
\end{equation*}
while Lemma~\ref{L-1} and \eqref{est1} imply that 
\begin{equation*}
\|\sigma \partial_z \chi \|_{H^1(\Omega(v))} \le c(\kappa)\,.
\end{equation*} 
Recalling that $\psi_v=\chi+h_v$, these properties and \eqref{est1} readily give \eqref{est2}.
\end{proof}

\subsection{$\Gamma$-Convergence of the Dirichlet Energy}\label{GCDE}

We next aim at extending Proposition~\ref{P1}~{\bf (b)} to all $v\in \bar S$, such as the one depicted in Figure~\ref{Fig3}. 
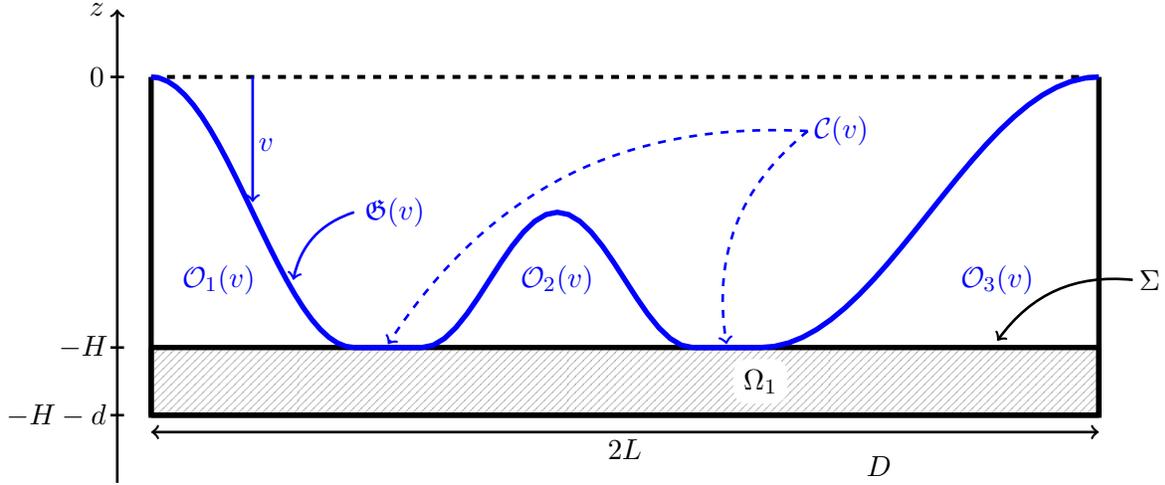
\begin{figure}
	\begin{tikzpicture}[scale=0.9]
	\draw[black, line width = 1.5pt, dashed] (-7,0)--(7,0);
	\draw[black, line width = 2pt] (-7,0)--(-7,-5);
	\draw[black, line width = 2pt] (7,-5)--(7,0);
	\draw[black, line width = 2pt] (-7,-5)--(7,-5);
	\draw[black, line width = 2pt] (-7,-4)--(7,-4);
	\draw[black, line width = 2pt, fill=gray, pattern = north east lines, fill opacity = 0.5] (-7,-4)--(-7,-5)--(7,-5)--(7,-4);
	\draw[blue, line width = 2pt] plot[domain=-7:-4] (\x,{-2-2*cos((pi*(\x+4)/3) r)});
	\draw[blue, line width = 2pt] (-4,-4)--(-3,-4);
	\draw[blue, line width = 2pt] plot[domain=-3:1] (\x,{-3+cos((pi*(\x+1)/2) r)});
	\draw[blue, line width = 2pt] (1,-4)--(2,-4);
	\draw[blue, line width = 2pt] plot[domain=2:7] (\x,{-2-2*cos((pi*(\x-2)/5) r)});
	\draw[blue, line width = 1pt, arrows=->] (-5.5,0)--(-5.5,-1.85);
	\node at (-5.3,-1) {${\color{blue} v}$};
	\node[draw,rectangle,white,fill=white, rounded corners=5pt] at (2,-4.5) {$\Omega_1$};
	\node at (2,-4.5) {$\Omega_1$};
	\node at (-6,-3) {${\color{blue} \mathcal{O}_1(v)}$};
	\node at (-1,-3) {${\color{blue} \mathcal{O}_2(v)}$};
	\node at (5.5,-3) {${\color{blue} \mathcal{O}_3(v)}$};
	\node at (3.75,-5.75) {$D$};
	\node at (7.75,-3) {$\Sigma$};
	\draw (7.5,-3) edge[->,bend right, line width = 1pt] (5.5,-3.9);
	\node at (-7.8,1) {$z$};
	\draw[black, line width = 1pt, arrows = ->] (-7.5,-6)--(-7.5,1);
	\node at (-8.4,-5) {$-H-d$};
	\draw[black, line width = 1pt] (-7.6,-5)--(-7.4,-5);
	\node at (-8,-4) {$-H$};
	\draw[black, line width = 1pt] (-7.6,-4)--(-7.4,-4);
	\node at (-7.8,0) {$0$};
	\draw[black, line width = 1pt] (-7.6,0)--(-7.4,0);
	\node at (0,-5.5) {$2L$};
	\draw[black, line width = 1pt, arrows = <->] (-7,-5.25)--(7,-5.25);
	\node at (3.2,-0.8) {${\color{blue} \mathcal{C}(v)}$};
	\draw (2.7,-0.8) edge[->,bend right,blue, dashed, line width = 1pt] (1.5,-3.95);
	\draw (2.7,-0.8) edge[->,bend right, blue, dashed, line width = 1pt] (-3.5,-3.95);
	\node at (-3.4,-2) {${\color{blue} \mathfrak{G}(v)}$};
	\draw (-4,-2) edge[->,bend right,blue, line width = 1pt] (-4.9,-3);
	\end{tikzpicture}
	\caption{Geometry of $\Omega(v)$ for a state $v\in \bar S$ with non-empty (and disconnected) coincidence set $\mathcal{C}(v)$.}\label{Fig3}
\end{figure}
For that purpose we show the $\Gamma$-convergence in $L_2$ of the functional $\mathcal{J}(v)$, defined in \eqref{sos}, with respect to $v$. More precisely,
fix $M>0$ and set 
$$
\Omega(M):=D\times (-H-d,2M)\,.
$$ 
For $v\in \bar S$ define
$$
G(v)[\theta]:=\left\{\begin{array}{ll} \dfrac{1}{2}\displaystyle\int_{ \Omega(v)} \sigma \vert\nabla (\theta+h_v)\vert^2\,\rd (x,z)\,, & \theta\in H_0^1(\Omega(v))\,,\\
\infty\,, & \theta\in L_2(\Omega(M))\setminus H_0^1(\Omega(v))\,.
\end{array}\right.
$$
Consider now $v\in \bar S$ and a sequence  $(v_n)_{n\ge 1}$ in $\bar S$ such that
\begin{equation}\label{o1}
v_n\rightarrow v \ \text{ in }\ H_0^1(D)\,,\qquad -H\le v, v_n\le M\,.
\end{equation}
Owing to the continuous embedding of $H_0^1(D)$ in $C(\bar D)$, a direct consequence of \eqref{o1} is that
\begin{equation}\label{o1b}
v_n\rightarrow v\quad\text{in }\ C(\bar D)\,.
\end{equation}

Let us first observe that, according to \eqref{bb} and \eqref{bobbybrown4}, both $h_{v_n}$ and $h_{v}$ belong to $H^1(\Omega(M))$. Moreover:

\begin{lemma}\label{L6}
Suppose \eqref{o1}. Then $h_{v_n} \rightarrow h_{v}$ in $H^1(\Omega(M))$ and
$$
\lim_{n\rightarrow\infty}\int_{\Omega(v_n)} \sigma \vert\nabla h_{v_n}\vert^2\,\rd (x,z)= \int_{\Omega(v)} \sigma \vert\nabla h_{v}\vert^2\,\rd (x,z)\,.
$$
\end{lemma}

\begin{proof}
Recall that $h_{v} (x,z)=h(x,z,v(x))$ for $(x,z)\in \Omega(M)$, so that
$$
\nabla h _{v}(x,z)=\big(\partial_x h(x,z,v(x)) + \partial_x v(x)\partial_wh(x,z,v(x)) , \partial_zh(x,z,v(x)) \big)\,,
$$
hence  $h_{v_n}, h_v\in H^1 (\Omega(M))$. Owing to \eqref{o1b} and the regularity of $h$ we obtain
$$
\lim_{n\rightarrow\infty}\sup_{(x,z)\in \overline{\Omega(M)}}\big\vert \left(\partial_xh,\partial_zh,\partial_x v_n \partial_wh\right)(x,z,v_n(x))-\left(\partial_xh,\partial_zh,\partial_x v \partial_wh\right)(x,z,v(x))\big\vert =0\,.
$$
Together with \eqref{o1}, this implies $h_{v_n}\rightarrow h_v$ in $H^1(\Omega(M))$. In particular, $\vert\nabla h_{v_n}\vert^2\rightarrow \vert\nabla h_v\vert^2$ in $L_1(\Omega(M))$. Since $\sigma$ is bounded and $\sigma{\bf 1}_{\Omega(v_n)}\rightarrow \sigma{\bf 1}_{\Omega(v)}$ pointwise, the last property stated in Lemma~\ref{L6} now follows from \cite[Proposition~2.61]{FoLe07}.
\end{proof}

Next, we show that the functional $G(v)$ is the $\Gamma$-limit of the sequence $(G(v_n))_{n\ge 1}$.

\begin{proposition}\label{P3}
Suppose \eqref{o1}. Then
$$
\Gamma-\lim_{n\rightarrow\infty} G(v_n)=G(v)\quad\text{in }\ L_2(\Omega(M))\,.
$$
\end{proposition}

\begin{proof}
\noindent\textbf{Step~1.} We begin with the asymptotic lower semicontinuity. Considering an arbitrary sequence $(\theta_n)_{n\ge 1}$ in $L_2(\Omega(M))$ and $\theta\in L_2(\Omega(M))$ such that
\begin{equation}\label{t0}
\theta_n\rightarrow \theta \ \text{ in }\ L_2(\Omega(M))\,,
\end{equation}
 we have to show that
\begin{equation}\label{G1}
G(v)[\theta]\le\liminf_{n\rightarrow\infty} G(v_n)[\theta_n]\,.
\end{equation}
We may assume that $\theta_n\in H_0^1(\Omega(v_n))$ for all $n\ge 1$ and that $(G(v_n)[\theta_n])_{n\ge 1}$ is bounded, since \eqref{G1} is clearly satisfied otherwise. In that case, if $\tilde\theta_n$ denotes the extension of $\theta_n$ by zero in $\Omega(M)\setminus\Omega(v_n)$, then it follows from \eqref{sigmamin} and Lemma~\ref{L6} that $(\tilde\theta_n)_{n\ge 1}$ is bounded in $H_0^1(\Omega(M))$ and thus
\begin{equation}\label{JJCale}
(\tilde\theta_n)_{n\ge 1}\ \text{ is weakly relatively compact in }\ H_0^1(\Omega(M))\,.
\end{equation}
Introducing $\tilde\theta:=\theta\mathbf{1}_{\Omega(v)}$ and noticing that 
\begin{equation*}
\begin{split}
\int_{\Omega(M)}\vert \tilde\theta_n-\tilde\theta\vert^2\,\rd (x,z)= &
\int_{\Omega(v_n)\cap \Omega(v)}\vert \theta_n-\theta\vert^2\,\rd (x,z) +\int_{\Omega(v_n)\cap(\Omega(M)\setminus\Omega(v))}\vert \theta_n\vert^2\,\rd (x,z)\\
& +\int_{(\Omega(M)\setminus\Omega(v_n))\cap\Omega(v)}\vert \theta\vert^2\,\rd (x,z)\\
\le &  \int_{\Omega(v_n)\cap \Omega(v)}\vert \theta_n-\theta\vert^2\,\rd (x,z) 
+2\int_{\Omega(v_n)\cap (\Omega(M)\setminus\Omega(v))}\vert \theta_n-\theta\vert^2\,\rd (x,z)\\
& +2\int_{\Omega(v_n)\cap(\Omega(M)\setminus\Omega(v))}\vert \theta\vert^2\,\rd (x,z)
+\int_{(\Omega(M)\setminus\Omega(v_n))\cap\Omega(v)}\vert \theta\vert^2\,\rd (x,z)\,,
\end{split}
\end{equation*}
we infer from \eqref{o1b}, \eqref{t0}, and Lebesgue's theorem that the right-hand side of the above inequality converges to zero as $n\rightarrow \infty$. Consequently, $(\tilde\theta_n)_{n\ge 1}$ converges to $\tilde\theta$ in $L_2(\Omega(M))$, which implies, together with \eqref{JJCale}, that $\tilde\theta\in H_0^1(\Omega(M))$ and
\begin{equation}\label{hh}
\tilde\theta_n \rightharpoonup \tilde\theta \ \text{ in }\ H^1(\Omega(M))\,,\qquad \tilde\theta_n \rightarrow \theta \ \text{ in }\ H^{3/4}(\Omega(v))\,.
\end{equation}
In particular, using Lemma~\ref{L6} and the continuity of the trace,
\begin{equation}\label{o2}
\tilde\theta_n+h_{v_n}\rightarrow \theta +h_{v}\quad\text{ in }\ L_2(\partial\Omega(v))\,.
\end{equation}
It remains to check that $\theta\in H_0^1(\Omega(v))$ for which we only have to show that $\theta$ vanishes (in the sense of traces) on the upper part $\mathfrak{G}(v)$ of the boundary $\partial\Omega(v)$, since $\theta=\tilde\theta$ vanishes on the other boundary parts of $\Omega(v)$. 
Since $\tilde\theta_n\in H_0^1(\Omega(v_n))$, it follows from H\"older's inequality that
\begin{equation*}
\begin{split}
&\left\vert h_{v_n,2}(x,v_n(x))-(\tilde\theta_n+h_{v_n,2})(x,v(x))\right\vert= \left\vert(\tilde\theta_n+h_{v_n,2})(x,v_n(x))-(\tilde\theta_n+h_{v_n,2})(x,v(x))\right\vert\\
&\qquad=\left\vert\int_{v(x)}^{v_n(x)}\partial_z (\tilde\theta_n+h_{v_n,2})(x,z)\,\rd z\right\vert\\
&\qquad\le \vert v_n(x)-v(x)\vert^{1/2}\left(\int_{-H}^M\vert\partial_z (\tilde\theta_n+ h_{v_n,2}) (x,z)\vert^2\,\rd z\right)^{1/2}
\end{split}
\end{equation*}
for almost every $x\in D$. Thus, by \eqref{sigmamin},
\begin{equation*}
\begin{split}
\int_D\vert h_{v_n,2}(x,v_n(x))-&(\tilde\theta_n+h_{v_n,2})(x,v(x))\vert^2\,\rd x\\
&\le\int_D \vert v_n(x)-v(x)\vert \int_{-H}^{M}\vert\partial_z (\tilde\theta_n+h_{v_n,2})(x,z)\vert^2\,\rd z\rd x\\
&\le \frac{\|v_n-v\|_{L_\infty(D)}}{\sigma_{min}} \int_{\Omega(M)}\sigma\vert\nabla (\tilde\theta_n+h_{v_n})(x,z)\vert^2\,\rd (x,z)\\
&= 2 \frac{\|v_n-v\|_{L_\infty(D)}}{\sigma_{min}}\, G(v_n)[\theta_n]\,.
\end{split}
\end{equation*}
Since $(G(v_n)[\theta_n])_{n\ge 1}$ is bounded and $v_n\rightarrow v$ in $C(\bar D)$ by \eqref{o1b}, the right-hand side of the above inequality converges to zero. Hence, due to \eqref{o2} and $h_{v_n,2}(\cdot,v_n)\rightarrow h_{v,2}(\cdot,v)$ in $C(\bar D)$, we conclude that indeed $\theta=0$ on $\mathfrak{G}(v)$. Therefore, $\theta\in H_0^1(\Omega(v))$. 
Now, by \eqref{hh} and Lemma~\ref{L6},
$$
\tilde\theta_n+h_{v_n} \rightharpoonup\tilde\theta+h_{v}\quad\text{in}\quad H_0^1(\Omega(M)) \,,
$$
so that
\begin{equation}\label{p1}
\int_{\Omega(M)}\sigma \vert \nabla(\tilde\theta+h_v)\vert^2\,\rd (x,z) \le
\liminf_{n\rightarrow\infty} \int_{\Omega(M)}\sigma \vert \nabla(\tilde\theta_n+h_{v_n})\vert^2\,\rd (x,z)\,.
\end{equation}
Since $\tilde\theta_n\in H_0^1(\Omega(v_n))$, 
$$
\int_{\Omega(M)\setminus\Omega(v_n)}\sigma \vert \nabla(\tilde\theta_n+h_{v_n})\vert^2\,\rd (x,z)=\int_{\Omega(M)\setminus\Omega(v_n)}\sigma \vert \nabla h_{v_n}\vert^2\,\rd (x,z)
$$
and we thus deduce from Lemma~\ref{L6}
\begin{equation}
\begin{split}\label{p2}
\lim_{n\rightarrow\infty } \int_{\Omega(M)\setminus\Omega(v_n)}\sigma \vert \nabla(\tilde\theta_n+h_{v_n})\vert^2\,\rd (x,z)&=
\int_{\Omega(M)\setminus\Omega(v)}\sigma \vert \nabla h_{v}\vert^2\,\rd (x,z)\\
&
= \int_{\Omega(M)\setminus\Omega(v)}\sigma \vert \nabla (\tilde\theta+h_{v})\vert^2\,\rd (x,z)\,,
\end{split}
\end{equation}
the last equality being due to $\tilde\theta\in H_0^1(\Omega(v))$.
Combining  \eqref{p1} and \eqref{p2} implies \eqref{G1}.

\smallskip

\noindent\textbf{Step~2.} We prove the existence of a recovery sequence. By definition of the functional $G(v)$ we only need to consider $\theta\in H_0^1(\Omega(v))$. Then $\theta\in H_0^1(\Omega(M))$ and $f:=-\Delta \theta\in H^{-1}(\Omega(M))$ can be considered also as an element of $H^{-1}(\Omega(v_n))$ by restriction. Let now $\theta_n\in H_0^1(\Omega(v_n))$ denote the unique weak solution to
$$
-\Delta \theta_n=f \quad\text{in }\ \Omega(v_n)\,,\qquad \theta_n=0 \quad\text{on }\ \partial\Omega(v_n)\,.
$$
Since the Hausdorff distance $d_H$ in $\Omega(M)$ (see \cite[Section~2.2.3]{HP05}) satisfies
$$
d_H(\Omega(v_n),\Omega(v)) \le \|v_n - v \|_{L_\infty(D)}\rightarrow 0
$$
by \eqref{o1b} and since $\overline{\Omega(M)}\setminus\Omega(v_n)$ has a single connected component for all $n\ge 1$ as $v_n>-H$, it follows from \cite[Theorem~4.1]{Sv93} and \cite[Theorem~3.2.5]{HP05} that 
$\theta_n\rightarrow \hat\theta$ in $H_0^1(\Omega(M))$,
where $\hat\theta\in H_0^1(\Omega(M))$ is the unique weak solution to
$$
-\Delta \hat\theta=f \quad\text{in }\ \Omega(M)\,,\qquad \hat\theta=0 \quad\text{on }\ \partial\Omega(M)\,.
$$
Clearly, $\hat\theta=\theta$ by uniqueness, so that $\theta_n\rightarrow \theta$ in $H_0^1(\Omega(M))$. Since $\theta_n\in H_0^1(\Omega(v_n))$ and $\theta\in H_0^1(\Omega(v))$, this convergence yields, with the help of  Lemma~\ref{L6},
\begin{equation*}
\begin{split}
\int_{\Omega(v)}\sigma \vert \nabla(\theta+h_v)\vert^2\,\rd (x,z)
&=
\int_{\Omega(v)}\sigma \left(\vert \nabla\theta\vert^2+2\nabla\theta\cdot \nabla h_v+\vert \nabla h_v\vert^2\right)\,\rd (x,z)\\
& = \lim_{n\rightarrow\infty} \int_{\Omega(M)}\sigma \left(\vert \nabla\theta_n\vert^2+2\nabla\theta_n\cdot \nabla h_{v_n}\right)\,\rd (x,z)\\
&\qquad +
\lim_{n\rightarrow\infty} \int_{\Omega(v_n)}\sigma \vert \nabla h_{v_n}\vert^2\,\rd (x,z)\\
&= \lim_{n\rightarrow\infty}\int_{\Omega(v_n)}\sigma \vert \nabla(\theta_n +h_{v_n})\vert^2\,\rd (x,z)\,;
\end{split}
\end{equation*}
that is,
$$
G(v)[\theta]=\lim_{n\rightarrow\infty} G(v_n)[\theta_n]\,.
$$
Combining the outcome of \textbf{Step~1} and \textbf{Step~2} implies the $\Gamma$-convergence of $(G(v_n))_{n\ge 1}$ to $G(v)$ in $L_2(\Omega(M))$.
\end{proof}

For the Dirichlet energy \eqref{DE}, which is given by
$$
\mathfrak{J}(v)=\mathcal{J}(v)[\psi_v]=\frac{1}{2}\int_{\Omega(v)} \sigma \vert\nabla \psi_v\vert^2\,\rd (x,z)\,,\quad v\in \bar S\,,
$$
with $\psi_v$ denoting the potential from Proposition~\ref{P1}, we now obtain:

\begin{corollary}\label{C3}
Suppose \eqref{o1}. Then
$$
\lim_{n\to\infty} \left\| (\psi_{v_n}-h_{v_n}) - (\psi_v - h_v) \right\|_{H_0^1(\Omega(M))} = 0
$$
and
$$
\lim_{n\to\infty} \mathfrak{J}(v_n) = \mathfrak{J}(v)\,.
$$
\end{corollary}

\begin{proof}
For $n\ge 1$, set 
$$
\chi_n:=\psi_{v_n}-h_{v_n}\in H_0^1(\Omega(v_n))\subset H_0^1(\Omega(M))\,,
$$ 
and recall that $\chi_n$ is a minimizer of $G(v_n)$ in $H_0^1(\Omega(v_n))$ by Proposition~\ref{P1}~\textbf{(a)}. Since  $(v_n)_{n\ge 1}$ is bounded in $H^1(D)$, it follows from \eqref{sigmamin}, Lemma~\ref{C1}, and Lemma~\ref{L6} that $(\chi_n)_{n\ge 1}$ is bounded in $H_0^1(\Omega(M))$. Hence, there are a subsequence $(n_j)_{j\ge 1}$ and $\chi\in H_0^1(\Omega(M))$ such that $\chi_{n_j}\rightarrow \chi$ in $L_2(\Omega(M))$ and $\chi_{n_j}\rightharpoonup \chi$ in $H_0^1(\Omega(M))$. By Proposition~\ref{P3} and the fundamental theorem of $\Gamma$-convergence, see \cite[Corollary~7.20]{DaM93}, $\chi$ is a minimizer of the functional $G(v)$ on $L_2(\Omega(M))$. Clearly, from the definition of $G(v)$ we see that $\chi+h_v\in \mathcal{A}(v)$ minimizes the functional $\mathcal{J}(v)$ on $\mathcal{A}(v)$, hence $\psi_v=\chi+h_v$ owing to  Proposition~\ref{P1}~\textbf{(a)}. The sequence $(\chi_n)_{n\ge 1}$ then has a unique cluster point in $L_2(\Omega(M))$ and is compact in that space and weakly compact in $H_0^1(\Omega(M))$. From this, we deduce that $\chi_{n}\rightarrow \chi$ in $L_2(\Omega(M))$ and $\chi_{n}\rightharpoonup \chi$ in $H_0^1(\Omega(M))$. Moreover, the fundamental theorem of $\Gamma$-convergence \cite[Corollary~7.20]{DaM93} also ensures 
$G(v_n)[\chi_n]\rightarrow G(v)[\chi]$; that is, $\mathfrak{J}(v_n)\rightarrow \mathfrak{J}(v)$  as $n\to\infty$. 

It remains to show the strong convergence of $(\chi_n)_{n\ge 1}$ in $H_0^1(\Omega(M))$. To this end, we infer from the convergence of $\left( \mathfrak{J}(v_n) \right)_{n\ge1}$ to $\mathfrak{J}(v)$ and Lemma~\ref{L6} that 
$$
\lim_{n\rightarrow\infty} \|\chi_n\|_{H_0^1(\Omega(M))}= \|\chi\|_{H_0^1(\Omega(M))}\,.
$$
Together with the already established weak convergence of $(\chi_n)_{n\ge 1}$ to $\chi$ in $H_0^1(\Omega(M))$, this gives the strong convergence.
\end{proof}

\subsection{$H^2$-Estimate for the Potential $\psi_v$}\label{HEPP}

Owing to the $H^2$-estimates on $\Omega_1\cup \Omega_2(v)$ derived in Proposition~\ref{P2}, we are able to improve Corollary~\ref{C3} to stronger topologies. 

\begin{proposition}\label{ACDC}
Consider $\kappa>0$, $v\in\bar S$, and a sequence $(v_n)_{n\ge 1}$ satisfying
\begin{equation}\label{Z1}
v_n\in S\cap W_\infty^2(D) \qquad \text{ with }\qquad  \|v_n\|_{H^2(D)}\le \kappa\,,
\end{equation}
and
\begin{equation}\label{Z2}
\lim_{n\to\infty} \| v_n - v\|_{H_0^1(D)} = 0\,.
\end{equation}
Then $\psi_v=(\psi_{v,1},\psi_{v,2})\in H^2(\Omega_1)\times H^2(\Omega_2(v))$ satisfies $\sigma\partial_z\psi_v\in H^1(\Omega(v))$ and
\begin{equation}\label{king}
\|\psi_v\|_{H^1(\Omega(v))}+\|\psi_{v,1}\|_{H^2(\Omega_1)}+\|\psi_{v,2}\|_{H^2(\Omega_2(v))} + \|\sigma \partial_z \psi_v \|_{H^1(\Omega(v))} \le c_{\ref{S3cst6}}(\kappa)
\end{equation}
 and
\begin{equation}\label{100}
\psi_{v_n,1}\rightharpoonup \psi_{v,1} \quad\text{in}\quad H^2(\Omega_1)\,,\qquad
\psi_{v_n,2}\rightharpoonup \psi_{v,2} \quad\text{in}\quad H^2(U)
\end{equation}
for any open set $U$ such that $\bar U$ is  a compact subset of $\Omega_2(v)$. 
\end{proposition}

\begin{proof}
It first follows from \eqref{Z1}, \eqref{Z2}, and the continuous embedding of $H^2(D)$ in $C(\bar{D})$ that there is $M>0$ such that \eqref{o1} is satisfied. Thus, by Corollary~\ref{C3},
\begin{equation}
 \psi_{v_n}-h_{v_n}\rightarrow  \psi_v-h_v\quad\text{in }\ H_0^1(\Omega(M))\,. \label{queen}
\end{equation}
Next, owing to \eqref{Z1}, we infer from \eqref{est2} that
\begin{equation}\label{z3}
\|\psi_{v_n,1}\|_{H^2(\Omega_1)}+\|\psi_{v_n,2}\|_{H^2(\Omega_2(v_n))} + \|\sigma \partial_z \psi_{v_n} \|_{H^1(\Omega(v_n))} \le c_{\ref{S3cst6}}(\kappa)\,,\quad n\ge 1\,.
\end{equation}
Now, \eqref{queen}, \eqref{z3}, and Lemma~\ref{L6} ensure that $\psi_{v,1}\in H^2(\Omega_1)$ and $\psi_{v_n,1}\rightharpoonup \psi_{v,1}$ in $H^2(\Omega_1)$. Similarly, for any open set $U$ such that $\bar U$ is a compact subset of $\Omega_2(v)$, we infer from \eqref{Z2} and the continuous embedding of $H_0^1(D)$ in $C(\bar D)$ that $U\subset \Omega_2(v_n)$ for $n$ large enough. Thus, \eqref{queen}, \eqref{z3}, and Lemma~\ref{L6} imply that $\psi_{v,2}\in H^2(U)$ and $\psi_{v_n,2}\rightharpoonup \psi_{v,2}$ in $H^2(U)$. In particular, the latter along with \eqref{z3} gives
$$
\int_U\vert\partial_x^j\partial_z^k \psi_{v,2}\vert^2\,\rd (x,z)\le\liminf_{n\rightarrow \infty} \int_U\vert\partial_x^j\partial_z^k \psi_{v_n,2}\vert^2\,\rd (x,z)\le c_{\ref{S3cst6}}(\kappa)\,,\quad j+k\le 2\,.
$$
We then use Fatou's lemma to conclude that $\partial_x^j\partial_z^k \psi_{v,2}$ belongs to $L_2(\Omega_2(v))$ for $j+k\le 2$; that is, $\psi_{v,2}\in H^2(\Omega_2(v))$. Finally, we deduce the estimate \eqref{king} from \eqref{100} and \eqref{z3} by a weak lower semicontinuity argument.
\end{proof}

Combining Corollary~\ref{C3} and Proposition~\ref{ACDC} allows us now to extend the validity of Proposition~\ref{P1}~\textbf{(b)} to all $v\in \bar S$. Recall that, for $v\in\bar{S}\setminus S$,  $\Omega_2(v)$ is a well-defined open, but disconnected, set in $\R^2$ with a non-Lipschitz boundary, see Figures~\ref{Fig2} and~\ref{Fig3} and Remark~\ref{ECwasHere}.

\begin{corollary}\label{P1b}
Let $v\in \bar S$ and let $\psi_v \in \mathcal{A}(v)$ be the  unique minimizer of $\mathcal{J}(v)$ on $\mathcal{A}(v)$ provided by Proposition~\ref{P1}. Then $\psi_v=(\psi_{v,1},\psi_{v,2})\in H^2(\Omega_1)\times H^2(\Omega_2(v))$ with $\sigma\partial_z\psi_v\in H^1(\Omega(v))$ satisfies the transmission problem
\begin{subequations}\label{a11}
\begin{align}
\mathrm{div}(\sigma\nabla\psi_v)&=0 \quad\text{in }\ \Omega(v)\,,\label{a11a}\\
\llbracket \psi_v \rrbracket = \llbracket \sigma \partial_z \psi_v \rrbracket &=0\quad\text{on }\ \Sigma(v)\,,\label{a11b}\\
\psi_v&=h_v\quad\text{on }\ \partial\Omega(v)\,.\label{a11c}
\end{align}
\end{subequations}
Moreover, 
\begin{equation}\label{king3}
\|\psi_v\|_{H^1(\Omega(v))}+\|\psi_{v,1}\|_{H^2(\Omega_1)}+\|\psi_{v,2}\|_{H^2(\Omega_2(v))}  + \|\sigma \partial_z \psi_v \|_{H^1(\Omega(v))}  \le c(\kappa)\,,
\end{equation}
 provided $\|v\|_{H^2(D)}\le \kappa$.
\end{corollary}

\begin{proof}
Let $v\in \bar S$ be fixed and $\kappa>0$ such that $\|v\|_{H^2(D)}\le \kappa/2$. We may choose a sequence $(v_n)_{n\ge 1}$ in $S\cap W_\infty^2(D)$ satisfying 
\begin{equation*}
v_n\rightarrow v \ \text{ in }\ H^2(D)\,,\qquad \sup_{n\ge 1}\,\|v_n\|_{H^2(D)}\le \kappa\,.
\end{equation*}
In particular, \eqref{Z1}-\eqref{Z2} are satisfied, so that Proposition~\ref{ACDC} implies that $(\psi_{v,1},\psi_{v,2})$ belongs to $H^2(\Omega_1)\times H^2(\Omega_2(v))$ and satisfies the estimate \eqref{king3}.

Regarding the transmission problem \eqref{a11}, recall first that $\psi_v$ satisfies \eqref{e1}. Since $v\in C(\bar D)$, we can write the open set $\{x\in D\,:\, v(x)>-H\}$  as a countable union of open intervals $((a_i,b_i))_{i\in I}$, see \cite[IX.Proposition~1.8]{AEIII}. 

Let $i\in I$ and set
$$
O_i(v):=\{(x,z)\in (a_i,b_i)\times\R\,:\, -H-d<z<v(x)\}\subset\Omega(v)
$$ and $\mathcal{O}_i(v):=\Omega_2(v)\cap O_i(v)$. It readily follows from  \eqref{e1} and the fact that $(\psi_{v,1},\psi_{v,2})$ belongs to $H^2(\Omega_1)\times H^2(\Omega_2(v))$ that $\mathrm{div}(\sigma\nabla\psi_v)=0$ in $\Omega_1$ and in each $\mathcal{O}_i(v)$, hence \eqref{a11a}. Moreover, for all $\theta\in \mathcal{D}(O_i(v))$ it follows from \eqref{e1} and Gau\ss' theorem that
\begin{equation*}
0=\int_{O_i(v)}\sigma\nabla\psi_v\cdot\nabla \theta\,\rd (x,z)  =\int_{a_i}^{b_i} \left(\llbracket\sigma\partial_z\psi_v\rrbracket\theta\right) (x,-H)\,\rd x\,,
\end{equation*}
hence $\llbracket\sigma\partial_z\psi_v\rrbracket (\cdot,-H)=0$ a.e. in $(a_i,b_i)$. Therefore, \eqref{a11b} holds, which in particular implies, together with the piecewise $H^2$-regularity of $\psi_v$, that $\sigma\partial_z\psi_v\in H^1(\Omega(v))$. Finally, since $\psi_v\in H^1(\Omega(v))$ we have $\llbracket \psi_v \rrbracket =0$ on $\Sigma(v)$,  while \eqref{a11c} is due to $\psi_v\in\mathcal{A}(v)$.
\end{proof}

Thanks to Corollary~\ref{P1b} we can extend the convergence established in Proposition~\ref{ACDC} to an arbitrary sequence $(v_n)_{n\ge 1}$ in $\bar{S}$. 

\begin{corollary}\label{Pam} 
Consider $\kappa>0$, $v\in\bar S$, and a sequence $(v_n)_{n\ge 1}$ satisfying
\begin{equation*}
v_n\in \bar{S} \qquad \text{ with }\qquad  \|v_n\|_{H^2(D)}\le \kappa\,,
\end{equation*}
and \eqref{Z2}. Then the convergence \eqref{100} holds true.
\end{corollary}

\begin{proof} The additional assumption $v_n\in S\cap W_\infty^2(D)$ is only used in the proof of Proposition~\ref{ACDC} to obtain the bound \eqref{z3}. Since such an estimate is now guaranteed by Corollary~\ref{P1b} as $\|v_n\|_{H^2(D)}\le \kappa$ for all $n\ge 1$, the proof of Corollary~\ref{Pam} follows the same lines as that of Proposition~\ref{ACDC}.
\end{proof}

The next step is to identify the limit of $\partial_z\psi_{v_n,2}(\cdot,v_n)$ as $n\rightarrow\infty$ within the framework of Proposition~\ref{ACDC}, which requires the following preparatory lemma.

\begin{lemma}\label{Eric} \refstepcounter{NumS3Const}\label{S3cst7}
Let $p\in [1,\infty)$, $\kappa>0$, and $v\in\bar S$ such that $\|v\|_{H^2(D)}\le \kappa$. Then there exists $c_{\ref{S3cst7}}(p,\kappa)>0$ such that
$$
\|\partial_z\psi_{v,2}(\cdot,v)\|_{L_p(D\setminus \mathcal{C}(v))}\le c_{\ref{S3cst7}}(p,\kappa)\,,
$$
the coincidence set $\mathcal{C}(v)$ of $v$ being defined in \eqref{CS}.
\end{lemma}

\begin{proof}
As in Corollary~\ref{P1b}, since $v\in C(\bar D)$, we can write the open set $\{x\in D\,:\, v(x)>-H\}$  as a countable union of open intervals $((a_i,b_i))_{i\in I}$ and define, for $i\in I$, 
$$
O_i(v):=\{(x,z)\in (a_i,b_i)\times \R\,:\, -H-d<z<v(x)\}\subset\Omega(v)
$$ and $\mathcal{O}_i(v):=\Omega_2(v)\cap O_i(v)$. As $D\setminus \mathcal{C}(v)$ has finite measure, we may assume that $p\in [3/2,\infty)$. Let $i\in I$. Since $\psi_{v,2}\in H^2(\mathcal{O}_i(v))$ by Corollary~\ref{P1b}, it follows from \eqref{a11b} and Young's inequality that, for $x\in (a_i,b_i)$,
\begin{equation*}
\begin{split}
\sigma_2^p\big\vert\partial_z\psi_{v,2}(x,v(x))\big\vert^p & \le \big\vert\sigma_2\partial_z\psi_{v,2}(x,-H)\big\vert^p +p\sigma_2^p\int_{-H}^{v(x)}\vert\partial_z\psi_{v,2}(x,z)\vert^{p-1} \vert\partial_z^2\psi_{v,2}(x,z)\vert\, \rd z\\
&\le \vert\sigma_1(x,-H)\partial_z\psi_{v,1}(x,-H)\vert^p+\frac{p}{2}\sigma_2^{2(p-1)} \int_{-H}^{v(x)}\vert\partial_z\psi_{v,2}(x,z)\vert^{2(p-1)} \, \rd z\\
&\qquad +\frac{p}{2}\sigma_2^{2} \int_{-H}^{v(x)}\vert\partial_z^2\psi_{v,2}(x,z)\vert^{2} \, \rd z\,.
\end{split}
\end{equation*}
Integrating with respect to $x\in (a_i,b_i)$, we find
\begin{equation*}
\begin{split}
\sigma_2^p\int_{a_i}^{b_i}\vert\partial_z\psi_{v,2}(x,v(x))\vert^p \,\rd x \le &
\int_{a_i}^{b_i}\vert\sigma_1(x,-H)\partial_z\psi_{v,1}(x,-H)\vert^p\,\rd x\\
&
+\frac{p}{2}\sigma_2^{2(p-1)} \int_{\mathcal{O}_i(v)}\vert\partial_z\psi_{v,2}(x,z)\vert^{2(p-1)} \, \rd (x,z)\\
& +\frac{p}{2}\sigma_2^{2} \int_{\mathcal{O}_i(v)}\vert\partial_z^2\psi_{v,2}(x,z)\vert^{2} \, \rd (x,z)\,.
\end{split}
\end{equation*}
Summing over all $i\in I$ we obtain
\begin{equation*}
\begin{split}
\sigma_2^p\int_{D\setminus\mathcal{C}(v)}\vert\partial_z\psi_{v,2}(x,v(x))\vert^p\,\rd x  \le &
\int_{D\setminus\mathcal{C}(v)}\vert \sigma_1(x,-H)\partial_z\psi_{v,1}(x,-H)\vert^p\,\rd x\\
&
+\frac{p}{2} \int_{\Omega_2(v)}\vert\sigma_2 \partial_z\psi_{v,2}(x,z)\vert^{2(p-1)} \, \rd (x,z)\\
& +\frac{p}{2} \sigma_2^2 \int_{\Omega_2(v)}\vert \partial_z^2\psi_{v,2}(x,z)\vert^{2} \, \rd (x,z)
\end{split}
\end{equation*}
and then infer from \eqref{sigmamin} and \eqref{king3} that
\begin{equation}\label{X1}
\begin{split}
\sigma_2^p\int_{D\setminus\mathcal{C}(v)}\vert\partial_z\psi_{v,2}(x,v(x))\vert^p\,\rd x \le &  \|\sigma_1\|_{L_\infty(D)} \int_{D}\vert\partial_z\psi_{v,1}(x,-H)\vert^p\,\rd x\\
& +\frac{p}{2} \|\sigma\partial_z\psi_{v}\|_{L_{2(p-1)}(\Omega(v))}^{2(p-1)} +\frac{p}{2}  \sigma_2^2 c(\kappa)^{2} \,.
\end{split}
\end{equation}
On the one hand,  $\partial_z\psi_{v,1}$ belongs to $H^1(\Omega_1)$ and the continuity of the trace from $H^1(\Omega_1)$ to $H^{1/2}(D\times\{-H\})$ combined with the continuous embedding of $H^{1/2}(D)$ in $L_p(D)$ and \eqref{king3} imply that 
\begin{equation}\label{X2}
\begin{split}
\|\partial_z\psi_{v,1}(\cdot,-H)\|_{L_p(D)}& \le c(p) \| \psi_{v,1}(\cdot,-H)\|_{H^{1/2}(D)}\\
& \le c(p) \| \psi_{v,1}\|_{H^{1}(\Omega_1)}\le c(p,\kappa)\,.
\end{split}
\end{equation}
On the other hand, $\sigma\partial_z\psi_v$ belongs to $H^1(\Omega(v))$ and it follows from Lemma~\ref{L4q} (with $q=2(p-1)$) and \eqref{king3} that
\begin{equation}\label{X3}
\begin{split}
\|\sigma \partial_z\psi_{v}\|_{L_{2(p-1)}(\Omega(v))}& \le c_{\ref{S3cst3}}(2p-1,\kappa) \|\sigma \partial_z\psi_{v}\|_{H^{1}(\Omega(v))} \le c(p,\kappa)\,.
\end{split}
\end{equation}
Combining \eqref{X1}-\eqref{X3} completes the proof.
\end{proof}

\subsection{Limit Behavior of the Trace of the Vertical Derivative}\label{LBTVD}

To derive the continuity property of the function $\mathfrak{g}$ stated in Theorem~\ref{Thm1b}, we shall next investigate the continuity with respect to $v\in \bar S$ of the potential's vertical derivative $\partial_z\psi_{v}$ traced along the graph $\mathfrak{G}(v)$. Recall that $\partial_z\psi_{v}$ along $\mathfrak{G}(v)$ consists of the two parts $$D\setminus \mathcal{C}(v)\rightarrow \R\,, \quad x\mapsto \partial_z\psi_{v,2}(x,v(x))$$ and $$\mathcal{C}(v)\rightarrow \R\,,\quad x\mapsto \partial_z\psi_{v,1}(x,-H)$$ and that the transition between $\partial_z\psi_{v,2}$ and $\partial_z\psi_{v,1}$ across the interface is prescribed by the transmission condition~\eqref{a11b} involving $\sigma$.

\begin{proposition}\label{ACDC1}
Consider $\kappa>0$, $v\in\bar S$, and a sequence $(v_n)_{n\ge 1}$ in $\bar S$ satisfying
\begin{equation}
\|v_n\|_{H^2(D)}\le \kappa \;\;\text{ and }\;\; \lim_{n\to\infty} \| v_n - v\|_{H^1(D)} =0\,. \label{plume}
\end{equation}
Then 
\begin{equation}\label{102}
\ell(v_n) \rightarrow  \ell(v) \quad\text{in}\quad L_p(D) 
\end{equation}
for $p\in [1,\infty)$, where $\ell(v)\in L_p(D)$ is given by
$$
\ell(v)(x):=\left\{\begin{array}{ll} \partial_z\psi_{v,2}(x,v(x))\,, & x\in D\setminus \mathcal{C}(v)\,,\\
\hphantom{x}\vspace{-3.5mm}\\
 \displaystyle{\frac{\sigma_1(x,-H)}{\sigma_2}\partial_z\psi_{v,1}(x,-H)}\,, & x\in \mathcal{C}(v)\,.
\end{array}\right.
$$
\end{proposition}

\begin{proof}
We first observe that the trace theorem, \eqref{sigmamin}, and the $H^2$-regularity of $\psi_{v,1}$ provided by Corollary~\ref{P1b} imply that $x\mapsto \partial_z \psi_{v,1}(x,-H)$ belongs to $L_p(D)$ for any $p\in [1,\infty)$. We deduce from this fact and Lemma~\ref{Eric} that $\ell(v)\in L_p(D)$ for $p\in [1,\infty)$. Also, it follows from \eqref{plume} that $\|v\|_{H^2(D)} \le \kappa$.

Let $\epsilon\in (0,H)$ be arbitrarily fixed. Due to $v_n\rightarrow v$ in $H_0^1(D)$ and the embedding of $H_0^1(D)$ in $C(\bar D)$, there is $n_\epsilon\ge 1$ such that
\begin{equation}\label{z7}
v(x)-\epsilon\le v_n(x)\le v(x)+ \epsilon\,,\qquad x\in \bar D\,,\quad n\ge n_\epsilon\,.
\end{equation}
Moreover, since $v\in C(\bar D)$ with $v(\pm L)=0$, the set 
\begin{equation*}
\Lambda(\epsilon):=\{x\in D\,:\,v(x)>-H+\epsilon\}
\end{equation*} 
is non-empty and open, and we can thus write it as a countable union of open intervals $(\Lambda_i(\epsilon))_{i\in I}$, see \cite[IX.Proposition~1.8]{AEIII}. For any fixed index $i\in I$  define the open set
$$
U_i(\epsilon):=\{(x,z)\in \Lambda_i(\epsilon) \times (-H,\infty)\,:\, z< v(x)- \epsilon\}
$$
and note that $U_i(\epsilon)$ has a Lipschitz boundary, with $U_i(\epsilon) \subset\Omega_2(v_n)$  for $n\ge n_\epsilon$ by \eqref{z7}. Thanks to \eqref{z3},  $(\psi_{v_n,2})_{n\ge 1}$ is relatively compact in $H^s(U_i(\epsilon))$ for any $s\in (3/2,2)$, so that \eqref{100} implies $\psi_{v_n,2}\rightarrow \psi_{v,2}$ in  $H^s(U_i(\epsilon))$. Using the continuity of the trace operator from $H^{s-1}(U_i(\epsilon))$ to $L_2(\partial U_i(\epsilon))$ and noticing the inclusion $\{(x,v(x)-\epsilon)\,:\, x\in \Lambda_i(\epsilon)\}\subset \partial U_i(\epsilon)$, we deduce 
 \begin{equation}\label{z9}
\partial_z\psi_{v_n,2}(\cdot,v- \epsilon)\rightarrow \partial_z\psi_{v,2}(\cdot,v-\epsilon) \quad\text{ in }\ L_2(\Lambda_i(\epsilon))\,. 
\end{equation}
Next, we put
$$
\Psi_n(x):=\partial_z\psi_{v,2}(x,v(x))-\partial_z\psi_{v_n,2}(x,v_n(x))\,,\qquad x\in \Lambda(\epsilon)\,,\quad n\ge 1\,,
$$
and observe that, for $i\in I$ and $n\ge n_\epsilon$,
\begin{equation*}
\begin{split}
\|\Psi_n\|_{L_2(\Lambda_i(\epsilon))}
& \le \|\partial_z\psi_{v,2}(\cdot,v-\epsilon)-\partial_z\psi_{v_n,2}(\cdot,v-\epsilon)\|_{L_2(\Lambda_i(\epsilon))}\\
&\qquad
+\left\| \int_{v-\epsilon}^v \partial_z^2\psi_{v,2}(\cdot,z)\,\rd z 
-\int_{v-\epsilon}^{v_n} \partial_z^2\psi_{v_n,2}(\cdot,z)\,\rd z\right\|_{L_2(\Lambda_i(\epsilon))}\\
& \le \|\partial_z\psi_{v,2}(\cdot,v-\epsilon)-\partial_z\psi_{v_n,2}(\cdot,v-\epsilon)\|_{L_2(\Lambda_i(\epsilon))}\\
&\qquad
+\left(\int_{\Lambda_i(\epsilon)} \epsilon \int_{v-\epsilon}^v \vert\partial_z^2\psi_{v,2} \vert^2\,\rd z\rd x\right)^{1/2} \\
&\qquad
+\left(\int_{\Lambda_i(\epsilon)} (v_n-v+\epsilon) \int_{v-\epsilon}^{v_n} \vert\partial_z^2\psi_{v_n,2} \vert^2\,\rd z\rd x\right)^{1/2}\,.
\end{split}
\end{equation*}
Using \eqref{z7} and the inequality $(a+b+c)^2\le 3(a^2+b^2+c^2)$, we infer that, for $n\ge n_\epsilon$,
\begin{equation*}
\begin{split}
\|\Psi_n\|_{L_2(\Lambda_i(\epsilon))}^2
 \le\ & 3\|\partial_z\psi_{v,2}(\cdot,v-\epsilon)-\partial_z\psi_{v_n,2}(\cdot,v-\epsilon)\|_{L_2(\Lambda_i(\epsilon))}^2\\
&+3\epsilon \int_{\Lambda_i(\epsilon)} \int_{v-\epsilon}^v \vert\partial_z^2\psi_{v,2} \vert^2\,\rd z\rd x 
+6\epsilon\int_{\Lambda_i(\epsilon)}  \int_{v-\epsilon}^{v_n} \vert\partial_z^2\psi_{v_n,2} \vert^2\,\rd z\rd x\\
\le\ &  3\|\partial_z\psi_{v,2}(\cdot,v-\epsilon)-\partial_z\psi_{v_n,2}(\cdot,v-\epsilon)\|_{L_2(\Lambda_i(\epsilon))}^2\\
&+3\epsilon \int_{\Lambda_i(\epsilon)} \int_{-H}^v \vert\partial_z^2\psi_{v,2} \vert^2\,\rd z\rd x 
+ 6\epsilon\int_{\Lambda_i(\epsilon)}  \int_{-H}^{v_n} \vert\partial_z^2\psi_{v_n,2} \vert^2\,\rd z\rd x\,.
\end{split}
\end{equation*}
Now, if $J$ is any finite subset of $I$, it follows from \eqref{king3} (applied to $v$ and $v_n$) and the previous inequality that, for $n\ge n_\epsilon$,
\begin{equation*}
\begin{split}
\sum_{i\in J}\|\Psi_n\|_{L_2(\Lambda_i(\epsilon))}^2
 \le\ & 3\sum_{i\in J}\|\partial_z\psi_{v,2}(\cdot,v-\epsilon)-\partial_z\psi_{v_n,2}(\cdot,v-\epsilon)\|_{L_2(\Lambda_i(\epsilon))}^2\\
&+3\epsilon \int_{\Lambda(\epsilon)} \int_{-H}^v \vert\partial_z^2\psi_{v,2} \vert^2\,\rd z\rd x 
+ 6\epsilon\int_{\Lambda(\epsilon)}  \int_{-H}^{v_n} \vert\partial_z^2\psi_{v_n,2} \vert^2\,\rd z\rd x\\
\le\ &  3 \sum_{i\in J}\|\partial_z\psi_{v,2}(\cdot,v-\epsilon)-\partial_z\psi_{v_n,2}(\cdot,v-\epsilon)\|_{L_2(\Lambda_i(\epsilon))}^2\\
&+3\epsilon \int_{\Omega_2(v)}  \vert\partial_z^2\psi_{v,2} \vert^2\,  \rd (x,z)
+6\epsilon\int_{\Omega_2(v_n)}  \vert\partial_z^2\psi_{v_n,2} \vert^2\, \rd (x,z)\\
\le\ &  3 \sum_{i\in J}\|\partial_z\psi_{v,2}(\cdot,v-\epsilon)-\partial_z\psi_{v_n,2}(\cdot,v-\epsilon)\|_{L_2(\Lambda_i(\epsilon))}^2+ c(\kappa) \epsilon\,.
\end{split}
\end{equation*}
Letting $n\rightarrow\infty$ and recalling \eqref{z9} and the finiteness of $J$ we get
\begin{equation}\label{z10a}
\limsup_{n\rightarrow\infty} \sum_{i\in J}\|\Psi_n\|_{L_2(\Lambda_i(\epsilon))}^2 \le c(\kappa)   \epsilon\,.
\end{equation}
Next, let $\delta\in (0,1)$. Since
$$
\sum_{i\in I} \vert \Lambda_i(\epsilon)\vert =\vert \Lambda(\epsilon)\vert\le \vert D\vert\,,
$$
there is a finite subset $J_\delta\subset I$ such that
\begin{equation}\label{X5}
\vert \Lambda(\epsilon)\vert-\delta\le \sum_{i\in J_\delta} \vert \Lambda_i(\epsilon)\vert \le \vert \Lambda(\epsilon)\vert\,.
\end{equation}
Hence, setting $\Lambda^\delta(\epsilon):=\bigcup_{i\in J_\delta}\Lambda_i(\epsilon)$, we get
\begin{equation*}
\begin{split}
\|\Psi_n\|_{L_2(\Lambda(\epsilon))}^2 \le\ & \sum_{i\in J_\delta}\| \Psi_n\|_{L_2(\Lambda_i(\epsilon))}^2 + 2 \int_{\Lambda(\epsilon)\setminus \Lambda^\delta(\epsilon)} \left(\vert\partial_z\psi_{v,2}(\cdot,v)\vert^2 +\vert\partial_z\psi_{v_n,2}(\cdot,v_n)\vert^2\right)\,\rd (x,z)\\
\le\ &  \sum_{i\in J_\delta}\| \Psi_n\|_{L_2(\Lambda_i(\epsilon))}^2
+2\vert \Lambda(\epsilon)\setminus \Lambda^\delta(\epsilon)\vert^{1/2}\, \|\partial_z\psi_{v,2}(\cdot,v)\|_{L_4(\Lambda(\epsilon)\setminus \Lambda^\delta(\epsilon))}^2 \\
&  + 2\vert \Lambda(\epsilon)\setminus \Lambda^\delta(\epsilon)\vert^{1/2}\, \|\partial_z\psi_{v_n,2}(\cdot,v_n)\|_{L_4(\Lambda(\epsilon)\setminus \Lambda^\delta(\epsilon))}^2 \,.
\end{split}
\end{equation*}
Since 
\begin{equation}
\Lambda(\epsilon)\subset D\setminus\mathcal{C}(v) \;\;\text{ and }\;\; \Lambda(\epsilon)\subset D\setminus\mathcal{C}(v_n)\ , \quad n\ge n_\epsilon\ , \label{plume2}
\end{equation}
by \eqref{z7}, we deduce from \eqref{X5}, Lemma~\ref{Eric} (with $p=4$), and the previous inequality that
$$
\|\Psi_n\|_{L_2(\Lambda(\epsilon))}^2\le \sum_{i\in J_\delta}\| \Psi_n\|_{L_2(\Lambda_i(\epsilon))}^2 + c(\kappa) \delta^{1/2}\,, \quad  n\ge n_\epsilon\,.
$$
Owing to the finiteness of $J_\delta$, we may let $n\rightarrow\infty$ in the previous inequality with the help of \eqref{z10a} and obtain that
$$
\limsup_{n\rightarrow\infty} \|\Psi_n\|_{L_2(\Lambda(\epsilon))}^2\le c(\kappa) \big(\epsilon+\delta^{1/2}\big)\,.
$$
Now, for $\epsilon_0\in (0,H)$ and $\delta\in (0,\epsilon_0)$, we have $\Lambda(\epsilon_0)\subset \Lambda(\delta)$, so that the previous estimate (with $\epsilon=\delta$) gives
$$
\limsup_{n\rightarrow\infty} \|\Psi_n\|_{L_2(\Lambda(\epsilon_0))}^2\le
\limsup_{n\rightarrow\infty} \|\Psi_n\|_{L_2(\Lambda(\delta))}^2\le c(\kappa) \big(\delta+\delta^{1/2}\big)\,.
$$
Letting $\delta\rightarrow 0$ and recalling the definition of $\Psi_n$ yield
\begin{equation}\label{z10}
\lim_{n\rightarrow\infty} \|\partial_z\psi_{v,2}(\cdot,v)-\partial_z\psi_{v_n,2}(\cdot,v_n)\|_{L_2(\Lambda(\epsilon_0))}^2=0\,,\quad \epsilon_0\in (0, H)\,.
\end{equation}
Next, let $\epsilon\in (0, H)$. The transmission condition \eqref{a11b} ensures
$$
\partial_z\psi_{v_n,2}(x,v_n(x))=\frac{\sigma_1(x,-H)}{\sigma_2}\partial_z\psi_{v_n,1}(x,-H)+\int_{-H}^{v_n(x)}\partial_z^2\psi_{v_n,2}(x,z)\,\rd z
$$
for $x\in D\setminus \mathcal{C}(v_n)$ and $n\ge 1$, from which we derive
\begin{equation*}
\begin{split}
\int_{D\setminus(\Lambda(\epsilon) \cup \mathcal{C}(v_n))}&\left\vert \partial_z\psi_{v_n,2}(x,v_n(x))-\frac{\sigma_1(x,-H)}{\sigma_2}\partial_z\psi_{v_n,1}(x,-H)\right\vert^2\,\rd x \\
&
\le   \int_{D\setminus (\Lambda(\epsilon) \cup \mathcal{C}(v_n))} (v_n(x)+H) \int_{-H}^{v_n(x)}\vert \partial_z^2\psi_{v_n,2}(x,z)\vert^2\,\rd z\rd x\,.
\end{split}
\end{equation*}
Thanks to \eqref{z7},
$$
v_n(x)+H\le v(x)+\epsilon+H \le 2\epsilon\,,\qquad x\in D\setminus \Lambda(\epsilon)\,,\quad n\ge n_\epsilon\,,
$$
so that, using \eqref{king3},
\begin{equation}\label{z11}
\int_{D\setminus (\Lambda(\epsilon) \cup \mathcal{C}(v_n))}\left\vert \partial_z\psi_{v_n,2}(x,v_n(x)) - \frac{\sigma_1(x,-H)}{\sigma_2}\partial_z\psi_{v_n,1}(x,-H)\right\vert^2\,\rd x 
\le   c(\kappa) \epsilon
\end{equation}
for $n\ge n_\epsilon$. Furthermore, \eqref{sigmamin}, Corollary~\ref{Pam}, and the continuity of the trace on $H^1(\Omega_1)$ ensure that
\begin{equation}\label{z12}
\lim_{n\rightarrow\infty}\int_{D}\left\vert \frac{\sigma_1(x,-H)}{\sigma_2}\partial_z\psi_{v_n,1}(x,-H) - \frac{\sigma_1(x,-H)}{\sigma_2}\partial_z\psi_{v,1}(x,-H)\right\vert^2\,\rd x =0\,.
\end{equation}
Now, recalling the definition of $\ell(v)$ on $\mathcal{C}(v)$,
\begin{align*}
&\int_{(D\setminus \Lambda(\epsilon))\cap \mathcal{C}(v_n)} \left\vert \frac{\sigma_1(x,-H)}{\sigma_2} \partial_z\psi_{v_n,1}(x,-H)-\ell(v)(x)\right\vert^2\, \rd x\\
&\le 2 \int_{(D\setminus \Lambda(\epsilon))\cap \mathcal{C}(v_n)} \left\vert \frac{\sigma_1(x,-H)}{\sigma_2} \big[ \partial_z\psi_{v_n,1}(x,-H) - \partial_z\psi_{v,1}(x,-H) \big]\right\vert^2\, \rd x\\
& \qquad + 2 \int_{(D\setminus \Lambda(\epsilon))\cap \mathcal{C}(v_n)} \left\vert \frac{\sigma_1(x,-H)}{\sigma_2} \partial_z\psi_{v,1}(x,-H)-\ell(v)(x)\right\vert^2\, \rd x\\
& \le 2 \int_{(D\setminus \Lambda(\epsilon))\cap \mathcal{C}(v_n)} \left\vert \frac{\sigma_1(x,-H)}{\sigma_2} \big[ \partial_z\psi_{v_n,1}(x,-H) - \partial_z\psi_{v,1}(x,-H) \big]\right\vert^2\, \rd x + 2 \mathcal{Q}_\epsilon\,,
\end{align*}
with
\begin{equation*}
\mathcal{Q}_\epsilon := \int_{D\setminus (\Lambda(\epsilon)\cup \mathcal{C}(v))} \left\vert \frac{\sigma_1(x,-H)}{\sigma_2} \partial_z\psi_{v,1}(x,-H)-\ell(v)(x)\right\vert^2\, \rd x\,.
\end{equation*}
Hence, owing to \eqref{z12},
\begin{equation}\label{morcheeba1}
\limsup_{n\rightarrow\infty} \int_{(D\setminus \Lambda(\epsilon))\cap \mathcal{C}(v_n)} \left\vert \frac{\sigma_1(x,-H)}{\sigma_2}\partial_z\psi_{v_n,1}(x,-H)-\ell(v)(x)\right\vert^2\, \rd x \le 2 \mathcal{Q}_\epsilon\,.
\end{equation}
In addition,
\begin{equation}\label{morcheeba2}
\int_{D\setminus (\Lambda(\epsilon)\cup \mathcal{C}(v_n))} \left\vert \frac{\sigma_1(x,-H)}{\sigma_2}\partial_z\psi_{v,1}(x,-H)-\ell(v)(x)\right\vert^2\, \rd x\le \mathcal{Q}_\epsilon\,.
\end{equation}
Using the disjoint union
\begin{equation*}
D = \Lambda(\epsilon) \cup \left[ (D\setminus\Lambda(\epsilon)) \cap \mathcal{C}(v_n)\right] \cup \left[ D \setminus (\Lambda(\epsilon) \cup \mathcal{C}(v_n))\right]
\end{equation*}
and recalling \eqref{z7} and the definition of $\ell$, we obtain that, for $n\ge n_\epsilon$,
\begin{align*}
& \left\| \ell(v_n)-\ell(v) \right\|_{L_2(D)}^2 \\
&\quad\le 3\int_{D\setminus (\Lambda(\epsilon)\cup \mathcal{C}(v_n))}  \left\vert \partial_z\psi_{v_n,2}(x,v_n(x))-\frac{\sigma_1(x,-H)}{\sigma_2}\partial_z\psi_{v_n,1}(x,-H)\right\vert^2\,\rd x\\
&\qquad + 3 \int_{D\setminus (\Lambda(\epsilon)\cup \mathcal{C}(v_n))}  \left\vert \frac{\sigma_1(x,-H)}{\sigma_2}\partial_z\psi_{v_n,1}(x,-H)-\frac{\sigma_1(x,-H)}{\sigma_2}\partial_z\psi_{v,1}(x,-H)\right\vert^2\,\rd x\\
&\qquad + 3 \int_{D\setminus (\Lambda(\epsilon)\cup \mathcal{C}(v_n))}  \left\vert \frac{\sigma_1(x,-H)}{\sigma_2}\partial_z\psi_{v,1}(x,-H)-\ell(v)(x)\right\vert^2\,\rd x\\
& \qquad +\int_{(D\setminus \Lambda(\epsilon))\cap \mathcal{C}(v_n)}  \left\vert \frac{\sigma_1(x,-H)}{\sigma_2} \partial_z\psi_{v_n,1}(x,-H) - \ell(v)(x)\right\vert^2\,\rd x\\
&\qquad + \int_{\Lambda(\epsilon)} \left\vert \partial_z\psi_{v_n,2}(x,v_n(x))-\partial_z\psi_{v,2}(x,v(x))\right\vert^2\,\rd x\,.
\end{align*}
It then follows from \eqref{z10}-\eqref{morcheeba2} that
\begin{equation}
\limsup_{n\rightarrow\infty} \left\| \ell(v_n)-\ell(v) \right\|_{L_2(D)}^2 \le c(\kappa) \epsilon + 5\mathcal{Q}_\epsilon\,. \label{plume3}
\end{equation}
At this point, we observe that 
\begin{equation*}
\lim_{\epsilon\rightarrow 0} \left|D\setminus (\Lambda(\epsilon)\cup \mathcal{C}(v))\right| = \lim_{\epsilon\to 0} \left|\{x\in D\,:\, -H<v(x)< -H+\epsilon\}\right| = 0\ ,
\end{equation*}
so that, since both $x\mapsto \sigma_1(x,-H) \partial_z \psi_{v,1}(x,-H)$ and $\ell(v)$ belong to $L_2(D)$, 
\begin{equation*}
\lim_{\epsilon\rightarrow 0} \mathcal{Q}_\epsilon = 0\,.
\end{equation*}
We then take the limit $\epsilon\rightarrow 0$ in \eqref{plume3} to conclude that
\begin{equation*}
\lim_{n\rightarrow\infty} \left\| \ell(v_n)-\ell(v) \right\|_{L_2(D)}^2 = 0\,.
\end{equation*}
Finally, $(\ell(v_n))_{n\ge 1}$ is bounded in $L_p(D)$ for any $p\in [1,\infty)$ by the trace theorem for $H^1(\Omega_1)$, \eqref{sigmamin}, the $H^2$-estimate \eqref{king3} on $(\psi_{v_n,1})_{n\ge 1}$, and Lemma~\ref{Eric}. Combining this bound with the previous convergence implies the convergence in $L_p(D)$ for $p\in [1,\infty)$ as stated in \eqref{102}.
\end{proof}

\begin{remark}
The proofs of Proposition~\ref{C3} and Proposition~\ref{ACDC1} greatly simplify when the sequence $(v_n)_{n\ge 1}$ decreases monotically to $v$. Indeed, in that case, $\Omega(v)\subset \Omega(v_n)$ for all $n\ge 1$ and, for instance, it is possible to use $\partial_z \psi_{v_n,2}(x,v(x))$ in the computations, since it is well-defined.
\end{remark}

\section{Shape Derivative of the Dirichlet Energy}\label{SDDE}

In order to compute the shape derivative of the Dirichlet energy defined in \eqref{DE}, the first step is to investigate the differentiability 
properties of $\psi_v$ with respect to $v\in S$. 

\begin{lemma}\label{P316}
Let $u\in S$ be fixed and define, for $v\in S$, the transformation 
$$
\Theta_v=(\Theta_{v,1},\Theta_{v,2}):\Omega(u)\rightarrow \Omega(v)
$$  
by 
\begin{subequations}\label{tr}
\begin{align}
&\Theta_{v,1}(x,z):=(x,z)\,,& &(x,z)\in\Omega_1\,,\\
& \Theta_{v,2}(x,z):=\left(x,z+\frac{v(x)-u(x)}{H+u(x)}(z+H)\right)\,,& &(x,z)\in\Omega_2(u) \,.
\end{align}
\end{subequations}
Then there exists a neighborhood $U$ of $u$ in $S$ such that
$$
U\rightarrow H_0^1(\Omega(u)),\quad v\mapsto \xi_v:=\chi_v\circ \Theta_v
$$
is continuously differentiable, where $\chi_v=\psi_v-h_v\in H_0^1(\Omega(v))$ and $S$ is endowed with the $H^2(D)$-topology.
\end{lemma}

\begin{proof}
We follow the lines of the proof of \cite[Theorem~5.3.2]{HP05}.  Recall that $\chi_v \in H_0^1(\Omega(v))$ satisfies the integral identity
\begin{equation}\label{w1}
\int_{\Omega(v)}\sigma\nabla\chi_v\cdot\nabla \theta\,\rd(\bar x,\bar z)
=-\int_{\Omega(v)}\sigma\nabla h_v\cdot\nabla \theta\,\rd(\bar x,\bar z)\,,\quad \theta\in H_0^1(\Omega(v))\,,
\end{equation}
which we next shall write as integrals over $\Omega(u)$. To this end, we first note that
$$
\xi_u=\chi_u\,,\qquad \nabla\xi_v=D\Theta_v^T\nabla\chi_v\circ \Theta_v\,,
$$
where $D\Theta_{v,1}=\mathrm{id}$ and
$$
 D\Theta_{v,2}(x,z)=\left(\begin{matrix}  1&0\\
\\
\displaystyle{(z+H)\partial_x\left(\frac{v-u}{H+u}\right)(x)} & \displaystyle{\frac{H+v(x)}{H+u(x)} }\end{matrix}\right)\,,\quad (x,z)\in \Omega_2(u)\,.
$$
For $\phi\in H_0^1(\Omega(u))$ we have
$$
\phi_v:=\phi\circ \Theta_v^{-1}\in H_0^1(\Omega(v))
$$
with  
$$
\nabla\phi_v=\big((D\Theta_v^T)^{-1}\nabla\phi\big)\circ \Theta_v^{-1}\,.
$$
Performing the change of variables $(\bar x,\bar z)=\Theta_v(x,z)$ in \eqref{w1} with $\theta=\phi_v$ gives
\begin{equation}\label{w6}
\begin{split}
\int_{\Omega(u)} J_v\,\sigma\,  &(D\Theta_v)^{-1} (D\Theta_v^T)^{-1}\nabla\xi_v\cdot\nabla\phi\,\rd (x,z)\\
&=  
-\int_{\Omega(u)} J_v\, (D\Theta_v)^{-1}(\sigma\nabla h_v)\circ \Theta_v\cdot\nabla\phi\,\rd (x,z)\,,
\end{split}
\end{equation}
where we used $\sigma\circ \Theta_v=\sigma$ due to $\Theta_{v,1}\equiv \mathrm{id}_{\Omega_1}$ and $\sigma_2=const$, and where $J_v:=|\mathrm{det}(D\Theta_v)|$ is
\begin{equation}\label{Jv}
J_{v,1}=1\,,\qquad  J_{v,2}=\frac{H+v}{H+u}\,.
\end{equation}
Introducing the notations
$$
A(v):=J_v\,\sigma\,  (D\Theta_v)^{-1} (D\Theta_v^T)^{-1}
$$
and 
$$
B(v):=\mathrm{div}\big(J_v\, (D\Theta_v)^{-1}(\sigma\nabla h_v)\circ \Theta_v\big)\,,
$$
we define the function 
$$
F: S\times H_0^1(\Omega(u))\rightarrow H^{-1}(\Omega(u))\,,\quad (v,\xi)\mapsto -\mathrm{div}\big(A(v)\nabla\xi\big)-B(v)
$$
and observe that \eqref{w6} is equivalent to 
\begin{equation}\label{w7}
F(v,\xi_v)=0\,,\quad v\in S\,.
\end{equation}
We then shall use the implicit function theorem to derive that $\xi_v$ depends smoothly on $v$.
For that purpose, let us first show that $F$ is Fr\'echet differentiable in $S\times H_0^1(\Omega(u))$. Indeed, define $P,Q\in C(\bar D\times\R,\R^3)$ by
$$
P(x,z):=\left(x,z-\frac{u(x)}{H+u(x)}(z+H),0\right)\,,\quad Q(x,z):=\left(0,\frac{H+z}{H+u(x)},1\right)\,,
$$ 
for $(x,z)\in \bar D\times\R$,
and note that
$$
\nabla h_{v,2} \circ \Theta_{v,2}=\left(\begin{matrix}  1&0&\partial_x v\\ 0&1&0 \end{matrix}\right) (\partial_x h_2,\partial_z h_2,\partial_w h_2)\circ (P+vQ)\,.
$$
Since $h_2$ is $C^2$-smooth, we clearly have
$$
\big[v\mapsto (\partial_x h_2,\partial_z h_2,\partial_w h_2)\circ (P+vQ)\big] \in C^1\big(S,L_2(\Omega_2(u),\R^3)\big)\,,
$$
so that, thanks to the continuous embedding of $H^2(D)$ in $W_\infty^1(D)$, we readily obtain that
$$
\big[v\mapsto\nabla h_{v,2} \circ \Theta_{v,2}\big]\in C^1\big(S,L_2(\Omega_2(u),\R^2)\big)\,.
$$
Since $\Theta_{v,1}\equiv \mathrm{id}$, a similar argument ensures that
$$
\big[v\mapsto\nabla h_{v,1} \circ \Theta_{v,1}\big]\in C^1\big(S,L_2(\Omega_1,\R^2)\big)\,,
$$
and therefore
$$
\big[v\mapsto\nabla h_{v} \circ \Theta_{v}\big]\in C^1\big(S,L_2(\Omega(u),\R^2)\big)\,.
$$
Moreover, $v\mapsto J_v$ and $v\mapsto (D\Theta_v)^{-1}$ are continuously differentiable from  $S$ to $L_\infty(\Omega(u))$ and $L_\infty(\Omega(u),\R^{2\times 2})$, respectively, and we conclude that
$$
v\mapsto J_v\, (D\Theta_v)^{-1}(\sigma\nabla h_v)\circ \Theta_v
$$
is continuously differentiable from $S$ to $L_2(\Omega(u),\R^{2})$, hence $B\in C^1(S,H^{-1}(\Omega(u)))$. The $C^1$-smoothness of $(v,\xi)\mapsto \mathrm{div}(A(v)\nabla\xi)$ is proved as in \cite[Theorem 5.3.2]{HP05} and we have indeed established 
$$
F\in C^1\big(S\times H_0^1(\Omega(u)), H^{-1}(\Omega(u))\big)\,.
$$
By  the Lax-Milgram theorem, the mapping $\zeta\mapsto\partial_\xi F(u,\chi_u)\zeta=-\mathrm{div}(\sigma\nabla\zeta)$ is bijective from $H_0^1(\Omega(u))$ to $H^{-1}(\Omega(u))$ and thus an isomorphism due to the open mapping theorem. Consequently, the implicit function theorem ensures the existence of a neighborhood $U$ of $u$ in $S$ and a function $\Xi\in C^1(U,H_0^1(\Omega(u))$ such that $\Xi(u)=\xi_u$ and $F(v,\Xi(v))=0$ for $v\in U$. By Corollary~\ref{C3}, $\xi_v\in \Xi(U)$ for $\|v-u\|_{H^2(D)}$ sufficiently small and consequently $\xi_v=\Xi(v)$ for $v\in U$ by the uniqueness provided by the implicit function theorem. 
\end{proof}

As a consequence of Lemma~\ref{P316}, we are now in a position to investigate differentiability properties of the Dirichlet energy
$$
\mathfrak{J}(u) =\frac{1}{2}\int_{\Omega(u)} \sigma \vert\nabla \psi_u\vert^2\,\rd (x,z)
$$
with respect to $u$. We begin with the case $u\in S$. At such functions, the  Dirichlet energy $\mathfrak{J}$ is Fr\'echet differentiable as shown next.

\begin{proposition}\label{C15}
Let $S$ be endowed with the $H^2(D)$-topology. Then the Dirichlet energy $\mathfrak{J}:S\rightarrow\R$ is continuously Fr\'echet differentiable with
\begin{equation*}
\begin{split}
\partial_u \mathfrak{J}(u)[\vartheta]
=& -\frac{1}{2}\int_D\sigma_2 (1+\vert\partial_x u(x)\vert^2)\,\big[\partial_z\psi_{u,2}-(\partial_z h_2)_u-(\partial_w h_2)_u\big]^2(x, u(x))\,\vartheta(x)\,\rd x\\
& +\frac{1}{2}\int_D\sigma_2 \,\big[\big((\partial_x h_2)_u\big)^2+ \big((\partial_z h_2)_u+(\partial_w h_2)_u\big)^2\big](x, u(x))\,\vartheta(x)\,\rd x\\
& -\int_D\big[\sigma_1(\partial_w h_1)_u\,\partial_z\psi_{u,1}\big](x,-H-d)\,\vartheta(x)\,\rd x
\end{split}
\end{equation*}
for $u\in S$ and $\vartheta\in H^2(D)\cap H_0^1(D)$.
\end{proposition}

\begin{proof}
We use the notation from Lemma~\ref{P316}. Let $u\in S$ be fixed and recall that, with the transformation
$
\Theta_v:\Omega(u)\rightarrow \Omega(v)
$
as in \eqref{tr}, the mapping  $v\mapsto \xi_v=\chi_v\circ \Theta_v$  is differentiable with respect to $v$ in a neighborhood $U$ of $u$ in $S$ and takes values in  $H_0^1(\Omega(u))$. Now, using $\psi_v=\chi_v+h_v$ and the change of variable $(\bar x,\bar z)=\Theta_v(x,z)$ in the integral defining $\mathfrak{J}(v)$, we have
\begin{equation*}
\begin{split}
\mathfrak{J}(v)&=\frac{1}{2}\int_{\Omega(v)} \sigma \vert\nabla \psi_v\vert^2\,\rd (\bar x,\bar z)\\
&=\frac{1}{2}\int_{\Omega(u)} \sigma \big\vert(D\Theta_v^T)^{-1}\nabla\xi_v+\nabla h_v\circ\Theta_v \big\vert^2\, J_v\,\rd (x,z)
\end{split}
\end{equation*}
for $v\in U$.
Therefore, introducing
$$
j(v):=(D\Theta_v^T)^{-1}\nabla\xi_v+\nabla h_v\circ\Theta_v
$$
and recalling that $h_i$, $i=1,2$, is $C^2$ in all its arguments, we deduce that the Fr\'echet derivative of $\mathfrak{J}$ at $u$ applied to $\vartheta\in H^2(D)\cap H_0^1(D)$ is
\begin{equation*}
\begin{split}
\partial_u\mathfrak{J}(u)[\vartheta]=\partial_v\mathfrak{J}(v)[\vartheta]\big\vert_{v=u}=& \int_{\Omega(u)} \sigma  j(u)\cdot (\partial_v j(v))[\vartheta]\big\vert_{v=u}\, J_u\,\rd (x,z)\\
&  +\frac{1}{2}\int_{\Omega(u)} \sigma \vert j(u)\vert^2\, (\partial_v J_v)[\vartheta]\big\vert_{v=u}\,\rd (x,z)\,.
\end{split}
\end{equation*}
Taking the identity $j(u)=\nabla\psi_u$ into account, we infer from \eqref{Jv} that
\begin{equation}\label{D1}
\begin{split}
\partial_u\mathfrak{J}(u)[\vartheta]=& \int_{\Omega(u)} \sigma  \nabla\psi_u\cdot \big(\partial_v j(v)[\vartheta]\big\vert_{v=u}\big)\,\rd (x,z)\\
& +\frac{1}{2}\int_{\Omega_2(u)} \sigma_2 \vert \nabla\psi_{u,2}\vert^2\, \frac{\vartheta}{H+u}\,\rd (x,z)\,.
\end{split}
\end{equation}
We next use that $\Theta_u$ is the identity on $\Omega(u)$ and that $\xi_u=\chi_u$ to compute from the definition of $j(v)$ that
\begin{equation}\label{D1a}
\begin{split}
\partial_v j(v)[\vartheta]\big\vert_{v=u}= &-\partial_v (D\Theta_v^T)[\vartheta]\big\vert_{v=u} \nabla \chi_u + \partial_v (\nabla\xi_v)[\vartheta]\big\vert_{v=u}\\
& +\partial_v (\nabla h_v\circ \Theta_v)[\vartheta]\big\vert_{v=u}\,.
\end{split}
\end{equation}
On the one hand, since $\Theta_{v,1}$ is independent of $v$ and $\xi_{v,1}=\chi_{v,1}$, we readily obtain in $\Omega_1$ that
\begin{equation}\label{D2}
\begin{split}
\partial_v j(v)[\vartheta]\big\vert_{v=u}= \nabla \big(\partial_v\chi_v[\vartheta]\big\vert_{v=u}\big)+\nabla\big((\partial_wh)_u\vartheta\big)\quad \text{ in }\ \Omega_1\,,
\end{split}
\end{equation}
where $(\partial_wh)_u=\partial_w h(\cdot,\cdot,u)$. On the other hand, in $\Omega_2(u)$ we have
\begin{equation}\label{D3}
-\partial_v (D\Theta_v^T)[\vartheta]\big\vert_{v=u} \nabla \chi_u=-\partial_z\chi_{u}\nabla\left(\frac{\vartheta(z+H)}{H+u}\right) \quad \text{ in }\ \Omega_2(u)
\end{equation}
and 
\begin{equation}\label{D4}
\partial_v (\nabla\xi_v)[\vartheta]\big\vert_{v=u}=\nabla \big(\partial_v \xi_v[\vartheta]\big\vert_{v=u}\big)\quad \text{ in }\ \Omega_2(u)\,.
\end{equation}
Moreover,
\begin{equation}\label{D5}
\partial_v (\nabla h_v\circ \Theta_v)[\vartheta]\big\vert_{v=u}=\nabla\big((\partial_wh)_u\vartheta\big)+\frac{\vartheta(z+H)}{H+u} \nabla\big((\partial_zh)_u\big)\quad \text{ in }\ \Omega_2(u)\,.
\end{equation}
Consequently, gathering \eqref{D1}-\eqref{D5}, we derive
\begin{equation}\label{D6}
\begin{split}
\partial_u\mathfrak{J}(u)[\vartheta]=& \int_{\Omega(u)} \sigma\,  \nabla\psi_u\cdot \nabla\Big(  \partial_v \xi_v[\vartheta]\big\vert_{v=u} + (\partial_wh)_u\vartheta\Big)\,\rd (x,z)\\
&  -\int_{\Omega_2(u)} \sigma_2\, \partial_z\chi_{u,2}\,  \nabla\psi_{u,2}\cdot \nabla \left(\frac{\vartheta(z+H)}{H+u}\right)\,\rd (x,z)\\
&  +\int_{\Omega_2(u)} \sigma_2\,  \frac{\vartheta(z+H)}{H+u}\, \nabla\psi_{u,2}\cdot \nabla \big((\partial_zh_2)_u\big) \,\rd (x,z)\\
&  +\frac{1}{2}\int_{\Omega_2(u)} \sigma_2\,  \vert \nabla\psi_{u,2}\vert^2 \,  \frac{\vartheta}{H+u} \,\rd (x,z)\,,
\end{split}
\end{equation}
and it remains to simplify the four integrals above. As for the first one we use \eqref{a11a}, \eqref{a11b}, and Gau\ss ' theorem to get
\begin{equation*}
\begin{split}
\int_{\Omega(u)} \sigma\,  \nabla&\psi_u\cdot \nabla\Big(  \partial_v \xi_v[\vartheta]\big\vert_{v=u} + (\partial_wh)_u\vartheta\Big)\,\rd (x,z) \\
=& \int_{\partial\Omega(u)} \Big(  \partial_v \xi_v[\vartheta]\big\vert_{v=u} + (\partial_wh)_u\vartheta\Big)\sigma\nabla\psi_u\cdot  {\bf n}_{\partial\Omega(u)}\,\rd S\\
&+\int_\Sigma\llbracket \partial_v \xi_v[\vartheta]\big\vert_{v=u} + (\partial_wh)_u\vartheta\rrbracket\, \sigma_1\,\partial_z\psi_{u,1}\,\rd S\,.
\end{split}
\end{equation*}
First note that the integral on $\Sigma$ vanishes. Indeed, since $\xi_{v,i}(x,-H)=\chi_{v,i}(x,-H)$ for $x\in D$ and $i=1,2$, we have $\llbracket\xi_v\rrbracket=0$ on $\Sigma$ by \eqref{bobbybrown4} and \eqref{a11b}, hence $\llbracket \partial_v\xi_v[\vartheta]\rrbracket=0$ on $\Sigma$. Similarly $\llbracket (\partial_w h)_u\rrbracket=0$ on $\Sigma$ owing to \eqref{bobbybrown2}.
Next, since $\partial_v \xi_v[\vartheta]\big\vert_{v=u}$ belongs to $H_0^1(\Omega(u))$ according to Proposition~\ref{P316}, it vanishes on the boundary $\partial\Omega(u)$. Moreover, since $\vartheta\in H_0^1(D)$, the term $(\partial_wh)_u\vartheta$ vanishes on the lateral parts of $\partial\Omega(u)$,
hence
\begin{equation*}
\begin{split}
\int_{\Omega(u)} &\sigma\,  \nabla\psi_u\cdot \nabla\Big(  \partial_v \xi_v[\vartheta]\big\vert_{v=u} + (\partial_wh)_u\vartheta\Big)\,\rd (x,z)
\\
=& \int_D\sigma_2 \partial_w h_2(x,u(x),u(x))\,\vartheta(x)\,\big(-\partial_x u(x)\partial_x\psi_{u,2}(x,u(x))+\partial_z\psi_{u,2}(x, u(x))\big)\,\rd x\\
& -\int_D\sigma_1(x,-H-d) \partial_w h_1(x,-H-d,u(x))\,\vartheta(x)\,\partial_z\psi_{u,1}(x, -H-d)\big)\,\rd x
\,.
\end{split}
\end{equation*}
Moreover, since
\begin{equation}\label{lop1}
(\partial_zh)_u=\partial_z\psi_u(\cdot,u)-\partial_z\chi_u(\cdot,u) \,,
\end{equation}
 we can write the second and the third integral in \eqref{D6} as
\begin{equation}\label{D10}
\begin{split}
-\int_{\Omega_2(u)} \sigma_2\, & \nabla\psi_{u,2}\cdot \left[\partial_z\chi_{u,2}\,\nabla \left(\frac{\vartheta(z+H)}{H+u}\right)  -  \frac{\vartheta(z+H)}{H+u}\, \nabla \big((\partial_zh_2)_u\big)\right] \,\rd (x,z)\\
 =&-\int_{\Omega_2(u)} \sigma_2\,  \nabla\psi_{u,2}\cdot \nabla\left(\partial_z\chi_{u,2}\,\frac{\vartheta(z+H)}{H+u}\right)\,\rd (x,z)\\
& + \int_{\Omega_2(u)}\sigma_2\nabla\psi_{u,2}\cdot \nabla(\partial_z\psi_{u,2})\,\frac{\vartheta(z+H)}{H+u}\,\rd (x,z)\,.
\end{split}
\end{equation}
For the first integral on the right-hand side of \eqref{D10} we use  \eqref{a11a} and Gauss' theorem 
and readily obtain, noticing that $(x,z)\mapsto (z+H)\vartheta(x)$ vanishes on all parts of the boundary $\partial\Omega_2(u)$ except on $\mathfrak{G}(u)$ and using \eqref{lop1}, that
\begin{equation}\label{D11}
\begin{split}
-\int_{\Omega_2(u)} \sigma_2\,  &\nabla\psi_{u,2}\cdot \nabla\left(\partial_z\chi_{u,2}\,\frac{\vartheta(z+H)}{H+u}\right)\,\rd (x,z) \\
& =-\int_D\sigma_2\vartheta(x)\,\big(\partial_z\psi_{u,2}(x, u(x))-(\partial_z  h)_u(x,u(x))\big)\,\\
&\qquad\qquad\qquad \times \big( -\partial_x u(x)\partial_x\psi_{u,2}(x, u(x))  +\partial_z\psi_{u,2}(x, u(x))\big)\,\rd x\,.
\end{split}
\end{equation}
The second integral on the right-hand side of \eqref{D10} is written in the alternative form
\begin{equation*}
\begin{split}
 \int_{\Omega_2(u)}&\sigma_2\nabla\psi_{u,2}\cdot \nabla(\partial_z\psi_{u,2})\,\frac{\vartheta(z+H)}{H+u}\,\rd (x,z) =  \frac{1}{2}\int_{\Omega_2(u)} \sigma_2\, \partial_z\vert\nabla\psi_{u,2}\vert^2\,\frac{\vartheta(z+H)}{H+u}\rd (x,z)\,.
\end{split}
\end{equation*}
We then integrate with respect to $z$ to obtain
\begin{equation}\label{D12}
\begin{split}
  \int_{\Omega_2(u)}&\sigma_2\nabla\psi_{u,2}\cdot \nabla(\partial_z\psi_{u,2})\,\frac{\vartheta(z+H)}{H+u}\,\rd (x,z)\\
&=  -\frac{1}{2}\int_{\Omega_2(u)} \sigma_2\,\vert\nabla\psi_{u,2}\vert^2\frac{\vartheta}{H+u}\rd (x,z)+ \frac{1}{2}\int_{D}\sigma_2 \vert\nabla\psi_{u,2}(x,u(x))\vert^2\,\vartheta(x)\,\rd x\,.
\end{split}
\end{equation}
Consequently, substituting  \eqref{D10}-\eqref{D12} in \eqref{D6}, we conclude that 
\begin{equation*}
\begin{split}
\partial_u \mathfrak{J}(u)[\vartheta]
=& -\int_D\sigma_2 \big[\partial_z\psi_{u,2}-(\partial_z h_2)_u-(\partial_w h_2)_u\big](x, u(x))\\
&\qquad\qquad \times\big[-\partial_x u\partial_x\psi_{u,2}+\partial_z\psi_{u,2}\big](x, u(x))\,\vartheta(x)\,\rd x\\
&+\frac{1}{2}\int_D\sigma_2\,\big\vert\nabla\psi_{u,2}(x,u(x))\big\vert^2\,\vartheta(x)\,\rd x\\
& -\int_D\big[\sigma_1(\partial_w h_1)_u\,\partial_z\psi_{u,1}\big](x,-H-d)\,\vartheta(x)\,\rd x\,.
\end{split}
\end{equation*}
It remains to rewrite the first two integrals on $\mathfrak{G}(u)$. For that purpose, it follows from \eqref{a11c} that
$$
\psi_{u,2}(x,u(x))=h_{u,2}(x,u(x))=h_{2}(x,u(x),u(x))\,,\quad x\in D\,,
$$
from which we deduce that
\begin{equation*}
\begin{split}
\partial_x\psi_{u,2}(x,u(x)) =& - \partial_x u(x) \big[\partial_z\psi_{u,2}-(\partial_z h_2)_u-(\partial_w h_2)_u\big](x, u(x))\\
&+ (\partial_x h_{2})_u (x, u(x))\,.
\end{split}
\end{equation*}
Using the above identity, it is easy to check that
\begin{equation*}
\begin{split}
\frac{1}{2}&\big\vert\nabla\psi_{u,2}(x,u(x))\big\vert^2-\big[\partial_z\psi_{u,2}-(\partial_z h_2)_u-(\partial_w h_2)_u\big]\big[-\partial_x u\partial_x\psi_{u,2}+\partial_z\psi_{u,2}\big](x, u(x))\\
& = -\frac{1}{2}(1+\vert\partial_x u(x)\vert^2)\,\big[\partial_z\psi_{u,2}-(\partial_z h_2)_u-(\partial_w h_2)_u\big]^2(x, u(x))\\
&\quad +\frac{1}{2}\big[\big((\partial_x h_2)_u\big)^2+ \big((\partial_z h_2)_u+(\partial_w h_2)_u\big)^2\big](x, u(x))\,,
\end{split}
\end{equation*}
so that we end up with
\begin{equation*}
\begin{split}
\partial_u \mathfrak{J}(u)[\vartheta]
=& -\frac{1}{2}\int_D\sigma_2 (1+\vert\partial_x u(x)\vert^2)\,\big[\partial_z\psi_{u,2}-(\partial_z h_2)_u-(\partial_w h_2)_u\big]^2(x, u(x))\,\vartheta(x)\,\rd x\\
& +\frac{1}{2}\int_D\sigma_2 \,\big[\big((\partial_x h_2)_u\big)^2+ \big((\partial_z h_2)_u+(\partial_w h_2)_u\big)^2\big](x, u(x))\,\vartheta(x)\,\rd x\\
& -\int_D\big[\sigma_1(\partial_w h_1)_u\,\partial_z\psi_{u,1}\big](x,-H-d)\,\vartheta(x)\,\rd x\,,
\end{split}
\end{equation*}
as claimed. It finally follows from Proposition~\ref{ACDC}, Corollary~\ref{Pam}, Proposition~\ref{ACDC1}, the continuity of the trace from $H^{3/4}(\Omega_1)$ to $L_2(D\times \{-H\})$, and the $C^2$-regularity on $h_1$ and $h_2$ that $\partial_u\mathfrak{J}: S\rightarrow \mathcal{L}(H^2(D)\cap H_0^1(D),\R)$ is continuous.
\end{proof}

We finish off this section by considering the differentiability properties of the Dirichlet energy $\mathfrak{J}$ at a function $u\in \bar S$. As pointed out in Theorem~\ref{Thm1b}, allowing also for non-empty coincidence sets restricts to directional derivatives  in the directions  $ -u+S$. Given $u\in\bar S$, let us recall that the function \mbox{$\mathfrak{g}(u):D\rightarrow [0,\infty)$} is given by
\begin{subequations}\label{ggg}
\begin{equation}\label{ggg2}
\mathfrak{g}(u)(x)= \frac{\sigma_2}{2} \left( 1+(\partial_x u(x))^2 \right)\,\big[ \partial_z\psi_{u,2}-(\partial_z h_2)_u-(\partial_w h_2)_u \big]^2(x, u(x))
\end{equation}
for $x\in D\setminus \mathcal{C}(u)$ and
\begin{equation}\label{ggg1}
\mathfrak{g}(u)(x) = \frac{\sigma_2}{2}\, \left[ \frac{\sigma_1}{\sigma_2} \partial_z\psi_{u,1}-(\partial_z h_2)_u-(\partial_w h_2)_u \right]^2(x, -H)
\end{equation}
for $x\in \mathcal{C}(u)$.
\end{subequations}

\begin{corollary}\label{C17}
Let $u\in\bar S$ and $w\in S$. Then
\begin{equation*}
\begin{split}
\lim_{t\rightarrow 0^+} \frac{1}{t}\big(\mathfrak{J}(&u+t(w-u))-\mathfrak{J}(u)\big)\\
= &-\int_D \mathfrak{g}(u) (w-u)\, \rd x\\
& +\frac{1}{2}\int_D\sigma_2 \,\left[ \big((\partial_x h_2)_u\big)^2+ \big((\partial_z h_2)_u+(\partial_w h_2)_u\big)^2 \right](\cdot, u)\,(w-u)\,\rd x\\
& -\int_D\big[\sigma_1(\partial_w h_1)_u\,\partial_z\psi_{u,1}\big](\cdot,-H-d)\,(w-u)\,\rd x\,.
\end{split}
\end{equation*}
Moreover, the function $\mathfrak{g}:\bar S\rightarrow L_p(D)$ is continuous for each $p\in [1,\infty)$.
\end{corollary}

\begin{proof}
Given $u\in\bar S$ and $w\in S$, note that
\begin{equation*}
u_s:= u+s(w-u)=(1-s)u+sw \in S\,,\quad s\in (0,1)\,.
\end{equation*}
Let $\psi_{u_s}$ denote the solution to \eqref{a11} associated with $u_s$ and set $\vartheta:=w-u$. Since $u_s\in S$ for $s\in (0,1)$, we obtain from Proposition~\ref{C15} that
\begin{equation*}
\begin{split}
\frac{\rd}{\rd s} \mathfrak{J}(u_s) =& -\frac{1}{2}\int_D\sigma_2 (1+\vert\partial_x u_s\vert^2)\,\big[\partial_z\psi_{u_s,2}-(\partial_z h_2)_{u_s}-(\partial_w h_2)_{u_s}\big]^2(\cdot, u_s)\,\vartheta\,\rd x\\
& +\frac{1}{2}\int_D\sigma_2 \,\big[\big((\partial_x h_2)_{u_s}\big)^2+ \big((\partial_z h_2)_{u_s}+(\partial_w h_2)_{u_s}\big)^2\big](\cdot, u_s)\,\vartheta\,\rd x\\
& -\int_D\big[\sigma_1(\partial_w h_1)_{u_s}\,\partial_z\psi_{u_s,1}\big](\cdot,-H-d)\,\vartheta\,\rd x
\end{split}
\end{equation*}
for $s\in (0,1)$. Therefore, letting $s\rightarrow 0$ we derive with the help of Proposition~\ref{ACDC1}, the $C^2$-regularity of $h_1$ and $h_2$, and \eqref{100} that
\begin{equation}\label{pppl}
\begin{split}
\lim_{s\rightarrow 0^+}\frac{\rd}{\rd s} \mathfrak{J}(u_s)=&-\int_D \mathfrak{g}(u) \vartheta\, \rd x\\
& +\frac{1}{2}\int_D\sigma_2 \,\big[\big((\partial_x h_2)_u\big)^2+ \big((\partial_z h_2)_u+(\partial_w h_2)_u\big)^2\big](\cdot, u)\,\vartheta\,\rd x\\
& -\int_D\big[\sigma_1(\partial_w h_1)_u\,\partial_z\psi_{u,1}\big](\cdot,-H-d)\,\vartheta\,\rd x\,.
\end{split}
\end{equation}
Now, Corollary~\ref{C3} guarantees that $\mathfrak{J}(u_s)\rightarrow\mathfrak{J}(u)$ as $s\rightarrow 0$,  so that
\begin{equation*}
\begin{split}
\mathfrak{J}(u_t)-\mathfrak{J}(u)= \int_0^t \frac{\rd}{\rd s} \mathfrak{J}(u_s)\,\rd s\,,\quad t\in (0,1)\,,
\end{split}
\end{equation*}
and we conclude from \eqref{pppl} that
\begin{equation*}
\begin{split}
\lim_{t\rightarrow 0^+} \frac{1}{t}\big(\mathfrak{J}(u_t)-\mathfrak{J}(u)\big)&= \lim_{t\rightarrow 0^+} \frac{1}{t}\int_0^t \frac{\rd}{\rd s} \mathfrak{J}(u_s)\,\rd s\\
&=-\int_D \mathfrak{g}(u) \vartheta\, \rd x\\
&\quad +\frac{\sigma_2}{2}\int_D \big[\big((\partial_x h_2)_u\big)^2+ \big((\partial_z h_2)_u+(\partial_w h_2)_u\big)^2\big](\cdot, u)\,\vartheta\,\rd x\\
&\quad -\int_D\big[\sigma_1(\partial_w h_1)_u\,\partial_z\psi_{u,1}\big](\cdot,-H-d)\,\vartheta\,\rd x\,.
\end{split}
 \end{equation*}
Recalling that $\vartheta=w-u$, the proof of Corollary~\ref{C17} is complete, except for the continuity of the function $\mathfrak{g}:\bar S\rightarrow L_p(D)$ for $p\in [1,\infty)$. However, this follows from  Corollary~\ref{Pam}, Proposition~\ref{ACDC1}, the continuity of the trace from $H^{s}(\Omega_1)$ to $L_p(D\times \{-H\})$ for $s\in (1-1/p,1)$, and the $C^2$-regularity of $h_1$ and $h_2$.
\end{proof}

\section{Least Energy Solution for a Stationary MEMS Model}\label{LESSMM}

We illustrate our findings on the shape derivative of the Dirichlet energy \eqref{DE}  with the existence of solutions to an elliptic variational inequality arising in the modeling of micromechanical systems (MEMS) \cite{Pel01a, PeB03}. Specifically, we consider an idealized MEMS device consisting of two plates held at different electrostatic potentials: a thin elastic plate is clamped at its boundary and suspended above a rigid ground plate, the latter being covered by a non-penetrable dielectric layer of thickness $d>0$ \cite{BG01}. Due to the potential difference between the two plates, a Coulomb force is created across the device, inducing a deformation of the elastic plate, thereby converting electrostatic energy to mechanical energy while changing the  geometry of the device. Considering a cross section of the device, the rigid plate and the dielectric layer are given by $D\times \{-H-d\}$ with $D=(-L,L)$ and
$$
 \Omega_1=D\times (-H-d,-H)\,,
$$
respectively. Denoting the deflection of the elastic plate by $u:\bar D\rightarrow [-H,\infty)$, the elastic plate is the graph 
\begin{equation*}
\mathfrak{G}(u) = \{ (x,u(x))\,:\, x\in D\}
\end{equation*}
of $u$, the latter satisfying the clamped boundary conditions 
\begin{equation}\label{clamped}
u(\pm L)=\partial_x u(\pm L)=0\,.
\end{equation} 
The space between the dielectric layer and the elastic plate is
$$
 \Omega_2(u)=\left\{(x,z)\in D\times \mathbb{R}\,:\, -H<  z <  u(x)\right\}\,,
$$
and $\Omega_1$ and $\Omega_2(u)$ are separated by the interface
$$
\Sigma(u)=\{(x,-H)\,:\, x\in D,\, u(x)>-H\}\,. 
$$
We finally set
$$
 \Omega({u})=\left\{(x,z)\in D\times \mathbb{R} \,:\, -H-d<  z <  u(x)\right\}=\Omega_1\cup  \Omega_2( {u})\cup \Sigma(u)\,,
$$
so that the geometry of the MEMS device is exactly that considered in the previous sections. The dielectric properties of the device are given by the permittivity of the dielectric layer $\Omega_1$, which is assumed to be a positive function $\sigma_1\in C^2({\overline \Omega_1})$, and the permittivity of $\Omega_2(u)$, which is taken to be a positive constant $\sigma_2$. Moreover, the two plates are held at constant potentials, being respectively taken to be zero on the rigid plate $D\times \{-H-d\}$ and equal to a positive constant $V$ on the elastic plate $\mathfrak{G}(u)$. The electrostatic potential $\psi_u$ in the device then solves the transmission problem \eqref{TMP}; that is,
\begin{align*}
\mathrm{div}(\sigma\nabla\psi_u)&=0 \quad\text{in }\ \Omega(u)\,, \\
\llbracket \psi_u \rrbracket = \llbracket \sigma \partial_z \psi_u \rrbracket &=0\quad\text{on }\ \Sigma(u)\,, \\
\psi_u&=h_u\quad\text{on }\ \partial\Omega(u)\,, 
\end{align*}
the corresponding boundary conditions being prescribed by a function $h$ satisfying \eqref{bobbybrown}-\eqref{bb}, as well as 
\begin{subequations}\label{bbbrown}
\begin{equation}\label{bobbybrown1}
h_1(x,-H-d,w)=h_2(x,w,w)- V=0\,,\quad (x,w)\in \bar{D}\times [-H,\infty)\,.
\end{equation}
Finally, the total energy 
$$
E(u):= E_m(u)+E_e(u)
$$ 
of the MEMS device is the sum of the mechanical energy $E_m(u)$ and the electrostatic energy $E_e(u)$. The former is given by
$$
E_m(u):=\frac{\beta}{2}\|\partial_x^2u\|_{L_2(D)}^2 +\left(\frac{\tau}{2}+\frac{a}{4}\|\partial_x u\|_{L_2(D)}^2\right)\|\partial_x u\|_{L_2(D)}^2
$$
with $\beta>0$, $\tau\ge 0$, and $a\ge 0$, taking into account bending and external stretching effects of the elastic plate. The electrostatic energy is
$$
E_e(u):=-\frac{1}{2}\int_{\Omega(u)} \sigma \vert\nabla \psi_u\vert^2\,\rd (x,z)\,;
$$
that is, $E_e(u):= -\mathfrak{J}(u)$, see \eqref{DE}. Introducing
\begin{equation*}
\bar{S_0} := \left\{ u\in H^2(D)\,:\, u(\pm L) = \partial_x u(\pm L)=0\,, \quad u\ge -H \text{ in } D \right\} \subset \bar{S}\,,
\end{equation*}
which takes into account that the elastic plate is clamped at its boundary, it is readily seen that $E_m(u)$ is well-defined for $u\in\bar{S_0}$, as are $\psi_u$ and $E_e(u)$ due to Theorem~\ref{Thm1}.

Equilibrium configurations of the MEMS device, if any, are then provided by critical points of the total energy $E$ in $\bar{S_0}$, and in particular by minimizers when they exist. A minimal requirement in that direction is the boundedness from below of $E$ on $\bar{S_0}$, for which the following additional assumptions on $h$ are sufficient: there are constants $m_i>0$, $i=1,2,3$, such that 
\begin{equation}
\vert \partial_x h_1(x,z,w)\vert +\vert\partial_z h_1(x,z,w)\vert  \le \sqrt{m_1+m_2 w^2}\,, \quad \vert\partial_w h_1(x,z,w)\vert \le \sqrt{m_3}\,, \label{bobbybrown5}
\end{equation}
for $(x,z,w)\in \bar D \times [-H-d,-H] \times [-H,\infty)$ and
\begin{equation} 
\vert \partial_x h_2(x,z,w)\vert +\vert\partial_z h_2(x,z,w)\vert \le \sqrt{\frac{m_1+m_2 w^2}{H+w}}\,,\quad \vert\partial_w h_2(x,z,w)\vert \le \sqrt{\frac{m_3}{H+w}}\,,\label{bobbybrown6}
\end{equation}
\end{subequations}
for $(x,z,w)\in \bar D \times [-H,\infty) \times [-H,\infty)$.

Within this framework, Theorem~\ref{Thm1} allows us to prove the existence of at least one minimal energy solution.

\begin{theorem}\label{Thm2}
Assume that $h$ satisfies \eqref{bobbybrown}-\eqref{bb} and \eqref{bbbrown} and set
$$
\mathfrak{K} :=\beta - 4L^2 \left[ (d+1) \sigma_{max}\left( 12 m_2L^2+2 m_3 \right)- \tau \right]_+ \,.
$$
If
\begin{equation}\label{bbb}
\max\{a,\mathfrak{K}\}>0\,,
\end{equation}
then the total energy $E$ has at least one minimizer $u_*$ in $\bar{S_0}$; that is, 
\begin{equation}\label{B}
E(u_*)=\min_{\bar{S_0}}E\,.
\end{equation}
\end{theorem}

It is yet unknown whether there is more than one equilibrium configuration or whether the minimizer provided by Theorem~\ref{Thm2} has empty or non-empty coincidence set $\mathcal{C}(u_*)$ (defined in \eqref{CS}). Even in much simpler situations as considered in \cite{LW17}, where the electrostatic potential is an explicitly computable function depending in a local way on $u$, the answer is rather complex. Indeed, minimizers may have empty or non-empty coincidence sets and may coexist with other critical points of $E$, depending on the boundary values of the electrostatic potential. We expect the same complexity in the model considered herein. 

\begin{remark}
Condition~\eqref{bbb} is obviously satisfied if $\mathfrak{K}>0$, which amounts to assuming that the applied voltage $V$ is sufficiently small compared to the dimensions of the device, see Example~\ref{Ex1} below. 
\end{remark}

Next, thanks to the analysis carried out in the previous sections, we can characterize any solution to \eqref{B} by means of a variational inequality. To this end, for $u\in \bar S$, we define the function $g(u)$ by
\begin{equation}\label{gg}
g(u)(x):=\left\{\begin{array}{ll}  \displaystyle{\frac{\sigma_2}{2} \big(1+(\partial_x u(x))^2\big)\big(\partial_z\psi_{u,2}(x,u(x))\big)^2} \,, & x\in D\setminus \mathcal{C}(u)\,,\\
\hphantom{x}\vspace{-3.5mm}\\
\displaystyle{\frac{\sigma_1(x,-H)^2}{2\sigma_2}\big(\partial_z\psi_{u,1}(x,-H)\big)^2}\,,  & x\in  \mathcal{C}(u)\,.
\end{array}\right.
\end{equation}
Actually, $g$ is nothing but the function $\mathfrak{g}$ defined in Theorems~\ref{Thm1bb} and~\ref{Thm1b}, taking into account the property 
\begin{equation}\label{f}
\partial_w h_1(x,-H-d,w) = \partial_x h_2(x,w,w) = \partial_z h_2(x,w,w) + \partial_w h_2(x,w,w)=0
\end{equation}
for $(x,w)\in D\times [-H,\infty)$, which is easily derived from \eqref{bobbybrown1}. In particular, $g:\bar S\rightarrow L_2(D)$ is continuous and represents the electrostatic force acting on the elastic plate $\mathfrak{G}(u)$.

\begin{theorem}\label{Thm3}
Assume that $h$ satisfies \eqref{bobbybrown}-\eqref{bb} and \eqref{bobbybrown1} and that there is a solution $u\in \bar{S_0}$ to the minimization problem \eqref{B}. Then $g(u)\in L_2(D)$ and $u$ is an $H^2$-weak solution to the variational inequality
\begin{equation}
\beta\partial_x^4u-(\tau+a\|\partial_x u\|_{L_2(D)}^2)\partial_x^2 u+\partial\mathbb{I}_{\bar{S_0}}(u) \ni -g(u) \;\;\text{ in }\;\; D\,, \label{bennygoodman}
\end{equation}
where $\partial\mathbb{I}_{\bar{S_0}}$ is the subdifferential of the indicator function $\mathbb{I}_{\bar S_0}$ of the closed convex subset $\bar{S_0}$ of $H^2(D)$; that is, 
$$
\int_D \left\{\beta\partial_x^2 u\,\partial_x^2 (w-u)+\big[\tau+a\|\partial_x u\|_{L_2(D)}^2\big]\partial_x u\, \partial_x(w-u)\right\}\,\rd x\ge -\int_D g(u) (w-u)\, \rd x  
$$
for all $w\in \bar{S_0}$. 
\end{theorem}

A minimizer $u$ of $E$ in $\bar{S_0}$ being a critical point of $E$ and satisfying the convex constraint $u\in \bar{S_0}$, the variational inequality \eqref{bennygoodman} is simply the corresponding Euler-Lagrange equation: it involves the derivative $\beta\partial_x^4u-( \tau+a\|\partial_x u\|_{L_2(D)}^2)\partial_x^2 u$ of the mechanical energy $E_m$ with respect to $u$, the subdifferential $\partial\mathbb{I}_{\bar{S_0}}(u)$ of the convex constraint, and the ``differential'' $g(u)$ of the electrostatic energy $E_e$  with respect to $u$, in the sense of Theorems~\ref{Thm1bb} and~\ref{Thm1b}.

\begin{remark}
Theorems~\ref{Thm2} and~\ref{Thm3} are also valid with $\bar{S}$ instead of $\bar{S_0}$, the only difference being that the minimizers of $E$ in $\bar{S}$ in Theorem~\ref{Thm3} now satisfy \eqref{bennygoodman} subject to the Navier or pinned boundary conditions $u(\pm L) = \partial_x^2 u(\pm L)=0$ instead of the clamped boundary conditions \eqref{clamped}.
\end{remark}

Before providing the proofs of Theorem~\ref{Thm2} and Theorem~\ref{Thm3}, let us give an example of a function $h$ describing the boundary conditions \eqref{TMP3} for the electrostatic potential.

\begin{example}\label{Ex1}
Let us consider the situation where $\sigma_1$ does not depend on the vertical variable $z$; that is, $\sigma_1=\sigma_1(x)$. In that case, we set
$$
h_1(x,z,w):=V\frac{\sigma_2 (H+z+d)}{\sigma_2 d+\sigma_1(x)(H+w)}\,,\quad (x,z,w)\in \bar{D}\times [-H-d,-H]\times [-H,\infty)\,,
$$
and
$$
h_2(x,z,w):=V\frac{\sigma_2 d+\sigma_1(x)(H+z)}{\sigma_2 d+\sigma_1(x)(H+w)}\,,\quad (x,z,w)\in \bar{D}\times [-H,\infty)\times [-H,\infty)\,.
$$
Then assumptions \eqref{bobbybrown}-\eqref{bb} and \eqref{bbbrown} are easily checked. Moreover, if $V$ is sufficiently small, then $\mathfrak{K}$ defined in Theorem~\ref{Thm2} is positive, hence \eqref{bbb} holds in that case.
\end{example}

\subsection{Existence of a Minimizer}

Given $u\in \bar S$ we recall that $\psi_u$ is the unique solution to the transmission problem \eqref{TMP} provided by Theorem~\ref{Thm1}. 

\begin{proof}[Proof of Theorem~\ref{Thm2}] We first note that the total energy $E$ is bounded from below and coercive. To this end, we recall the Poincar\'e and Poincar\'e-Wirtinger inequalities
\begin{equation}
\|u\|_{L_2(D)}\le \vert D\vert \|\partial_x u\|_{L_2(D)}\,, \qquad \|\partial_x u\|_{L_2(D)}\le \vert D\vert \|\partial_x^2 u\|_{L_2(D)}\,, \label{pw}
\end{equation}
which are valid for all $u\in \bar S$. Let $u\in \bar S$. It follows from \eqref{bobbybrown5}, \eqref{bobbybrown6}, Lemma~\ref{C1}, and Young's inequality that
\begin{equation*}
\begin{split}
-E_e(u)&=\frac{1}{2}\int_{\Omega(u)}\sigma\vert\nabla\psi_u\vert^2\,\rd(x,z)\le \frac{1}{2} \int_{\Omega(u)}\sigma\vert\nabla h_u\vert^2\,\rd(x,z)\\
& \le \int_{\Omega(u)} \sigma\left[\left(\partial_x h(x,z,u(x))\right)^2+\left(\partial_w h(x,z,u(x))\right)^2 (\partial_x u(x))^2\right]\,\rd (x,z)\\
&\qquad + \frac{1}{2} \int_{\Omega(u)} \sigma \left(\partial_z h(x,z,u(x))\right)^2\,\rd (x,z)\\
& \le (d+1) \sigma_{max} \int_D \left[ \frac{3}{2} (m_1+m_2u(x)^2) + m_3(\partial_x u(x))^2\right]\,\rd x\,.
\end{split}
\end{equation*}
Using \eqref{pw} we get
$$
-E_e(u)\le \frac{d+1}{2} \sigma_{max} \left[3 m_1\vert D\vert +\left(3 m_2 \vert D\vert^2+2 m_3\right)\|\partial_x u\|_{L_2(D)}^2\right]\,.
$$
Therefore,
\begin{align}
E(u) & \ge \frac{\beta}{2}\|\partial_x^2u\|_{L_2(D)}^2 +\frac{a}{4}\|\partial_x u\|_{L_2(D)}^4 - \frac{3(d+1)}{2} \sigma_{max} m_1\vert D\vert \nonumber \\
& \qquad - \left[ \frac{d+1}{2} \sigma_{max} \left(3 m_2 \vert D\vert^2+2 m_3\right)-\frac{\tau}{2}\right]\|\partial_x u\|_{L_2(D)}^2 \nonumber\\
& \ge \frac{\beta}{2}\|\partial_x^2u\|_{L_2(D)}^2 +\frac{a}{4}\|\partial_x u\|_{L_2(D)}^4 - \frac{3(d+1)}{2} \sigma_{max} m_1\vert D\vert \nonumber \\
& \qquad - \left[ \frac{d+1}{2} \sigma_{max} \left(3 m_2 \vert D\vert^2+2 m_3\right)-\frac{\tau}{2}\right] _+\|\partial_x u\|_{L_2(D)}^2 \,. \label{E1}
\end{align}
Now, if $a>0$, then Young's inequality and \eqref{E1} give
$$
E(u)\ge \frac{\beta}{2}\|\partial_x^2u\|_{L_2(D)}^2 - C_1
$$
for some constant $C_1>0$ independent of $u\in\bar S$. If $a=0$, then we infer from \eqref{bbb} with $\vert D\vert =2L$, \eqref{pw}, and \eqref{E1} that $\mathfrak{K}>0$ and 
$$
E(u)\ge \frac{\mathfrak{K}}{2} \|\partial_x^2u\|_{L_2(D)}^2 - \frac{3(d+1)}{2} \sigma_{max} m_1\vert D\vert \,.
$$
Consequently, $E$ is coercive when \eqref{bbb} is satisfied. 

Now, take a minimizing sequence $(u_j)_{j\ge 1}$ of $E$ in $\bar{S_0}\subset \bar{S}$. Then
$$
\lim_{j\rightarrow \infty} E(u_j)=\inf_{\bar{S_0}} E\,,
$$ 
and the just established coercivity of $E$ guarantees that $(u_j)_{j\ge 1}$ is bounded in $H^2(D)$. We thus may assume that $(u_j)_{j\ge 1}$ converges weakly towards some $u_*$ in $H^2(D)$ and strongly in  $H^s(D)$ for $s\in [1,2)$. Obviously $u_*\in \bar{S_0}$ and
$$
E_m(u_*)\le \liminf_{j\rightarrow\infty} E_m(u_j)\,.
$$
Moreover, since $H^2(D)$ is continuously embedded in $L_\infty(D)$, we may invoke Corollary~\ref{C3} to obtain that
$$
E_e(u_*)= \lim_{j\rightarrow\infty} E_e(u_j)\,.
$$
Consequently, $u_*$ minimizes $E$ on $\bar{S_0}$ and the proof of Theorem~\ref{Thm2} is complete. 
\end{proof}

\subsection{Euler-Lagrange Equation}

 We finally prove Theorem~\ref{Thm3} which requires deriving the Euler-Lagrange equation satisfied by any minimizer of $E$ on $\bar{S_0}$. We first observe that the additional assumption \eqref{bobbybrown1} simplifies the directional derivative with respect to $u\in \bar S$ of the electrostatic energy $E_e$, which is given in Theorems~\ref{Thm1bb} and~\ref{Thm1b}.
  
\begin{proposition}\label{C17b}
Let $u\in\bar S$ and $w\in S$. Then
\begin{equation*}
\lim_{s\rightarrow 0^+} \frac{1}{s}\big(E_e(u+s(w-u))-E_e(u)\big)= \int_D g(u) (w-u)\, \rd x\,.
\end{equation*}
\end{proposition}

\begin{proof}
As already mentioned, we infer from \eqref{f} that $g(u)=\mathfrak{g}(u)\in L_2(D)$. Therefore, since $E_e(u)= -\mathfrak{J}(u)$ by \eqref{DE}, we deduce from Theorems~\ref{Thm1bb} and~\ref{Thm1b} that
\begin{equation}\label{lopp}
\begin{split}
& \lim_{s\rightarrow 0^+} \frac{1}{s}\big(E_e(u+s(w-u))-E_e(u)\big)\\
& \qquad = \int_D g(u)(x) (w-u)(x)\, \rd x\\
& \qquad\quad -\frac{1}{2}\int_D\sigma_2 \,\big[\big((\partial_x h_2)_u\big)^2+ \big((\partial_z h_2)_u+(\partial_w h_2)_u\big)^2\big](x, u(x))\,(w-u)(x)\,\rd x\\
& \qquad\quad +\int_D\big[\sigma_1(\partial_w h_1)_u\,\partial_z\psi_{u,1}\big](x,-H-d)\,(w-u)(x)\,\rd x\,.
\end{split}
\end{equation}
Now observe that the first identity of \eqref{bobbybrown1} implies 
$$
(\partial_w h_1)_u(x,-H-d)=\partial_w h_1(x,-H-d,u(x)) =0\,,\quad x\in D\,,
$$
so that the last integral on the right-hand side of \eqref{lopp} vanishes.
Moreover, the second identity of \eqref{bobbybrown1} yields 
$$
(\partial_xh_2)_u(x,u(x)) =0\,,\quad x\in D\,,
$$
which, together with \eqref{f}, implies that the second  integral on the right-hand side of \eqref{lopp} also vanishes. 
\end{proof}

\begin{proof}[Proof of Theorem~\ref{Thm3}]
Consider a minimizer $u\in\bar{S_0}$ of $E$ on $\bar{S_0}$ and fix 
\begin{equation*}
w \in S_0 :=  \left\{ v\in H^2(D)\,:\, v(\pm L) = \partial_x v(\pm L)=0\,, \quad v> -H \text{ in } D \right\} \subset S\,,
\end{equation*}
Owing to the convexity of $\bar{S_0}$, the function $u+s(w-u)=(1-s)u+sw$ belongs to $S_0$  for all $s\in (0,1]$ and the minimizing property of $u$ guarantees that
$$
0\le \liminf_{s\rightarrow 0^+} \frac{1}{ s}\big(E(u+s(w-u))-E(u)\big)\,.
$$
Proposition~\ref{C17b} then implies that
\begin{align*}
0 & \le  \int_D \left\{\beta\partial_x^2 u\,\partial_x^2 (w-u) + \left( \tau+a\|\partial_x u\|_{L_2(D)}^2 \right) \partial_x u\, \partial_x (w-u)\right\}\,\rd x \\
& \qquad + \int_D g(u) (w-u)\, \rd x
\end{align*}
for all $w\in S_0$. Since $S_0$ is dense in $\bar{S_0}$,  this inequality also holds for any $w\in \bar{S_0}$, which completes the proof of Theorem~\ref{Thm3}. 
\end{proof}

\begin{remark}\label{remEGL} In Theorem~\ref{Thm3}, a salient feature of $g(u)$, which is given by \eqref{gg} and coincides with the directional derivative of $E_e=-\mathfrak{J}$, is that it is non-negative, a property which is due to the uniform potentials applied on both the rigid plate $D\times\{-H-d\}$ and the elastic plate $\mathfrak{G}(u)$. When the applied potential on the elastic plate $\mathfrak{G}(u)$ is non-constant, the formula for the directional derivative of $E_e=-\mathfrak{J}$ provided by Theorems~\ref{Thm1bb} and~\ref{Thm1b} involves a positive term and a negative term, and its sign is not determined  \textit{ a priori}. A similar observation is made in \cite{EGL} for a  related model. In fact, if $d=0$ (that is, there is no dielectric layer) and if, instead of assuming \eqref{bobbybrown1}, the function $h$ is taken to be
$$
h(x,z,w)=\frac{H+z}{H+w} \, p(x,w)\,,\quad (x,z,w)\in \bar D\times [-H,\infty)\times [-H,\infty)\,,
$$
for a suitable function $p$, then one easily recovers the model considered in \cite{EGL} from Theorem~\ref{Thm1bb}. 
\end{remark}

\bibliographystyle{siam}
\bibliography{BG_Transmission_Model}

\end{document}